\documentclass[a4paper, 10pt]{article}
\usepackage[latin1]{inputenc}
\usepackage[english]{babel}

\usepackage{dsfont}
\usepackage{graphicx,color}
\usepackage{float}
\usepackage{placeins}

\usepackage{amsmath}
\usepackage{amsthm}
\usepackage{amssymb}
\usepackage{color}
\theoremstyle{plain}
\newtheorem{theorem}{Theorem}
\newtheorem{proposition}[theorem]{Proposition}
\newtheorem{lemma}[theorem]{Lemma}
\newtheorem{corollary}[theorem]{Corollary}
\newtheorem{remark}[theorem]{Remark}
\newtheorem{definition}[theorem]{Definition}
\newtheorem{assumption}[theorem]{Assumption}

\newcommand{\R}{\mathbb{R}}
\newcommand{\dd}{\, \mathrm{d}}
\newcommand{\intSmall}{\textstyle{\int} \displaystyle}

\usepackage{hyperref}
\usepackage{vmargin}
\setmarginsrb{3cm}{2.5cm}{3cm}{1.5cm}{0cm}{0cm}{0cm}{1.5cm}

\begin{document}

\title{Hybrid Optimal Control Problems for a Class of Semilinear Parabolic Equations}
\date{\today}
\author{S\'ebastien Court\thanks{Institute for Mathematics and Scientific Computing, Karl-Franzens-Universit\"{a}t, Heinrichstrasse 36, 8010 Graz, Austria, email: {\tt sebastien.court@uni-graz.at}, {\tt laurent.pfeiffer@uni-graz.at}.} \and Karl Kunisch\thanks{Institute for Mathematics and Scientific Computing, Karl-Franzens-Universit\"{a}t, Heinrichstrasse 36, 8010 Graz, Austria, and Radon Institute, Austrian Academy of Sciences, email: {\tt karl.kunisch@uni-graz.at}.} \and Laurent Pfeiffer\footnotemark[1]}

\maketitle

\begin{abstract}
A class of optimal control problems of hybrid nature governed by semilinear parabolic equations is considered. These problems involve the optimization of switching times  at which the dynamics, the integral cost, and the bounds on the control may change. First- and second-order optimality conditions  are derived. The analysis is based on a reformulation involving a judiciously chosen transformation of the time domains. For autonomous systems and time-independent integral cost, we prove that the Hamiltonian is constant in time when evaluated along the optimal controls and trajectories. A numerical example is provided.
\end{abstract}
\hfill \\

\hfill 

\tableofcontents

\section{Introduction}

This article is dedicated to the derivation of optimality conditions for the following class of optimal control problems of hybrid semilinear parabolic equations:
\begin{eqnarray*}
& & \min_{\begin{subarray}{c} 0= \tau_0 < \tau_1 < \hdots < \tau_K= T \\ u \in L^\infty((0,T) \times \Omega) \\ y \in \mathcal{Y}(0,T) \end{subarray}} \
\sum_{k=1}^{K} \left[ \int_{\tau_{k-1}}^{\tau_k} \int_{\Omega} \ell_k(t,x,y(t,x),u(t,x)) \dd x \dd t
+ \int_{\Omega} \phi_k(\tau_k,x,y(\tau_k,x)) \dd x \right] \\
& & \text{subject to: } \\
& & \ \quad u_k^-(x) \leq u(t,x) \leq u_k^+(x), \text{ for a.\,e.\@ $(t,x) \in (\tau_{k-1},\tau_k) \times \Omega$, $\forall k=1,...,K$}, \phantom{q_{q_{q_q}}} \\
& & \left\{\begin{array}{rcll}
\dot{y}(t,x)-\Delta y(t,x) - f_k(t,x,y(t,x),u(t,x)) & = & 0 & \text{in $(\tau_{k-1},\tau_k) \times \Omega$, $\forall k=1,...,K$}, \\
y(t,x)& = & 0 & \text{on $(0,T) \times \delta \Omega$}, \\
y(0,x)& = & y_0(x) & \text{in $\Omega$},
\end{array} \right.
\end{eqnarray*}
Further specifications on the problem data as well as the definition of the space $\mathcal{Y}(0,T)$ will be given below. Continuity in time and space of the state variable $y$, in particular at the times $\tau_1$,...,$\tau_{K-1}$, will be guaranteed.
Here the term {\it hybrid} refers to the following  three features of the problem:
\begin{itemize}
\item At time $\tau_k$, the nonlinear part of the system may change, passing from $f_{k}$ to $f_{k+1}$.
\item At time $\tau_k$, the constraints on the control may change, passing from $u_{k}^-$ and $u_{k}^+$ to $u_{k+1}^-$ and $u_{k+1}^+$ respectively.
\item The cost function incorporates terms depending on the value of the state at time $\tau_k$.
\end{itemize}
The optimization variables $\tau_1$,...,$\tau_{K-1}$ are called switching times. They are not fixed a priori, but part of the optimization problem. A more general definition of hybrid control problems is provided in~\cite{GP05}. In that reference, as well as in other publications in the field, the dynamics $f_k$, the integral cost $\ell_k$ and the control constraints $u_k^-$ and $u_k^+$ used on the time interval $(\tau_{k-1},\tau_k)$ can be chosen in a finite set. This combinatorial aspect is not considered in the present paper.

In this article, for the sake of clarity, the analysis is done for $K=2$: There is only one switching time to be optimized.
The provided results can, however, be extended to the case $K>2$ without essential  difficulties. They can actually be extended to the following situations:
\begin{itemize}
\item The state variable and the control variable are multi-dimensional.
\item The functions $f_k$, $\ell_k$, $\phi_k$ all depend on all the switching times.
\item Constraints on the switching times are considered.
\end{itemize}
The boundary conditions can also be changed to mixed boundary conditions, and a Neumann boundary control can also be considered.
However, our analysis cannot be extended in a straightforward manner to the situation where the functions $u_1^-$,...,$u_K^-$ and $u_1^+$,...,$u_K^+$ depend on $t$ and/or $\tau_1$,...,$\tau_{K-1}$.

Semilinear parabolic equations have been used in various fields of applications. There is also a rich literature on the optimal control of such equations. In chemical kinetics, they can model the evolution of the concentration of the substances involved in a chemical reaction (see for example~\cite{BJT10,MBH16}). In ecology, they can describe  the evolution of interacting species (see for example~\cite{AAC11}). In physics, they can model the phase change (e.g.\@ from solid to liquid) of materials (see~\cite{Hei97}). In neurobiology, the Hodgkin-Huxley and FitzHugh-Nagumo equations describe the transmission of electrical signals in axons (see~\cite{KPR16}).

While hybrid problems of finite dimensional systems have been extensively studied in the last years, hybrid problems of systems described by partial differential equations have received little attention so far. Let us motivate each of the three features of the hybrid problem under consideration.
\begin{itemize}
\item Many of the semilinear parabolic models used in the literature involve physical coefficients which are often assumed to be constant. When these physical coefficients depend on the experimentation conditions (such as temperature or pressure), then a sudden change of conditions will cause a change of the coefficients, and therefore a change of dynamics. In chemical kinetics for example, the reaction rate is related to the temperature, as stated by the Arrhenius' law. In population dynamics, the reproduction and death rates of some species can be quickly modified by pesticide application, for example.
\item In some situations, the control of the system is realized through different actuators, which cannot be all activated at the same time (see for example~\cite{CRK16,CRKB16}). If the order of utilization of the different actuators is fixed, then a change of the control constraints at the switching time can be used to describe the change of actuators.  We note here that our assumptions allow the equality $u^-_k =u^+_k$ on subsets of $\Omega$.

In some other situations, the actuator should not be utilized all the time. The time interval $[0,T]$ can then be decomposed into active and inactive phases. Our framework enables to optimize the switching times between active and inactive phases. It is, for example, possible to penalize the duration of active phases. In this respect, our approach can be seen as complementary to the one consisting in adding a sparsity-promoting (with respect to time) penalization term to the cost functional. See~\cite{CK16}, for example. Indeed, while the use of such a penalization term ensures the sparsity of the optimal control in a somehow indirect way, it enables us to find the global structure  of the  solution (such as the ordering of the actuators or the number of active and inactive phases). On the contrary, in our approach, the structure of the solution should be known, and sparsity can be imposed in the model.

Let us also mention that sometimes, it is only possible to implement a piecewise-constant control. This can also be modeled by a change of the control constraints at a switching time. Optimal control problems with piecewise-constant control are studied in e.g.~\cite{BT16}, where they are called {\it optimal sampled-data control problems}.

\item  Finally, the introduction of a switching time in the cost function has also its own interest. One can be interested, for example, in taking into account the minimum taken by a function of the state variable. This can be done by introducing a switching time $\tau$ at which the minimum of the function of the state variable is reached. Transforming the problem into a maximization problem, we obtain an $L^\infty$-term in the cost function, as is explained in~\cite{CKP16}.
\end{itemize}

As far as we know, this article is the first one dealing with the derivation of optimality conditions for control problems of semilinear parabolic equations with controls and  switching times simultaneously. In~\cite{RH16}, first-order optimality conditions for switching times are provided for a hybrid problem of semilinear parabolic equations. The only controls considered in~\cite{RH16} are the switching times themselves: In that framework, the state variable is differentiable with respect to the switching times. This property is not satisfied anymore when a  control is incorporated.
Moreover, as far as we know, only the recent reference~\cite{LGW15} deals with second-order optimality conditions for finite-dimensional hybrid systems.

The theory of first- and second-order optimality conditions for optimal control problems of semilinear parabolic equations (without switching times) is now well-established. It is described in the textbooks~\cite{HPUU09,Tro10} and was developed in the last two decades. We invite the reader to refer to~\cite{CdlRT08,CT02,RT00,Sil16,BS15} and the references therein on this topic.
The main difficulty in the derivation of sufficient optimality conditions, called {\it two-norm discrepancy}, lies in the fact that the cost function is twice differentiable for the $L^\infty$-norm, but one can only assume that its Hessian is coercive for the $L^2$-norm on an appropriate subspace.

The derivation of optimality conditions with respect to the switching time for our hybrid optimal control problem constitutes a significant challenge. Indeed, the cost function is  not differentiable with respect to the switching times, since the controls are not continuous, in general. For the same reason, the control-to-state mapping is  not differentiable with respect to the switching times.
First-order optimality conditions with respect to the switching times -- called transversality conditions -- can be derived by designing specific needle perturbations, as in~\cite{GP05}. Here, we follow another approach  to obtain first- and second-order optimality conditions. It consists in introducing a change of variables (in time) in the evolution equation. The idea goes back to -- at least --~\cite{RZ99,RZ00}, and has been also used in~\cite{CKP16,IK10,LGW15}, for example. The change of variables that we use has a crucial feature: After reformulation, the discontinuity occurring in the dynamics, the control constraints, and the integral cost occurs at a fixed time. The reformulated problem is then a parametric optimal control problem. The reformulation is not only useful for the desired sensitivity analysis but also for solving the problem numerically.

In the context of optimal control of ordinary differential equations, it is well-known that the Hamiltonian is constant in time when evaluated along optimal controls and trajectories, provided that  the dynamics and the integral cost are autonomous. This property is also true for the optimal control problems of semilinear parabolic equations, without switching times (see~\cite{Fat94}). We extend it to the class of hybrid optimal control problems considered in this paper. Our approach is inspired by~\cite{RZ99}.

The article is organized as follows. In Section~\ref{sectionSetting}, we formulate the required assumptions on the data of the problem and we reformulate it, with a change of variables. In Section~\ref{sectionDerivatives}  the derivatives of the cost functional are calculated. Section~\ref{sectionRegularity} contains a technical result, used for the second-order analysis. First-order optimality conditions (Theorem~\ref{theoremCN1}) and necessary and sufficient second-order optimality conditions (Theorems~\ref{theoremCN2} and~\ref{theoremCS2}) are given in Section~\ref{sectionOptiCond}. Constancy of the Hamiltonian (Theorem~\ref{theoremConstancy}) is provided in Section~\ref{sectionConstancy}. A numerical illustration is presented in section~\ref{best-sec}. In the appendix, we recall some properties of Nemytskii operators.

\section{Setting} \label{sectionSetting}

\subsection{Formulation of the problem}

Let $T>0$, and $\Omega \subset \R^n$ be a bounded domain with smooth boundary. Denote $Q_0= (0,T) \times \Omega$, $\Sigma_0= (0,T) \times \partial \Omega$,
$Q= (0,2) \times \Omega$ and  $\Sigma= (0,2) \times \partial \Omega$.
Given $t_1<t_2$, we define the function spaces
\begin{eqnarray*}
W(t_1,t_2)&= & L^2 \big( (t_1,t_2),H_0^1(\Omega) \big) \cap W^{1,2} \big( (t_1,t_2),H^{-1}(\Omega) \big), \\
\mathcal{Y}(t_1,t_2)& = & L^2 \big((t_1,t_2),H^{2}(\Omega) \cap H_0^1(\Omega) \big) \cap H^1 \big( (t_1,t_2),L^2(\Omega) \big) \cap C\big( (t_1,t_2) \times \Omega \big).
\end{eqnarray*}
The space $\mathcal{Y}(t_1,t_2)$ is a Banach space, when equipped with the norm given by:
\begin{eqnarray*}
\| y \|_{\mathcal{Y}(t_1,t_2)}& = &
\| y \|_{L^2((t_1,t_2)\times \Omega)}
+ \| \dot{y} \|_{L^2((t_1,t_2)\times\Omega)} \\
& & + \, \| D_x y \|_{L^2((t_1,t_2) \times \Omega, \R^n)}
+ \| D_{xx}^2 y \|_{L^2((t_1,t_2) \times \Omega, \R^{n \times n}) }
+ \| y \|_{L^\infty((t_1,t_2)\times \Omega)}.
\end{eqnarray*}
We shall write $\mathcal{Y}$ instead of $\mathcal{Y}(0,2)$.
Further we recall the continuous embedding $W(t_1,t_2) \hookrightarrow C([t_1,t_2], L^2(\Omega))$.

The following functions are given: $f_1$, $f_2$, $\ell_1$, $\ell_2: [0,T] \times \Omega \times \R \times \R \rightarrow \R$, $\phi_1:[0,T] \times \Omega \times \R \rightarrow \R$, and $\phi_2: \Omega \times \R \rightarrow \R$.
We assume that these functions satisfy assumption~\ref{hypRegularity} stated below.
In this assumption, the notions of boundedness and Lipschitz continuity of order 2 are used. They are recalled in Definition~\ref{defNemytskii} given in Appendix~\ref{appendix}.

\begin{assumption} \label{hypRegularity}
\hfill
\begin{enumerate}
\item The functions $f_1$, $f_2$, $\ell_1$, $\ell_2$, $\phi_1$, and $\phi_2$ are measurable. For a.\,e.\@ $x \in \Omega$, they are twice continuously differentiable with respect to all the other variables. They satisfy the boundedness condition (with respect to $x$) and the Lipschitz condition of order 2 with respect to the other variables ($(t,y,u)$ for $f_1$, $f_2$, $\ell_1$, and $\ell_2$, $(t,y)$ for $\phi_1$, and $y$ for $\phi_2$).
\item For all $R \geq 0$, there exists a constant $C_R > 0$ such that for a.\,e.\@ $(t,x,y,u) \in Q \times \R^2$,
\begin{eqnarray*}
|u| \leq R & \Rightarrow &
\left( f_1(t,x,y,u)y \leq C_R (1 + y^2) \quad \text{and} \quad
f_2(t,x,y,u)y \leq C_R (1 + y^2) \right).
\end{eqnarray*}
\end{enumerate}
\end{assumption}

We can now formulate the hybrid system under consideration.
Let $\tau \in (0,T)$, $y_0 \in H_0^1(\Omega) \cap C(\bar{\Omega})$, and $v \in L^\infty(Q)$. Consider:
\begin{eqnarray}
\left\{\begin{array}{rcll}
\dot{y}(t,x) - \Delta y(t,x) - f_1(t,x,y(t,x),v(t,x)) & = & 0 & \text{ in $(0,\tau) \times \Omega$}, \\
\dot{y}(t,x) - \Delta y(t,x) - f_2(t,x,y(t,x),v(t,x)) & = & 0 & \text{ in $(\tau,T) \times \Omega$}, \\
y(t,x) & = & 0 & \text{ on $\Sigma_0$}, \\
y(0,x)& = & y_0(x) & \text{ in $\Omega$}.
\end{array}\right.
\end{eqnarray}
Under assumption~\ref{hypRegularity}, it is well-known that this system has a unique weak solution in $W(0,T)$, that we denote $S_0(v,\tau)$. It can be shown  that the solution lies in $\mathcal{Y}(0,T)$, see~\cite[Proposition 2.3]{Sil16} and~\cite[Proposition 2.1]{HY95}.
For $y \in \mathcal{Y}(0,T)$, $v \in L^\infty(Q_0)$, and $\tau \in (0,T)$ we define the cost and the reduced cost functionals:
\begin{eqnarray*}
\tilde{J}_0(y,v,\tau)& = & \int_0^\tau \int_{\Omega} \ell_1(t,x,y(t,x),v(t,x)) \dd x \dd t + \int_{\Omega} \phi_1(\tau,x,y(\tau,x)) \dd x \\
& & + \int_\tau^T \int_{\Omega} \ell_2(t,x,y(t,x),u(t,x)) \dd x \dd t + \int_{\Omega} \phi_2(x,y(T,x)) \dd x, \\
J_0(v,\tau)& = & \tilde{J}_0(S_0(v,\tau),v,\tau).
\end{eqnarray*}
Finally, we fix four functions $u_1^- \leq u_1^+ \in L^\infty(\Omega)$ and $u_2^- \leq u_2^+ \in L^\infty(\Omega)$ and we define:
\begin{eqnarray*}
\mathcal{V}_{\text{ad}}& = & \Big\{ (v,\tau) \in L^\infty(Q_0) \times (0,T) \,|\,
u_1^-(x) \leq v(t,x) \leq u_1^+(x), \text{ for a.\,e.\@ }t \in (0,\tau), \\
& & \qquad \qquad \qquad \qquad \qquad \qquad 
u_2^-(x) \leq v(t,x) \leq u_2^+(x), \text{ for a.\,e.\@ }t \in (\tau,T) \Big\}.
\end{eqnarray*}
The optimal control problem to be investigated is given by
\begin{equation*} \label{problem0} \tag{$\mathcal{P}_0$}
\min_{(v,\tau) \in \mathcal{V}_{\text{ad}} } J_0(v,\tau).
\end{equation*}

\paragraph{Notation.}

Given a function $y \in L^\infty(Q)$, we denote for a.\,e.\@ $t \in [0,T]$ by $y(t)$ the function $x \in \Omega \mapsto y(t,x)$.
Given $t \in [0,T]$, $t' \in [0,T]$, $y$ and $u \in L^\infty(Q)$, we denote by $F_1(t',y(t),u(t))$ the function $x \in \Omega \mapsto f_1(t',x,y(t,x),u(t,x))$. In the same manner, we associate with the functions $f_2$, $\ell_1$, $\ell_2$, $\phi_1$, and $\phi_2$ the functions $F_2$, $L_1$, $L_2$, $\Phi_1$, and $\Phi_2$ respectively. This enables us to omit formally the dependence on $x$ for these functions.
We further define:
\begin{eqnarray} \label{eqConvention}
F: (s,t,y,u) \in (0,2) \times (0,T) \times L^\infty(\Omega) \times L^\infty(\Omega)
\mapsto
\begin{cases}
\begin{array}{ll}
F_1(t,y,u) & \text{ if $s \in (0,1)$}, \\
F_2(t,y,u) & \text{ if $s \in (1,2)$}.
\end{array}
\end{cases} \in L^\infty(\Omega)
\end{eqnarray}
Similarly, we define the mapping $L(s,\cdot,\cdot,\cdot)$ equal to $L_1(\cdot,\cdot,\cdot)$ for $s \in (0,1)$, equal to $L_2(\cdot,\cdot,\cdot)$ for $s \in (1,2)$. We later use a similar convention for the Hamiltonian.
Similarly, we set:
\begin{eqnarray*}
u^-(s,x)= \begin{cases}
\begin{array}{ll}
u_1^-(x) & \text{ if $s \in (0,1)$}, \\
u_2^-(x) & \text{ if $s \in (1,2)$},
\end{array}
\end{cases}
\quad
u^+(s,x)= \begin{cases}
\begin{array}{ll}
u_1^+(x) & \text{ if $s \in (0,1)$}, \\
u_2^+(x) & \text{ if $s \in (1,2)$}.
\end{array}
\end{cases}
\end{eqnarray*}
The interest of this notation will be clear in Section~\ref{sectionChangeOfVar}.

The partial derivative with respect to $x_i$ of a mapping $g \colon x=(x_1,...,x_n) \in X_1 \times ... \times X_n \rightarrow Y$ (where $X_1$,...,$X_n$ are Banach spaces) is denoted: $D_{x_i} g(x) \in \mathcal{L}(X_i,Y)$. The second-order partial derivative with respect to $x_i$ and $x_j$ will be denoted $D_{x_i,x_j}^2g(x)$ and will be seen as a bilinear mapping from $X_i \times X_j$ to $Y$.
The first-order and second-order derivatives with respect to all variables are denoted by $Dg(x)$ and $D^2g(x)$, respectively. Twice continuously Fr\'echet differentiable mappings will sometimes be simply called $C^2$ mappings.
Note that later, the Hamiltonian is never differentiated with respect to the variable $p$.

\subsection{Change of variables} \label{sectionChangeOfVar}

The formulation of problem~\eqref{problem0} does not enable us to derive optimality conditions. Indeed, the feasible controls are not continuous at time $\tau$ and the corresponding state trajectories are not continuously differentiable (in time), in general. Therefore, for fixed values of $v$ and $y$, the cost function $\tau \mapsto \tilde{J}_0(y,v,\tau)$ is not differentiable and for a fixed value of $v$, the control-to-state mapping $\tau \mapsto S(v,\tau)$ is not differentiable neither. Eventually, the mapping $(v,\tau) \mapsto J_0(v,\tau)$ is not differentiable with respect to $\tau$.
In the formulation of problem~\eqref{problem0}, an additional difficulty lies in the fact that the set of feasible controls depends on $\tau$.
These difficulties can be overcome by reformulating the problem with the following change of variable.
Consider the mapping $\pi:[0,2] \times (0,T) \rightarrow [0,T]$ defined by:
\begin{eqnarray*}
\pi(t,\tau) =
\begin{cases}
\begin{array}{ll}
\tau t & \text{ if $t \in [0,1]$,} \\
(T-\tau)t - T + 2 \tau & \text{ if $t \in [1,2]$}.
\end{array}
\end{cases}
\end{eqnarray*}
Observe that: $\pi(0,\tau)= 0$, $\pi(1,\tau)= \tau$, and $\pi(2,\tau)= T$. For future reference, we introduce the time-derivative $\dot{\pi}$ of $\pi$ (with respect to $t$) and the partial derivative of $\pi$ and $\dot{\pi}$ with respect to $\tau$, denoted respectively by $\pi_\tau$ and $\dot{\pi}_\tau$. It holds:
\begin{align*}
\dot{\pi}(t,\tau)=&
\begin{cases}
\begin{array}{ll}
\tau & \text{ if $t \in [0,1]$}, \\
(T-\tau) & \text{ if $t \in [1,2]$,}
\end{array}
\end{cases}
\dot{\pi}_\tau(t)=
\begin{cases}
\begin{array}{ll}
1 & \text{ if $t \in [0,1]$}, \\
-1 & \text{ if $t \in [1,2]$},
\end{array}
\end{cases} \\
\pi_\tau(t)=& \begin{cases}
\begin{array}{ll}
t & \text{ if $t \in [0,1]$, } \\
2-t & \text{ if $t \in [1,2]$}.
\end{array}
\end{cases}
\end{align*}
Note that we omit to write the variable $\tau$, for $\pi_\tau$ and $\dot{\pi}_\tau$.
Accordingly we define a new class  of feasible controls by
\begin{eqnarray*}
\mathcal{U}_{\text{ad}} & = & \Big\{ u \in L^\infty(Q) \,|\,
u_1^-(x) \leq u(t,x) \leq u_1^+(x), \text{ for a.\,e.\@ }t \in (0,1) \\
& & \qquad \qquad \qquad \ 
u_2^-(x) \leq u(t,x) \leq u_2^+(x), \text{ for a.\,e.\@ }t \in (1,2) \Big\},
\end{eqnarray*}
and consider the following mapping
\begin{eqnarray*}
\chi: (v,\tau) \in \mathcal{V}_{\text{ad}} \mapsto (u,\tau) \in \mathcal{U}_{\text{ad}} \times (0,T), \quad \text{where: } \quad u(t,x)= v(\pi(t,\tau),x).
\end{eqnarray*}
One can easily check that this is a bijection. This mapping will enable us to justify the well-posedness of the reparameterized system and then to establish the equivalence of problems~\eqref{problem0} and~\eqref{problem1}.

Given $(u,\tau) \in \mathcal{U}_{\text{ad}} \times (0,T)$, we set $S(u,\tau)= S_0 \big( \chi^{-1}(u,\tau) \big) \circ \pi(\cdot,\tau)$, where $\chi^{-1}$ is the inverse of $\chi$.
We observe that
$S(u,\tau) \in \mathcal{Y}$ is the strong variational solution to the following hybrid system:
\begin{eqnarray} \label{eqSystemReparam}
\left\{
\begin{array}{rcll}
\dot{y}(t) - \dot{\pi}(t,\tau) \big[ \Delta y(t) + F \big(t, \pi(t,\tau),y(t), u(t) \big) \big]
& = & 0 & \text{ in $Q$}, \\
y(t) & = & 0 & \text{ on $\Sigma$}, \\
y(0) & = & y_0 & \text{ in $\Omega$}.
\end{array}
\right.
\end{eqnarray}
For $y \in \mathcal{Y}$, $u \in L^\infty(Q)$, and $\tau \in (0,T)$, we set:
\begin{equation*}
\tilde{J}(y,u,\tau)= \int_{Q} \dot{\pi}(t,\tau) L\big(t, \pi(t,\tau),y(t),u(t) \big) \dd x \dd t + \int_{\Omega} \big[ \Phi_1(\tau,y(1)) + \Phi_2(y(2)) \big] \dd x,
\end{equation*}
and we define
\begin{eqnarray*}
J(u,\tau)= & \tilde{J}(S(u,\tau),u,\tau).
\end{eqnarray*}
These considerations lead to the \emph{reparameterized} problem:
\begin{equation*} \label{problem1} \tag{$\mathcal{P}_1$}
\min_{(u,\tau) \in \mathcal{U}_{\text{ad}} \times (0,T)} J(u,\tau).
\end{equation*}

\begin{lemma} \label{lemmaFirstEquivalence}
For all $(v,\tau) \in \mathcal{V}_{\text{ad}}$, we have $J(\chi(v,\tau))= J_0(v,\tau)$.
Consequently, if $(v,\tau) \in \mathcal{V}_{\text{ad}}$ is a global solution to~\eqref{problem0}, then $\chi(v,\tau)$ is a global solution to~\eqref{problem1} and conversely, if $(u,\tau) \in \mathcal{U}_{\text{ad}} \times (0,T)$ is a global solution to~\eqref{problem1}, then $\chi^{-1}(u,\tau)$ is a global solution to~\eqref{problem0}.
\end{lemma}

The verification is straightforward and therefore it is left to the reader. Further below we give optimality conditions for problem~\eqref{problem1}. Our sufficient second-order optimality conditions will ensure the local optimality property with respect to  the $L^\infty$-norm for problem~\eqref{problem1} (for the control variable). While there is an equivalence between globally optimal solutions, as stated in the above lemma, there are, however, no obvious notions of local optimality for~\eqref{problem0} and~\eqref{problem1} defined in such a way that if a given $(\bar{u},\bar{\tau})$ is a local solution to~\eqref{problem1}, then $\chi^{-1}(\bar{u},\bar{\tau})$ is a local solution to~\eqref{problem0}, and conversely.

\section{First- and second-order derivatives of the cost function} \label{sectionDerivatives}

In this section, we compute the first- and second-order derivatives $DJ(\bar{u},\bar{\tau})$ and $D^2J(\bar{u},\bar{\tau})$, at a fixed  $(\bar{u},\bar{\tau}) \in \mathcal{U}_{\text{ad}} \times (0,T)$. We further denote $\bar{y}=S(\bar{u},\bar{\tau})$ and henceforth  use the following notation:
\begin{eqnarray*}
F[t]= F(t,\pi(t,\bar{\tau}),\bar{y}(t),\bar{u}(t)).
\end{eqnarray*}
The analogous  notation is used  for $L$ and the derivatives of $F$ and $L$.

\subsection{Linearization of the system}

In this subsection,  the regularity of the control-to-state operator is investigated. We first give a regularity result on linear parabolic equations, Lemma~\ref{lemmaRegularityLinParabolic}, which is used throughout the paper.

\begin{lemma} \label{lemmaRegularityLinParabolic}
We consider
\begin{eqnarray}\label{eq35}
\left\{ \begin{array}{rcll}
\dot{z}(t,x)-\nu(t)[\Delta z(t,x)+b(t,x)z(t,x)] & = & \xi (t) & \text{ in $Q$} \\
z(t,x) & = & 0 & \text{ on $\Sigma$} \\
z(0,x) & = & z_0 & \text{ in $\Omega$},
\end{array}
\right.
\end{eqnarray}
where $\nu$ satisfies $a.\, e.$ in $(0,T)$ the condition $\frac{1}{c}\leq \nu \leq c$ for some constant $c>0$, $\|b \|_{L^{\infty}(Q)}\leq c$, and $\xi \in L^2(Q)$.
\begin{itemize}

\item [(a)]For all $z_0 \in L^{2}(\Omega)$, there exists a weak solution $z\in W(0,2)$ of~\eqref{eq35} and a constant $C$ depending on $c$ such that
\begin{eqnarray}\label{eq37a}
\|z\|_{W(0,2)} & \leq & C(\|\xi\|_{{L^2}(Q)}+\|z_0\|_{L^{2}(\Omega)}).
\end{eqnarray}
\item [(b)] If moreover $z_0 \in L^{\infty}(\Omega)$, we  have $z\in W(0,2)\cap L^{\infty}(Q) $ and there exists a constant $C$ depending on $c$ such that
\begin{eqnarray}\label{eq37}
\|z\|_{L^{\infty}(Q)} & \leq & C(\|\xi\|_{{L^2}(Q)}+\|z_0\|_{L^{\infty}(\Omega)}).
\end{eqnarray}
\item [(c)] For all $z_0 \in H^1_0(\Omega)\cap C(\bar \Omega)$, there exists a strong variational solution $z\in \mathcal{Y}$ of~\eqref{eq35} and a constant $C$ depending on $c$ such that
\begin{eqnarray}\label{eq36}
\|z\|_{\mathcal{Y}} & \leq & C(\|\xi\|_{{L^2}(Q)}+\|z_0\|_{H^1_0(\Omega)} +\|z_0\|_{C(\bar \Omega)}).
\end{eqnarray}
\end{itemize}
\end{lemma}

\begin{proof}
Assertion (a) is a standard result. For $z_0\in L^{\infty}(\Omega)$, the fact that~\eqref{eq35} has a weak variational solution in $W(0,2)\cap L^{\infty}(Q)$ and that~\eqref{eq37} holds follows from~\cite[Theorem 2.1, page 143]{LSU68} and~\cite[Theorem 7.1, page 181]{LSU68}.

Next we can take $\nu$, $b$ and $z\in L^2(\Omega)$ from the left to the right hand side in the first equation of~\eqref{eq35} and apply~\cite[Theorem 6.1, page 178]{LSU68} and~\cite[(6.10), page 176]{LSU68} in order to obtain~\eqref{eq36}, with  $C(\bar \Omega)$ and  $C(\bar Q)$ replaced by $L^\infty (\Omega)$ and $L^\infty( Q)$. To argue that $z \in C(\bar Q)$ we consider
\begin{eqnarray}\label{eq35a}
\begin{cases}
\begin{array}{rcll}
\dot{\zeta}(t,x) - c^{-1} \Delta \zeta  & =& h(t,x) + \xi (t)& \text{in $Q$}, \\
\zeta(t,x) & = & 0 & \text{in $\Sigma$}, \\
\zeta(0,x) & = & z_0 & \text{in $\Omega$},
\end{array}
\end{cases}
\end{eqnarray}
where
\begin{eqnarray*}
h(t,x)=(\nu(t)-c^{-1}) \Delta (t,x) z(t,x)+ \nu(t) b(t,x)z(t,x).
\end{eqnarray*}
Since $z \in L^2 \big((0,2),H^{2}(\Omega) \cap H_0^1(\Omega) \big) \cap W^{1,2} \big( (0,2),L^2(\Omega) \big)$ we have that $h\in L^2((0,2), L^2(\Omega))$, and, of course, that $\zeta=z$ in $L^\infty(Q)$.
The continuity of the solution to an equation with constant coefficients like~\eqref{eq35a} is proved in~\cite[Lemma 7.12]{Tro10}, in the case of a Neumann boundary condition. The proof can be adapted to Dirichlet boundary conditions by replacing all along the proof the space $C(\bar{\Omega})$ by the space $C_0(\bar{\Omega})= \{ y \in C(\bar{\Omega}) \,|\, y_{|\Sigma}= 0 \}$.
\end{proof}

\begin{lemma} \label{lemmaRegS}
The control-to-state mapping $(u,\tau) \in L^\infty(Q) \times (0,T) \mapsto S(u,\tau) \in \mathcal{Y}$ is twice Fr\'echet-differentiable.
\end{lemma}

\begin{proof}

\emph{Step 1: General approach}.

Consider the mapping $e: {\mathcal{Y}} \times L^\infty(Q) \times (0,T) \rightarrow L^2(Q) \times (H^1_0(\Omega)\cap C(\bar \Omega))$, defined by:
\begin{eqnarray*}
e(y,u,\tau) & = & \big (
\dot{y}(\cdot) - \dot{\pi}(\cdot,\tau)\big[ \Delta y(\cdot) + F \big(\cdot, \pi(\cdot,\tau),y(\cdot),u(\cdot) \big) \big] , y(0)-y_0 \big).
\end{eqnarray*}
For any $(y,u,\tau) \in {\mathcal{Y}} \times \mathcal{U}_{\text{ad}} \times [0,T]$, the equality $y= S(u,\tau)$ holds if and only if $e(y,u,\tau)=0$. In the next two steps, we prove that $e$ is $C^2$ and that for any $(y,u,\tau) \in \mathcal{Y} \times L^\infty(Q) \times (0,T)$, $D_y e(y,u,\tau)$ is surjective. Then, it follows with the implicit function theorem that $S$ is $C^2$.

\emph{Step 2: Regularity of $e$}.
Observe first that the mappings
\begin{eqnarray*}
G_1 \colon y \in {\mathcal{Y}} \mapsto \dot{y} \in L^2(Q) & \text{and} &
G_2 \colon y \in {\mathcal{Y}} \mapsto \Delta y \in L^2(Q)
\end{eqnarray*}
are well-defined, linear and continuous, thus $C^2$.
It is easy to check that the mappings
\begin{eqnarray*}
G_3 \colon \tau \in [0,T] \mapsto \pi(\cdot,\tau) \in L^\infty(0,2) & \text{and} &
G_4 \colon \tau \in [0,T] \mapsto \dot{\pi}(\cdot,\tau) \in L^\infty(0,2)
\end{eqnarray*}
 are $C^2$. The mapping
\begin{eqnarray*}
G_5 \colon (h_1,h_2) \in L^\infty(Q) \times L^2(Q) \mapsto h_1(\cdot,\cdot) h_2(\cdot,\cdot) \in L^2(Q)
\end{eqnarray*}
is well-defined, continuous and bilinear (by Cauchy-Schwarz inequality), therefore $C^2$. Define:
\begin{eqnarray*}
f \colon (s,t,x,y,u) \in [0,2] \times [0,T] \times \Omega \times \R^2 \mapsto
\begin{cases}
\begin{array}{ll}
f_1(t,x,y,u) & \text{ if $s \in (0,1)$}, \\
f_2(t,x,y,u) & \text{ if $s \in (1,2)$}.
\end{array}
\end{cases}
\end{eqnarray*}
The mapping $f$ is measurable and satisfies the boundedness condition with respect to $(s,x)$ and the local Lipschitz condition of order 2 with respect to $(t,y,u)$ (see Definition~\ref{defNemytskii}). Therefore, by Lemma~\ref{lemmaNemytskii} of the Appendix, the following mapping is $C^2$:
\begin{eqnarray*}
G_6:(\rho,y,u) \in L^\infty(0,2) \times L^\infty(Q)^2
\mapsto F \big( \cdot, \rho(\cdot),y(\cdot),u(\cdot) \big) \in L^\infty(Q).
\end{eqnarray*}
Its restriction to $L^\infty(0,2) \times {\mathcal{Y}} \times L^\infty(Q)$ is also $C^2$, and its image is embedded into $L^2(Q)$.
We deduce that $e$ is $C^2$, since:
\begin{eqnarray} \label{eqRepresentationE}
e(y,u,\tau)= \big (G_1(y)- G_5 \big( G_4(\tau), G_2(y) + G_6(G_3(\tau),y,u)  \big), y(0)-y_0 \big).
\end{eqnarray}

\emph{Step 3: Surjectivity of $D_y e(y,u,\tau)$}.
Let $(y,u,\tau) \in \mathcal{Y} \times L^\infty(Q) \times (0,T)$ and let $\xi \in L^2(Q)$. By Lemma~\ref{lemmaRegularityLinParabolic}, the following linear parabolic equation has a unique solution in $\mathcal{Y}$:
\begin{eqnarray} \label{eqLinearizedSystem}
\begin{cases}
\begin{array}{rcll}
\dot{z}(t)-\dot{\pi}(t,\tau) \big[ \Delta z(t) + D_y F \big(t, \pi(t,\tau), y(t), u(t) \big) z(t) \big]&=& \xi(t)  & \text{ in $Q$}, \\
z(t)&=& 0 & \text{ on $\Sigma$}, \\
z(0)&=& z_0 & \text{ in $\Omega$}.
\end{array}
\end{cases}
\end{eqnarray}
This proves the surjectivity of $D_y e(y,u,\tau): {\mathcal{Y}}_0 \to  H^1_0(\Omega)\cap C(\bar \Omega)$ and concludes the proof.
\end{proof}

\begin{remark}
The implicit function theorem may appear as a too strong tool for verifying the differentiability of $S$, since it involves the use of strong variational solutions. Their use, however, is unavoidable for our analysis. Indeed, consider the linearization of system~\eqref{eqSystemReparam} with respect to $\tau$:
\begin{eqnarray*}
\begin{cases}
\begin{array}{rcll}
\dot{z}(t) - \dot{\pi}(t,\tau) \big[ \Delta z(t) + D_y F \big(t, \pi(t,\tau),y(t), u(t) \big) z(t) \big]& = &
\xi(t) & \text{ in $Q$}, \\
z(t) & = & 0 & \text{ on $\Sigma$}, \\
z(0)& = & z_0 & \text{ in $\Omega$},
\end{array}
\end{cases}
\end{eqnarray*}
where:
\begin{eqnarray*}
\xi(t)= \pi_\tau(t) \big[ \Delta y(t) + F \big( t,\pi(t,\tau),y(t),u(t) \big) \big]
+ \pi(t,\tau) D_tF \big( t,\pi(t,\tau),y(t),u(t) \big).
\end{eqnarray*}
We shall need an  $L^\infty(Q)$ bound for $z$. This can only be achieved if the right-hand side $\xi$ is an element of $L^2(Q)$, which explains why $\Delta y$ should itself be in $L^2(Q)$.
\end{remark}

Consider now the linear mapping defined by
\begin{eqnarray} \label{eqDefOperatorK}
\mathcal{K}: \xi \in L^2(Q) & \mapsto & (z(1), z(2), z) \in  C(\bar \Omega)^2 \times \mathcal{Y} \subset L^2(\Omega)^2 \times L^2(Q),
\end{eqnarray}
where $z$ is the solution to the linearized equation~\eqref{eqLinearizedSystem}, with $(y,u,\tau)= (\bar{y},\bar{u},\bar{\tau})$ and $z_0=0$. The well-posedness and the continuity of the linear mapping $\mathcal{K}$ is ensured by Lemma~\ref{lemmaRegularityLinParabolic}.

We introduce the following space:
\begin{eqnarray*}
W\big( (0,1),(1,2) \big) & = & \big\{ y \in L^2((0,2),H_0^1(\Omega)) \,|\, y_{|(0,1) \times \Omega} \in W(0,1), y_{|(1,2) \times \Omega} \in W(1,2) \big\}.
\end{eqnarray*}
Any function $y \in W\big( (0,1),(1,2) \big)$ is continuous on $[0,1)$ and $(1,2]$ and has a left- and a right-hand limit at time 1, that we denote by $y(1^-)$ and $y(1^+)$, respectively. The jump is denoted by $[y](1)=y(1^+)-y(1^-)$.

We define now the mapping $\mathcal{K}^*:(a,b,\xi) \in L^2(\Omega)^2 \times L^2(Q) \mapsto q \in W\big( (0,1),(1,2) \big)\subset L^2(Q)$, where $q$ is the solution to:
\begin{eqnarray*}
\begin{cases}
\begin{array}{rcll}
-\dot{q}(t)-\dot{\pi}(t,\bar{\tau}) \big[ \Delta q(t) + D_y F_2 \big( \pi(t,\bar{\tau}), \bar{y}(t), \bar{u}(t) \big) q(t) \big]&=&
\xi(t)  & \text{ in $(1,2) \times \Omega$}, \\
q(t)&=& 0 & \text{ on $(1,2) \times \partial \Omega$}, \\
q(2)&=& b & \text{ in $\Omega$}, \\
-\dot{q}(t)-\dot{\pi}(t,\bar{\tau}) \big[ \Delta q(t) + D_y F_1 \big( \pi(t,\bar{\tau}), \bar{y}(t), \bar{u}(t) \big) q(t) \big]&=&
\xi(t)  & \text{ in $(0,1) \times \Omega$}, \\
q(t)&=& 0 & \text{ on $(0,1) \times \partial \Omega$}, \\
q(1^-)&= & q(1^+)+a & \text{ in $\Omega$}. \\
\end{array}
\end{cases}
\end{eqnarray*}
The function $q$ is the concatenation of the solution to two backward linear parabolic equations: One on $(1,2) \times \Omega$, with terminal condition $b$, and one on $(0,1) \times \Omega$, with terminal condition $q(1^+)+a$. The well-posedness is ensured by Lemma~\ref{lemmaRegularityLinParabolic} (a) (after the change of variable $t \mapsto 2-t$). Note here that $q_{|(1,2) \times \Omega} \in W(1,2)$ and hence $q(1^+) \in L^2(\Omega)$. We establish below that $\mathcal{K}^*$ is the adjoint operator of $\mathcal{K} \in \mathcal{L}(L^2(Q),L^2(\Omega)^2 \times L^2(Q))$.

\begin{lemma} \label{lemmaAdjoint}
For all $\xi \in L^2(Q)$, for all $(a,b,w) \in L^2(\Omega)^2 \times L^2(Q)$, it holds:
\begin{eqnarray*}
\langle (a,b,w), \mathcal{K}(\xi) \rangle_{L^2(\Omega)^2 \times L^2(Q)}
= \langle \mathcal{K}^*(a,b,w), \xi \rangle_{L^2(Q)}.
\end{eqnarray*}
\end{lemma}

\begin{proof}
Set $z= \mathcal{K}(\xi) \in \mathcal{Y}$ and $q = \mathcal{K}^*(a,b,w)$. Let $q_1= q_{|(0,1) \times \Omega} \in W(0,1)$, let $q_{|(1,2) \times \Omega} \in W(1,2)$.
We obtain, by integration by parts:\\
\vspace*{-10pt}
\begin{flushleft} 
$\langle (a,b,w), \mathcal{K}(\xi) \rangle_{L^2(\Omega)^2 \times L^2(Q)}$
\end{flushleft}
\begin{eqnarray*}
& = &
\langle w, z \rangle_{L^2(Q)} + \langle a, z(1) \rangle_{L^2(\Omega)} + \langle b, z(2) \rangle_{L^2(\Omega)} \\
& = & \int_0^1 \langle -\dot{q}_1(t) - \dot{\pi}(t,\tau)\big[ \Delta q_1(t) + D_yF_1[t] q_1(t) \big],z(t) \big] \rangle_{H^{-1}(\Omega),H_0^1(\Omega)} \dd t \\
& &  + \int_1^2 \langle -\dot{q}_2(t) - \dot{\pi}(t,\tau)\big[ \Delta q_2(t) + D_yF_2[t] q_2(t) \big],z(t) \big] \rangle_{H^{-1}(\Omega),H_0^1(\Omega)} \dd t \\
& &  + \langle a, z(1) \rangle_{L^2(\Omega)} + \langle b, z(2) \rangle_{L^2(\Omega)} \\
& = & \int_0^1 \langle q(t), \dot{z}(t) - \dot{\pi}(t,\tau) \big[ \Delta z(t) + D_yF_1[t] z(t) \big] \rangle_{L^2(\Omega)} \dd t - \langle q(1^-), z(1) \rangle_{L^2(\Omega)} \\
& &  + \int_1^2 \langle q(t), \dot{z}(t) - \dot{\pi}(t,\tau) \big[ \Delta z(t) + D_yF_2[t] z(t) \big] \rangle_{L^2(\Omega)} \dd t \\
& & + \langle q(1^+), z(1) \rangle_{L^2(\Omega)} - \langle q(2), z(2) \rangle_{L^2(\Omega)} 
 + \langle a, z(1) \rangle_{L^2(\Omega)} + \langle b, z(2) \rangle_{L^2(\Omega)} \\
& = & \int_0^2 \langle q(t),\xi(t) \rangle_{L^2(\Omega)} \dd t.
\end{eqnarray*}
The proof is complete.
\end{proof}

We define now the following mapping:
\begin{eqnarray*}
\mathbf{S}: (u,\tau) \in L^\infty(Q) \times (0,T) \mapsto \big( S(u,\tau)(1), S(u,\tau)(2), S(u,\tau) \big) \in C(\bar\Omega)^2 \times \mathcal{Y}.
\end{eqnarray*}
As a consequence of Lemma~\ref{lemmaRegS}, $\textbf{S}$ is twice continuously Fr\'echet differentiable, with:
\begin{eqnarray*}
D_u \mathbf{S}(\bar{u},\bar{\tau}) v & = & \mathcal{K}\big( \dot{\pi}(\cdot,\bar{\tau}) D_u F[\cdot] v(\cdot) \big), \text{ for $v \in L^\infty(Q)$,} \\
D_\tau \mathbf{S}(\bar{u},\bar{\tau})& = & \mathcal{K} \big(\dot{\pi}_\tau(\cdot) (\Delta \bar{y}(\cdot) + F[\cdot]) + \dot{\pi}(\cdot,\bar{\tau})D_t F[\cdot]\pi_\tau(\cdot) \big).
\end{eqnarray*}
Lemma~\ref{lemmaAdjoint} allows us to compute the adjoint operators of $D_u\mathbf{S}(\bar{u},\bar{\tau})$ and $D_\tau \mathbf{S}(\bar{u},\bar{\tau})$ as mappings from $L^2(Q)$ and $\R$ respectively, to $L^2(\Omega)^2\times L^2(Q)$.
For $v \in L^2(Q)$ and $(a,b,w) \in L^2(\Omega)^2 \times L^2(Q)$, it holds:
\begin{eqnarray*}
 \langle (a,b,w), D_u \mathbf{S}(\bar{u},\bar{\tau}) v \rangle_{L^2(\Omega)^2 \times L^2(Q)}
& = &  \big\langle \dot{\pi}(\cdot,\bar{\tau}) D_u F[\cdot] \mathcal{K}^*(a,b,w), v(\cdot) \big\rangle_{L^2(Q)}, \\
 \langle (a,b,w), D_\tau \mathbf{S}(\bar{u},\bar{\tau}) \rangle_{L^2(\Omega)^2 \times L^2(Q)}
& = & \int_{Q} \big( \dot{\pi}_\tau(\cdot) F[\cdot] + \dot{\pi}(\cdot,\bar{\tau}) D_tF[\cdot] \pi_\tau(\cdot) \big) \mathcal{K}^*(a,b,w) \dd x \dd t.
\end{eqnarray*}

\subsection{Computation of the derivatives of the cost functional}

First we define the mappings $H_1$ and $H_2$ as follows
\begin{eqnarray*}
H_i: (t,y,u,p) \in (0,T) \times C(\bar{\Omega}) \times L^\infty(\Omega) \times L^2(\Omega) \mapsto
\int_{\Omega} L_i(t,y,u) \dd x + \langle p, F_i(t,y,u) \rangle_{L^2(\Omega)},
\end{eqnarray*}
for $i \in \{ 1, 2\}$. We define the Hamiltonian $H$ in the same fashion as in~\eqref{eqConvention}:
\begin{equation} \label{eqDefHam}
 \begin{array} {lrcl}
H: &  (0,2) \times [0,T] \times C(\bar{\Omega}) \times L^\infty(\Omega) \times L^2(\Omega) 
& \rightarrow & \R \\
 & (s,t,y,u,p)
&\mapsto &
\begin{cases}
\begin{array}{ll}
H_1(t,y(t),u(t),p(t)) & \text{ if $s \in (0,1)$}, \\
H_2(t,y(t),u(t),p(t)) & \text{ if $s \in (1,2)$}.
\end{array}
\end{cases}
\end{array} 
\end{equation}
Note that in what follows, the partial derivative $D_tH$ (as well as the other second-order partial derivatives in $t$) refers to the partial derivative with respect to the second variable in the definition of $H$.

We define the costate $\bar{p} \in W\big((0,1),(1,2)\big)$ as:
\begin{eqnarray} \label{eqAdjoint}
\bar{p} & = & \mathcal{K}^* \big( D_y \Phi_1(\bar{\tau},\bar{y}(1)),\, D_y \Phi_2(\bar{y}(2)),\, \dot{\pi}(\cdot, \bar{\tau}) D_y L(\cdot, \pi(\cdot,\bar{\tau}),\bar{y}(\cdot), \bar{u}(\cdot) )\big).
\end{eqnarray}
Observe that $\bar{p}$ is the solution to the following equation:
\begin{eqnarray} \label{eqCostate}
\begin{cases}
\begin{array}{rcll}
-\dot{\bar{p}}(t)-\dot{\pi}(t,\bar{\tau}) \big[ \Delta \bar{p}(t) + D_y H \big(t,\pi(t,\bar{\tau}),\bar{y}(t),\bar{u}(t),\bar{p}(t) \big) \big]&=&
0  & \text{ in $Q$}, \\
\bar{p}&=& 0 & \text{ on $\Sigma$}, \\
\bar{p}(2)&=& D_y \Phi_2(\bar{y}(2)) & \text{ in $\Omega$}, \\
-[\bar{p}](1)&= & D_y \Phi_1(\bar{\tau},\bar{y}(1)) & \text{ in $\Omega$}. \\
\end{array}
\end{cases}
\end{eqnarray}
From now, we use the following notation:
\begin{eqnarray*}
H[t]& = &H \big(t,\pi(t,\bar{\tau}),\bar{y}(t),\bar{u}(t),\bar{p}(t) \big).
\end{eqnarray*}
We define the following Lagrangian functional by:
\begin{eqnarray*}
\begin{array} {rrcl}
\mathcal{L} :  & (C(\bar{\Omega})^2 \times \mathcal{Y}) \times L^\infty(Q) \times (0,T) & \rightarrow & \R \\
& \big( \mathbf{y}= (a_1,a_2,y), u, \tau \big) & \mapsto & \displaystyle
\int_{\Omega} \Phi_1(\tau,a_1) \dd x + \int_{\Omega} \Phi_2(a_2) \dd x \\
& & & \displaystyle + \int_0^2 \dot{\pi}(t,\tau) H \big( t,\pi(t,\tau),y(t),u(t),\bar{p}(t) \big) \dd t \\
& & & \displaystyle + \phantom{\Big|}\langle \bar{p}(\cdot), \Delta y(\cdot) -\dot{y}(\cdot) \rangle_{L^2(Q)} 
 - \langle \bar{p}(0), y(0) - y_0 \rangle_{L^2(\Omega)} \\
& & & \displaystyle+ \phantom{\Big|} \langle \bar{p}(2), y(2)-a_2 \rangle_{L^2(\Omega)}
- \langle [\bar{p}](1), y(1)-a_1 \rangle_{L^2(\Omega)}.
\end{array}
\end{eqnarray*}
In the above definition, $\bar{p} \in W\big((0,1),(1,2)\big)$ is the costate associated with $(\bar{y},\bar{u},\bar{\tau})$. The use of an appropriately defined Lagrange functional is important for obtaining convenient expressions for the derivatives of the reduced cost functional $J(u,\tau)$. See Proposition~\ref{le:derredcost} below. We now denote
\begin{eqnarray*}
\bar{\mathbf{y}}= (\bar{y}(1),\bar{y}(2),\bar{y})= \mathbf{S}(\bar{u}, \bar{\tau}).
\end{eqnarray*}

\begin{lemma} \label{lemmaTwiceDiffLag}
The Lagrangian is twice continuously differentiable. The first-order partial derivatives at $(\bar{\mathbf{y}},\bar{u},\bar{\tau})$ are given by:
\begin{eqnarray*}
 D_{\mathbf{y}} \mathcal{L}(\bar{\mathbf{y}},\bar{u},\bar{\tau}) \mathbf{z} & = & 0, 
 \text{ for } \mathbf{z} \in C(\bar{\Omega})^2\mathcal{Y},\\
 D_u \mathcal{L}(\bar{\mathbf{y}},\bar{u},\bar{\tau}) v
& = & \int_0^2 \dot{\pi}(t,\bar{\tau}) D_uH[t] v(t) \dd t, \text{ for } v \in L^{\infty}(Q), \\
D_{\tau} \mathcal{L}(\bar{\mathbf{y}},\bar{u},\bar{\tau}) &= &
 \int_0^2 \big( \dot{\pi}_\tau(t) H[t] + \dot{\pi}(t,\bar{\tau}) D_tH[t] \pi_{\tau}(t) \big) \dd t + \int_{\Omega} D_\tau \Phi_1(\bar{\tau},\bar{y}(1)) \dd x.
\end{eqnarray*}
The second-order partial derivatives at $(\bar{\mathbf{y}},\bar{u},\bar{\tau})$ in directions $z,\hat z$ in ${\mathcal{Y}}$, $v,\hat v$ in $L^\infty(\Omega)$, $\delta a_i,  \delta \hat a_i$ for $i\in\{1,2\}$ in $C(\bar{\Omega})$ and $\theta \in \R$ are given by:
\begin{eqnarray*}
 D_{a_1,a_1}^2 \mathcal{L}(\bar{\mathbf{y}},\bar{u},\bar{\tau})(\delta a_1, \delta \hat{a}_1)
& = & \int_{\Omega} D_{yy}^2 \Phi_1(\bar{\tau},\bar{y}(1))(\delta a_1, \delta \hat{a}_1) \dd x, \\
 D_{a_2,a_2}^2 \mathcal{L}(\bar{\mathbf{y}},\bar{u},\bar{\tau})(\delta a_2, \delta \hat{a}_2)
& = & \int_{\Omega} D_{yy}^2 \Phi_2(\bar{y}(2))(\delta a_2, \delta \hat{a}_2) \dd x, \\
 D_{\tau,a_1}^2 \mathcal{L}(\bar{\mathbf{y}},\bar{u},\bar{\tau})(\theta,\delta a_1)
& = & \int_{\Omega} D_{\tau,y}^2 \Phi_1(\bar{\tau},\bar{y}(1))(\theta,\delta a_1) \dd x, \\
D_{(y,u)^2}^2 \mathcal{L}(\bar{\mathbf{y}},\bar{u},\bar{\tau})
\scriptscriptstyle
\begin{pmatrix}
\begin{pmatrix} z \\ v \end{pmatrix},
\begin{pmatrix} \hat{z} \\ \hat{v} \end{pmatrix}
\end{pmatrix}
\displaystyle
& = & \int_0^2 \dot{\pi}(t,\bar{\tau}) D_{(y,u)^2}^2 H[t]
\scriptscriptstyle
\begin{pmatrix}
\begin{pmatrix} z(t) \\ v(t) \end{pmatrix},
\begin{pmatrix} \hat{z}(t) \\ \hat{v}(t) \end{pmatrix}
\end{pmatrix}
\displaystyle
\dd t, \\
D_{\tau,(y,u)}^2 \mathcal{L} (\bar{\mathbf{y}},\bar{u},\bar{\tau})
\scriptscriptstyle
\begin{pmatrix}
\theta,
\begin{pmatrix}
z \\ v
\end{pmatrix}
\end{pmatrix}
\displaystyle
& = & \int_0^2 \dot{\pi}_\tau(t) \theta D_{(y,u)} H[t]
\begin{pmatrix}
z(t) \\ v(t)
\end{pmatrix}
\dd t
\displaystyle \\
& & + \int_0^2
\dot{\pi}(t,\bar{\tau}) D^2_{\tau,(y,u)} H[t] \pi_\tau(t)
\scriptscriptstyle
\begin{pmatrix} \theta,
\begin{pmatrix}
z(t) \\ v(t)
\end{pmatrix}
\end{pmatrix}
\displaystyle
\dd t, \\
 D_{\tau \tau}^2 \mathcal{L}(\bar{\mathbf{y}},\bar{u},\bar{\tau})
 & = & \int_0^2 \left(2 \dot{\pi}_\tau(t) D_\tau H[t] \pi_\tau(t) + \dot{\pi}(t,\bar{\tau}) D_{tt}^2 H[t] \pi_\tau(t)^2 \right) \dd t.
\end{eqnarray*}
The derivatives which are not given above are null.
\end{lemma}

\begin{proof}

We give a proof of the differentiability of the Lagrangian and compute the first-order derivative with respect to $y$. The proofs for the other first-order derivatives and for the second-order derivatives are left to the reader.

\emph{Step 1: Differentiability of $\mathcal{L}$.}
Note that the mappings $G_1$,...,$G_5$ used in this proof are different from the ones introduced in the proof of Lemma~\ref{lemmaRegS}.
Observe first that all the linear terms involved in $\mathcal{L}$ are continuous and  are thus $C^2$.
The mappings $G_1 \colon h \in C(\bar{\Omega}) \mapsto \int_Q h$ and $G_2 \colon h \in L^\infty(Q) \mapsto \langle \bar{p},h \rangle_{L^2(Q)}$ are linear and continuous, thus $C^2$.  
By Lemma~\ref{lemmaNemytskii}, the mappings $(\tau,a_1) \in (0,T) \times C(\bar{\Omega}) \mapsto \Phi_1(\tau,a_1)$ and $a_2 \in C(\bar{\Omega}) \mapsto \Phi_2(a_2)$ are $C^2$, thus by composition with $G_1$, the following mappings are $C^2$:
\begin{eqnarray*}
(\tau, a_1) \in (0,T) \times C(\bar{\Omega}) \mapsto \int_{\Omega} \Phi_1(\tau,a_1) \dd x 
& \text{and} &
a_2 \in C(\bar{\Omega}) \mapsto \int_{\Omega} \Phi_2(a_2) \dd x.
\end{eqnarray*}
By using techniques of the proof of Lemma~\ref{lemmaRegS}, we can verify that the following mappings are $C^2$:
\begin{eqnarray*}
G_3 \colon (y,u,\tau) \in \mathcal{Y} \times L^\infty(Q) \times (0,T) &  \mapsto & \dot{\pi}(\cdot,\tau) L \big( \cdot,\pi(\cdot,\tau),y(\cdot),u(\cdot)\big) \in L^\infty(Q), \\
G_4 \colon (y,u,\tau) \in \mathcal{Y} \times L^\infty(Q) \times (0,T) & \mapsto & \dot{\pi}(\cdot,\tau) F \big( \cdot,\pi(\cdot,\tau),y(\cdot),u(\cdot)\big) \in L^\infty(Q).
\end{eqnarray*}
Thus, since $G_5(y,u,\tau)= G_1(G_3(y,u,\tau)+ G_2(G_4(y,u,\tau))$, the following mapping is $C^2$:
\begin{eqnarray*}
G_5:(y,u,\tau) \in \mathcal{Y} \times L^\infty(Q) \times [0,T] \mapsto \int_0^2 \dot{\pi}(t,\tau) H \big(t,\pi(t,\tau),y(t),u(t),\bar{p}(t) \big) \dd x \dd t.
\end{eqnarray*}
This concludes the first part of the proof.

\emph{Step 2: Computation of $D_{\mathbf{y}}\mathcal{L}(\bar{\mathbf{y}},\bar{u},\bar{\tau})$}.
Observe first that by~\eqref{eqCostate} we find
\begin{eqnarray*}
D_{a_1} \mathcal{L}(\bar{\mathbf{y}},\bar{u},\bar{\tau})= D_y \Phi_1(\bar{\tau},\bar{y}(1)) + [\bar{p}](1) = 0, & & 
D_{a_2} \mathcal{L}(\bar{\mathbf{y}},\bar{u},\bar{\tau})= D_y \Phi_2(\bar{\tau},\bar{y}(2)) - \bar{p}(2) = 0.
\end{eqnarray*}
Moreover,
\begin{eqnarray*}
D_y \mathcal{L}(\bar{\mathbf{y}},\bar{u},\bar{\tau})z & = & \int_0^2 \dot{\pi}(t,\bar{\tau}) D_y H[t] z(t) \dd t
+ \left\langle \bar{p}, \dot{\pi}(\cdot,\bar{\tau}) \big[ \Delta z(\cdot) + D_yF[\cdot] z(\cdot) \big] - \dot{z}(\cdot) \right\rangle_{L^2(Q)} \\
& & - \langle \bar{p}(0), \delta y(0) \rangle_{L^2(\Omega)}
+ \langle \bar{p}(2), \delta y(2) \rangle_{L^2(\Omega)}
- \langle [\bar{p}](1), \delta y(1) \rangle_{L^2(\Omega)} \\
& = & \langle \dot{\pi}(\cdot,\bar{\tau}) D_yL[\cdot], z \rangle_{L^2(Q)}
+ \left\langle \bar{p}, \dot{\pi}(\cdot,\bar{\tau}) \big[ \Delta z(\cdot) + D_y F[\cdot]z(\cdot) \big] - \dot{z}(\cdot) \right\rangle_{L^2(Q)} \\
& & - \langle \bar{p}(0), z(0) \rangle_{L^2(\Omega)}
+ \langle \bar{p}(2), z(2) \rangle_{L^2(\Omega)}
- \langle [\bar{p}](1), z(1) \rangle_{L^2(\Omega)} \\
& = &  0.
\end{eqnarray*}
Indeed, by Lemma~\ref{lemmaAdjoint} and using the definition of $\bar{p}$ given by~\eqref{eqAdjoint}, we obtain\\
\vspace*{-10pt}
\begin{flushleft}
$ \left\langle \bar{p},\, \dot{\pi}(\cdot,\bar{\tau}) \big[ \Delta z(\cdot) + D_y F[\cdot]z(\cdot) \big]- \dot{z}(\cdot) \right\rangle_{L^2(Q)}$
\end{flushleft}\vspace*{-10pt}
\begin{eqnarray*}
& = &  \left\langle (D_y\Phi_1(1,\bar{y}(1)), D_y \Phi_2(2,\bar{y}(2)), \dot{\pi}(\cdot,\bar{\tau}) D_yL[\cdot],\,
\mathcal{K}\big( \dot{\pi} \big[ \Delta z + D_y F[\cdot] z(\cdot) ]- \dot{z}(\cdot) \big) \right\rangle_{L^2(\Omega)^2 \times L^2(Q)} \\
& = & -\left\langle \big( D_y\Phi_1(1,\bar{y}(1)), D_y \Phi_2(2,\bar{y}(2)), \dot{\pi}(\cdot,\bar{\tau})D_yL[\cdot] \big),\, (z(1),z(2),z) \right\rangle_{L^2(\Omega)^2 \times L^2(Q)}\\
& =  &  -\left\langle -[\bar p](1), \bar p(2), \dot{\pi}(\cdot,\bar{\tau})D_y L[\cdot] \big),\, (z(1),z(2),z) \right\rangle_{L^2(\Omega)^2 \times L^2(Q)}.
\end{eqnarray*}
Thus the proof is complete
\end{proof}

In the next lemma, we compute the first- and second-order derivatives of $J$ at $(\bar{u},\bar{\tau})$.
Observe that the second-order derivative
$D^2 \mathcal{L}(\bar{\mathbf{y}},\bar{u},\bar{\tau})$ can be identified with a continuous endomorphism in $(C(\bar{\Omega})^2 \times \mathcal{Y}) \times L^\infty(Q) \times \R$ by writing
\begin{eqnarray*}
D^2 \mathcal{L}(\bar{\mathbf{y}},\bar{u},\bar{\tau})
\scriptscriptstyle
\begin{pmatrix}
\begin{pmatrix} \mathbf{z} \\ v \\ \theta \end{pmatrix},
\begin{pmatrix} \hat{\mathbf{z}} \\ \hat{v} \\ \hat{\theta} \end{pmatrix}
\end{pmatrix}
\displaystyle
& = &
\left\langle
D^2 \mathcal{L}(\bar{\mathbf{y}},\bar{u},\bar{\tau})
\scriptscriptstyle
\begin{pmatrix} \mathbf{z} \\ v \\ \theta \end{pmatrix} \displaystyle , \scriptscriptstyle
\begin{pmatrix} \hat{\mathbf{z}} \\ \hat{v} \\ \hat{\theta} \end{pmatrix}
\right\rangle_{(L^2(\Omega)^2 \times L^2(Q)) \times L^2(Q) \times \R}.
\end{eqnarray*}
In formula~\eqref{eqSecondDerivative} of Proposition~\ref{le:derredcost} below, the multiplication of matrices of linear mappings is defined by the usual rules of matrix calculus. The multiplication of two linear mappings is defined as their composition.

\begin{proposition}\label{le:derredcost}
For all $(u,\tau) \in L^\infty(Q) \times (0,T)$,
\begin{eqnarray} \label{eqJisL}
J(u,\tau) & = & \mathcal{L}(\mathbf{S}(u,\tau),u,\tau).
\end{eqnarray}
The first-order derivatives of $J$ at $(\bar{u},\bar{\tau})$ are given by
\begin{eqnarray*}
D_u J(\bar{u},\bar{\tau})= D_u \mathcal{L}(\bar{\mathbf{y}},\bar{u},\bar{\tau}) 
& \text{and} &
D_\tau J(\bar{u},\bar{\tau})= D_\tau \mathcal{L}(\bar{\mathbf{y}},\bar{u},\bar{\tau}),
\end{eqnarray*}
where $\bar{\mathbf{y}}= S(\bar{u},\bar{\tau})$.
The second-order derivative of $J$ at $(\bar{u},\bar{\tau})$, $D^2 J(\bar{u},\bar{\tau}),
\scriptscriptstyle
\begin{pmatrix}
\begin{pmatrix} v \\ \theta \end{pmatrix},
\begin{pmatrix} \hat{v} \\ \hat{\theta} \end{pmatrix}
\end{pmatrix}
\displaystyle
$ is given by
\begin{eqnarray} \label{eqSecondDerivative}
\left\langle
\begin{pmatrix}
[D_u \mathbf{S}(\bar{u},\bar{\tau})]^* & I & 0 \\
[D_\tau \mathbf{S}(\bar{u},\bar{\tau})]^* & 0 & 1
\end{pmatrix}
D^2 \mathcal{L}(\bar{\mathbf{y}}, \bar{u}, \bar{\tau})
\begin{pmatrix}
D_u \mathbf{S}(\bar{u},\bar{\tau}) & D_\tau \mathbf{S}(\bar{u},\bar{\tau}) \\
I & 0 \\
0 & 1
\end{pmatrix}
\scriptscriptstyle \begin{pmatrix} v \\ \theta \end{pmatrix} \displaystyle
,
\scriptscriptstyle \begin{pmatrix} \hat{v} \\ \hat{\theta} \end{pmatrix} \displaystyle
\right\rangle_{L^2(Q) \times \R}.
\end{eqnarray}
\end{proposition}

\begin{proof}
Relation~\eqref{eqJisL} follows directly from the definition of $\mathcal{L}$. Differentiating this relation twice and using the fact that $D_{\mathbf{y}} \mathcal{L}(\bar{\mathbf{y}},\bar{u},\bar{\tau})= 0$, we obtain the expressions of the first- and second-order derivatives. 
For example, we have
\begin{eqnarray*}
D_u  J(\bar{u},\bar{\tau}) & = & D_u \mathcal{L}(\bar{\mathbf{y}},\bar{u},\bar{\tau}) v
+ \underbrace{D_{\mathbf{y}} \mathcal{L}(\bar{\mathbf{y}},\bar{u},\bar{\tau})}_{=0} \mathbf{S}(\bar{u},\bar{\tau}) v
= D_u \mathcal{L}(\bar{\mathbf{y}},\bar{u},\bar{\tau}) v,
\end{eqnarray*}
which yields the announced expression.
\end{proof}

\section{Regularity of the Hessian of the cost function} \label{sectionRegularity}

In Proposition~\ref{propositionContinuityL2} below, we show that $D^2J(u,\tau)$ can be extended to a continuous bilinear form on $L^2(Q) \times \R$ and that it is Lipschitz continuous with respect to $u$ (for the $L^\infty$-norm) and $\tau$. The proof of the proposition is based on two other technical lemmas (Lemmas~\ref{lemmaLisC2} and Lemma~\ref{lemmaSisC2}). They state that $\tilde{J}$ and $S$ satisfy the following property: Their second-order derivatives can be extended to a continuous bilinear mapping on an appropriate $L^2$ space and they are Lipschitz-continuous for the $L^\infty$-norm.

All along the section, a pair $(\bar{u},\bar{\tau}) \in \mathcal{U}_{\text{ad}} \times (0,T)$ is fixed.
Let us fix two real numbers $0 <\varepsilon_0< \min(\bar{\tau},T-\bar{\tau})$ and $c>0$ satisfying the following properties:
\begin{itemize}
\item For all $u \in \mathcal{U}_{\text{ad}}$ and $\tau \in (0,T)$, setting $y= S(u,\tau)$, we have
\begin{eqnarray*}
\| u-\bar{u} \|_{L^\infty(Q)} + |\tau-\bar{\tau}| \leq \varepsilon_0 
& \Rightarrow &
\begin{cases}
\begin{array}{l}
\| y \|_{L^\infty(Q)} \leq c, \\
\| D_y F_1 \big(\cdot,y(\cdot),u(\cdot) \big) \|_{L^\infty(Q)} \leq c, \\
\| D_y F_2 \big(\cdot,y(\cdot),u(\cdot) \big) \|_{L^\infty(Q)} \leq c.
\end{array}
\end{cases}
\end{eqnarray*}
\item For all $\tau \in [\bar{\tau}-\varepsilon_0,\bar{\tau}+ \varepsilon_0]$ and $t \in [0,2]$, the inequality
$1/c \leq \dot{\pi}(t,\tau) \leq c$ holds.
\end{itemize}
We define:
\begin{eqnarray*}
\mathcal{W}& = & \left\{ (u,\tau) \in \mathcal{U}_{\text{ad}} \times (0,T) \,|\,
\| u - \bar{u} \|_{L^\infty(Q)} + |\tau-\bar{\tau}| \leq \varepsilon_0 \right\}.
\end{eqnarray*}

\begin{proposition} \label{propositionContinuityL2}
There exist constants $C_{1a}>0$ and $C_{1b}>0$ such that for all $(u,\tau) \in \mathcal{W}$, for all $(v,\theta)$ and $(\hat{v},\hat{\theta}) \in L^\infty(Q) \times \R$,
\begin{eqnarray*}
 \big| D^2 J(u,\tau)\big( (v,\theta),(\hat{v},\hat{\theta}) \big) \big|
& \leq & C_{1a} \| (v,\theta) \|_{L^2(Q) \times \R} \| (\hat{v},\hat{\theta}) \|_{L^2(Q) \times \R}, \\
 \big| \big(D^2 J(u,\tau)-D^2J(\bar{u},\bar{\tau}) \big) \big( (v,\theta),(\hat{v},\hat{\theta}) \big) \big| 
&  \leq & C_{1b} \big( \| u-\bar{u} \|_{L^\infty(Q)} + |\tau-\bar{\tau} |\big) \\
& & \times\| (v,\theta) \|_{L^2(Q) \times \R} \| (\hat{v},\hat{\theta}) \|_{L^2(Q) \times \R}.
\end{eqnarray*}
\end{proposition}

\begin{proof}
The second-order derivative of $J$ is given by:
\begin{eqnarray*}
D^2 J(u,\tau) \big( (v,\theta), (\hat{v},\hat{\theta}) \big)
& = &
D_{y}\tilde{J} \big( S(u,\tau),u,\tau \big) D^2 S(u,\tau) \big( (v,\theta), (\hat{v},\hat{\theta}) \big) \\
& & + D^2 \tilde{J} \big( S(u,\tau), u, \tau \big)
\big( DS(u,\tau)(v,\theta), v, \theta \big)
\big( DS(u,\tau)(\hat{v},\hat{\theta}), \hat{v}, \hat{\theta} \big).
\end{eqnarray*}
By Lemmas~\ref{lemmaLisC2} and~\ref{lemmaSisC2}, the result holds for $C_{1a}= C_{2a} C_{8a} + C_{3a} (1+ C_{7a}^2)$ and
\begin{eqnarray*}
C_{1b} & = & C_{2a} C_{8b} + C_{2b} C_{8a} + 2C_{3a} C_{7b} \sqrt{1 + C_{7a}^2} + C_{3b}(1+ C_{7a}^2),
\end{eqnarray*}
which concludes the proof.
\end{proof}

\begin{lemma} \label{lemmaLisC2}
There exist constants $C_{2a}>0$, $C_{2b}>0$, $C_{3a}>0$, and $C_{3b}>0$ such that for all $(u,\tau) \in \mathcal{W}$, for all $(z,v,\theta)$ and $(\hat{z},\hat{v},\hat{\theta}) \in \mathcal{Y} \times L^\infty(Q) \times \R$,
\begin{eqnarray*}
 \big| D_y\tilde{J}(y,u,\tau)(z,v,\theta) \big|
& \leq & C_{2a} \| (y,v,\theta) \|_{L^2(Q)^2 \times \R}, \\
 \big| \big( D_y\tilde{J}(y,u,\tau)- D\tilde{J}(\bar{y},\bar{u},\bar{\tau})\big) (z,v,\theta) \big| 
&  \leq & C_{2b} \| (z,v,\theta) \|_{L^2(Q)^2 \times \R} \\
& & \hspace*{-3pt} \times \big( \| y-\bar{y} \|_{L^\infty(Q)} + \| u- \bar{u} \|_{L^\infty(Q)} + |\tau-\bar{\tau}| \big) , \\
 \big| D^2\tilde{J}(y,u,\tau)\big( (z,v,\theta),(\hat{z},\hat{v},\hat{\theta}) \big) \big|
& \leq & C_{3a} \| (z,v,\theta) \|_{L^2(Q)^2 \times \R} \| (\hat{z},\hat{v},\hat{\theta}) \|_{L^2(Q)^2 \times \R}, \\
 \big| \big( D^2\tilde{J}(y,u,\tau) - D^2\tilde{J}(\bar{y},\bar{u},\bar{\tau}) \big) \big( (z,v,\theta),(\hat{z},\hat{v},\hat{\theta}) \big) \big| 
&  \leq & C_{3b} \| (z,v,\theta) \|_{L^2(Q)^2 \times \R} \| (\hat{z},\hat{v},\hat{\theta}) \|_{L^2(Q)^2 \times \R}  \\
& & \hspace*{-3pt} \times \big(\|y - \bar{y} \|_{L^\infty(Q)} + \| u- \bar{u} \|_{L^\infty(Q)} + |\tau-\bar{\tau}| \big) ,
\end{eqnarray*}
where $y= S(u,\tau)$.
\end{lemma}

\begin{proof}
By construction, the set $\{ S(u,\tau),u,\tau) \,|\, (u,\tau) \in \mathcal{W} \}$ is bounded in $L^\infty(Q) \times L^\infty(Q) \times (0,T)$, and $\tilde{J}$ is the integral of a Nemytskii operator. The result then follows from Lemma~\ref{lemmaNemytskii} and Lemma~\ref{lemmaNemytskii2}.
\end{proof}

Define the mapping $\tilde{F}$ as follows:
\begin{eqnarray*}
(y,u,\tau) \in \mathcal{Y} \times \mathcal{U}_{\text{ad}} \times (0,T)
\mapsto \tilde{F}(y,u,\tau) = \dot{\pi}(\cdot,\tau)F(\cdot,\pi(\cdot,\tau),y(\cdot),u(\cdot) \big) \in L^\infty(Q).
\end{eqnarray*}

\begin{lemma} \label{lemmaFisC2}
The mapping $\tilde{F}$ is $C^2$. Moreover, there exist constants $C_{4a}$, $C_{4b}$, $C_{5a}$ and $C_{5b} > 0$ such that for all $(u,\tau) \in \mathcal{W}$, for all $(z,v,\theta)$ and $(\hat{z},\hat{v},\hat{\theta}) \in \mathcal{Y} \times L^\infty(Q) \times \R$,
\begin{eqnarray*} \|D\tilde{F}(y,u,\tau)(z,v,\theta)\|_{L^\infty(Q)} & \leq & C_{4a} \| (z,v,\theta) \|_{L^2(Q)^2 \times \R}, \\
\big\| \big( D \tilde{F}(y,u,\tau)- D \tilde{F}(\bar{y},\bar{u},\bar{\tau}) \big) (z,v,\theta) \big\|_{L^\infty(Q)} 
& \leq & C_{4b} \| (z,v,\theta) \|_{L^2(Q)^2 \times \R} \\
& & \times \big( \| y - \bar{y} \|_{L^\infty(Q)} + \| u-\bar{u} \|_{L^\infty(Q)} + |\tau-\bar{\tau}| \big), \\
 \big\| D^2\tilde{F}(y,u,\tau) \big( (z,v,\theta),(\hat{z},\hat{v},\hat{\theta}) \big) \big\|_{L^\infty(Q)}
& \leq & C_{5a} \| (z,v,\theta) \|_{L^2(Q)^2 \times \R} \| (\hat{z},\hat{v},\hat{\theta}) \|_{L^2(Q)^2 \times \R},
\end{eqnarray*}\vspace*{-20pt}
\begin{flushleft}
$\big\| \big( D^2\tilde{F}(y,u,\tau) - D^2 \tilde{F}(\bar{y},\bar{u},\bar{\tau}) \big) \big( (z,v,\theta),(\hat{z},\hat{v},\hat{\theta}) \big) \big \|_{L^\infty(Q)}$
\end{flushleft}
\vspace*{-15pt}
\begin{flushright}
$\leq  C_{5b} \| (z,v,\theta) \|_{L^2(Q)^2 \times \R} \| (\hat{z},\hat{v},\hat{\theta}) \|_{L^2(Q)^2 \times \R} \big( \| y - \bar{y} \| + \| u-\bar{u} \|_{L^\infty(Q)} + |\tau-\bar{\tau}| \big),$
\end{flushright}
where $y= S(u,\tau)$.
\end{lemma}

\begin{proof}
Again, this result follows from Lemma~\ref{lemmaNemytskii} and Lemma~\ref{lemmaNemytskii2}.
\end{proof}

For all $(u,\tau) \in \mathcal{U}_{\text{ad}} \times (0,T)$, we define the operator $\mathcal{Z}(u,\tau): \xi \in L^2(Q) \mapsto z \in \mathcal{Y}$, where $z$ is the solution to:
\begin{eqnarray} \label{eqLinearizedSystemGen}
\begin{cases}
\begin{array}{rcll}
\dot{z}(t)-\dot{\pi}(t,\tau) \big[ \Delta z(t) + D_y F \big(t, \pi(t,\tau), y(t), u(t) \big) z(t) \big]&=& \xi(t)  & \text{ in $Q$}, \\
z(t)&=& 0 & \text{ on $\Sigma$}, \\
z(0)&=& 0 & \text{ in $\Omega$},
\end{array}
\end{cases}
\end{eqnarray}
and where $y= S(u,\tau)$.

\begin{lemma} \label{lemmaKisC2}
There exist constants $C_{6a}>0$ and $C_{6b} > 0$ such that for all $(u,\tau) \in \mathcal{W}$ and for all $\xi \in L^2(Q)$,
\begin{eqnarray*}
 \big\| \mathcal{Z}(u,\tau)\xi \big\|_{L^\infty(\Omega)^2 \times \mathcal{Y}} 
 &\leq& C_{6a} \| \xi \|_{L^2(Q)}, \\
 \big\| \big( \mathcal{Z}(u,\tau)-\mathcal{Z}(\bar{u},\bar{\tau}) \big) \xi \big\|_{L^\infty(\Omega)^2 \times \mathcal{Y}} 
 &\leq& C_{6b} \big(\| y - \bar{y} \|_{L^\infty(Q)} + \|u -\bar{u} \|_{L^\infty(Q)} + |\tau-\bar{\tau}| \big) \| \xi \|_{L^2(Q)},
\end{eqnarray*}
where $y= S(u,\tau)$.
\end{lemma}

\begin{proof}
The first estimate is a direct consequence of Lemma~\ref{lemmaRegularityLinParabolic}.
Next, let $(u,\tau) \in \mathcal{W}$ and $\xi \in L^2(Q)$, define $z= \mathcal{Z}(u,\tau)\xi$, $\bar{z}= \mathcal{Z}(\bar{u},\bar{\tau})\xi$, and set $\delta z= z-\bar{z}$. The function $\delta z$ is solution to the following linear parabolic equation:
\begin{eqnarray*}
\begin{cases}
\begin{array}{rcll}
\delta \dot{z}(t)-\dot{\pi}(t,\bar{\tau}) \big[ \Delta (\delta z)(t) + D_y F \big(t, \pi(t,\bar{\tau}), \bar{y}(t), \bar{u}(t) \big) \delta z(t) \big]&=& \zeta(t)  & \text{ in $Q$}, \\
\delta z&=& 0 & \text{ on $\Sigma$}, \\
\delta z(0)&=& 0 & \text{ in $\Omega$},
\end{array}
\end{cases}
\end{eqnarray*}
where
\begin{eqnarray*}
\zeta(t)&= & \big[ \dot{\pi}(t,\tau)-\dot{\pi}(t,\bar{\tau}) \big] \Delta z(t) \\
& & + \big[ \pi(t,\tau) D_yF \big( t,\pi(t,\tau),y(t),u(t) \big)- \pi(t,\bar{\tau}) D_yF \big( t,\pi(t,\bar{\tau}),\bar{y}(t),\bar{u} \big) \big] z(t).
\end{eqnarray*}
It holds
\begin{eqnarray*}
\| \zeta \|_{L^2(Q)} \leq \varepsilon_0 \| \Delta z \|_{L^2(Q)} + C_{4b} \| z \|_{L^2(Q)}
\leq (\varepsilon_0 + C_{4b}) C_{6a} \| \xi \|_{L^2(Q)}.
\end{eqnarray*}
The second estimate follows, taking $C_{6b}= (\varepsilon_0 + C_{4b}) C_{6a}^2$ and using once again Lemma~\ref{lemmaRegularityLinParabolic}.
\end{proof}

\begin{lemma} \label{lemmaSisC2}
There exist constants $C_{7a}>0$, $C_{7b}>0$, $C_{8a}>0$ and $C_{8b}>0$ such that for all $(u,\tau) \in \mathcal{W}$, for all $(v,\theta)$ and $(\hat{v},\hat{\theta}) \in L^\infty(Q) \times \R$,
\begin{eqnarray*}
 \big\| DS(u,\tau)(v,\theta) \big\|_{\mathcal{Y}} & \leq & C_{7a} \| (v,\theta) \|_{L^2(Q) \times \R}, \\
 \| S(u,\tau)-S(\bar{u},\bar{\tau}) \|_{\mathcal{Y}} & \leq & \sqrt{2|\Omega|} C_{7a} \big( \| u- \bar{u} \|_{L^\infty(Q)} + |\tau- \bar{\tau}| \big), \\
 \big\| \big( DS(u,\tau)-DS(\bar{u},\bar{\tau}) \big) (v,\theta) \big\|_{\mathcal{Y}} 
 & \leq & C_{7b} \big(\| u- \bar{u} \|_{L^\infty(Q)} + |\tau-\bar{\tau}| \big) \| (v,\theta) \|_{L^2(Q) \times \R} \\
 \big\| D^2S(u,\tau)\big((v,\theta),(\hat{v},\hat{\theta}) \big) \big\|_{\mathcal{Y}} 
 & \leq & C_{8a} \| (v,\theta) \|_{L^2(Q) \times \R} \| (\hat{v},\hat{\theta}) \|_{L^2(Q) \times \R}, \\
 \big\| \big( D^2S(u,\tau)-D^2S(\bar{u},\bar{\tau}) \big) \big((v,\theta),(\hat{v},\hat{\theta}) \big) \big\|_{\mathcal{Y}}
&  \leq & C_{8b}
\| (v,\theta) \|_{L^2(Q) \times \R}  \| (\hat{v},\hat{\theta}) \|_{L^2(Q) \times \R} \\
& & \times \big( \| u- \bar{u} \|_{L^\infty(Q)} + |\tau-\bar{\tau}| \big).
\end{eqnarray*}
\end{lemma}

\begin{proof}
The first- and second-order derivatives of $S$ are given by:
\begin{flushleft}
$DS(u,\tau)(v,\theta)= \mathcal{Z}(u,\tau)\big[ D_u \tilde{F}(S(u,\tau),u,\tau) v + D_\tau \tilde{F}(S(u,\tau),u,\tau) \theta \big],$ \\
$D^2S(u,\tau)\big( (v,\theta),(\hat{v},\hat{\theta}) \big) = \mathcal{Z}(u,\tau)
\big[ D^2 \tilde{F}(S(u,\tau),u,\tau)\big( (DS(u,\tau)(v,\theta),v,\theta),(DS(u,\tau)(\hat{v},\hat{\theta}),\hat{v},\hat{\theta}) \big) \big].$
\end{flushleft}
The first estimate follows directly from Lemma~\ref{lemmaFisC2} and Lemma~\ref{lemmaKisC2}, taking $C_{7a}= C_{6a}C_{4a}$. For all $(u,\tau) \in \mathcal{W}$,
\begin{eqnarray*}
S(u,\tau)-S(\bar{u},\bar{\tau}) & = & \int_0^1 DS(\bar{u} + \theta (u-\bar{u}), \bar{\tau} + \theta (\tau-\bar{\tau})) \dd \theta,
\end{eqnarray*}
and therefore,
\begin{eqnarray*}
\| S(u,\tau)- S(\bar{u},\bar{\tau}) \|_{\mathcal{Y}}
\leq C_{7a} \| (u-\bar{u}, \tau-\bar{\tau}) \|_{L^2(Q) \times \R}
\leq \sqrt{2|\Omega|} C_{7a} \big( \| u -\bar{u} \|_{L^\infty(Q)} + |\tau-\bar{\tau}| \big).
\end{eqnarray*}
The second estimate follows. The three other estimates are easily obtained with:
\begin{eqnarray*}
& & C_{7b} =  C_{6a} C_{4b} \sqrt{1+ C_{7a}^2} + C_{6b}C_{4a}, \qquad
C_{8a} =   C_{6a} C_{5a} (1+ C_{7a}^2), \\
& & C_{8b} =  2 C_{6a} C_{5a} C_{7b} \sqrt{1 + C_{7a}^2}
+ C_{6a} C_{5b} (1+ C_{7a}^2) + C_{6b} C_{5a} (1+ C_{7a}^2),
\end{eqnarray*}
and thus the proof is complete.
\end{proof}

\section{Optimality conditions} \label{sectionOptiCond}

\subsection{First-order optimality conditions}

We define:
\begin{eqnarray*}
h \colon (s,t,x,y,u,p) \in (0,2) \times (0,T) \times \R^n \times \R^3
\mapsto
\begin{cases}
\begin{array}{ll}
\ell_1(t,x,y,u)+ p f_1(t,x,y,u,p) & \text{ if $s \in (0,1)$}, \\
\ell_2(t,x,y,u)+ p f_2(t,x,y,u,p) & \text{ if $s \in (1,2)$}.
\end{array}
\end{cases}
\end{eqnarray*}

\begin{theorem} \label{theoremCN1}
Assume that $(\bar{u},\bar{\tau})$ is a solution to~\eqref{problem1}. Then, for all $u \in \mathcal{U}_{\text{ad}}$,
\begin{eqnarray} \label{eqNecCond1a}
D_u J(\bar{u},\bar{\tau})(u-\bar{u}) & \geq & 0.
\end{eqnarray}
Moreover,
\begin{eqnarray} \label{eqNecCond1b}
D_\tau J(\bar{u},\bar{\tau}) & = & 0.
\end{eqnarray}
Therefore, for a.\,e.\@ $(t,x) \in Q$,
\begin{eqnarray} \label{eqPontryaginWeak}
\begin{array} {rcl}
\bar{u}(t,x) \in (u^-(t,x), u^+(t,x)) & \Rightarrow &  D_u h \big( t,\pi(t,\bar{\tau}),x,\bar{y}(t,x),\bar{u}(t,x),\bar{p}(t,x) \big) = 0, \\
\bar{u}(t,x) = u^-(t,x) & \Rightarrow & D_u h \big( t,\pi(t,\bar{\tau}),x,\bar{y}(t,x),\bar{u}(t,x),\bar{p}(t,x) \big) \geq 0, \\
\bar{u}(t,x)= u^+(t,x) & \Rightarrow & D_u h \big( t,\pi(t,\bar{\tau}),x,\bar{y}(t,x),\bar{u}(t,x),\bar{p}(t,x) \big) \leq 0.
\end{array}
\end{eqnarray}
\end{theorem}

\begin{proof}
Inequality~\eqref{eqNecCond1a} follows directly from the convexity of $\mathcal{U}_{\text{ad}}$ and the differentiability of $J$. The equality and the inequalities given in~\eqref{eqPontryaginWeak} follow from the optimality of $\bar{u}$ and a Lebesgue point argument (see for instance~\cite[Theorem 5.12 and Conclusion]{Tro10}).
\end{proof}

\subsection{Second-order optimality conditions}

In this subsection, $(\bar{u},\bar{\tau})$ is assumed to satisfy the first-order optimality conditions~\eqref{eqNecCond1a} and~\eqref{eqNecCond1b}. For the formulation of the second-order order optimality conditions and, more precisely, for the definition of the critical cone, we introduce the following subset of $Q$, parameterized by $\delta \geq 0$:
\begin{eqnarray*}
A_{\delta}(\bar{u},\bar{\tau}) & = & \big\{ (t,x) \in \Omega \,|\, \big|D_uh \big( t,\pi(t,\bar{\tau}),x,\bar{y}(t,x),\bar{u}(t,x),\bar{p}(t,x) \big) \big| > \delta \big\}.
\end{eqnarray*}
Observe that by~\eqref{eqPontryaginWeak}, for all $\delta \geq 0$, for a.\,e.\@ $(t,x) \in Q$,
\begin{eqnarray*}
(t,x) \in A_{\delta}(\bar{u},\bar{\tau}) & \Rightarrow & 
\bar{u}(t,x) \in \{ u^-(t,x), u^+(t,x) \}.
\end{eqnarray*}
We define the tangential set to $\mathcal{U}_{\text{ad}}$ at $\bar{u}$:
\begin{eqnarray*}
T_{\mathcal{U}_{\text{ad}}}(\bar{u}) & = &
\big\{v \in L^\infty(Q) \,|\, 
\bar{u}(t,x)= u^-(t,x) \Rightarrow v(t,x) \geq 0, \\
& & \qquad \qquad \qquad \,
\bar{u}(t,x)= u^+(t,x) \Rightarrow v(t,x) \leq 0, \text{ for a.\,e.\@ $(t,x) \in Q$}
\big\}.
\end{eqnarray*}
We now define, for all $\delta \geq 0$, the set
\begin{eqnarray*}
C_\delta(\bar{u},\bar{\tau}) & = & \big\{ v \in T_{\mathcal{U}_{\text{ad}}}(\bar{u}) \,|\,
(t,x) \in A_\delta(\bar{u},\bar{\tau}) \Rightarrow v(t,x)= 0,
\text{ for a.\,e.\@ $(t,x) \in Q$} \big\}.
\end{eqnarray*}
The set $C_0(\bar{u},\bar{\tau})$ is called critical cone.
Observe that for all $v \in C_0(\bar{u},\bar{\tau})$,
\begin{eqnarray*}
D_uJ(\bar{u},\bar{\tau})v & = & \int_{Q\backslash A_0(\bar{u},\bar{\tau})} \dot{\pi}(t,\bar{\tau}) D_u h \big( t,\pi(\bar{\tau},t),\bar{y}(t,x),\bar{u}(t,x),\bar{p}(t,x) \big) v(t,x) \dd x \dd t = 0.
\end{eqnarray*}
Note also that for all $0 \leq \delta_1 \leq \delta_2$, $A_{\delta_2}(\bar{u},\bar{\tau}) \subseteq A_{\delta_1}(\bar{u},\bar{\tau})$, and thus $C_{\delta_1}(\bar{u},\bar{\tau}) \subseteq C_{\delta_2}(\bar{u},\bar{\tau})$.

\begin{theorem} \label{theoremCN2}
Assume that $(\bar{u},\bar{\tau})$ is a solution to~\eqref{problem1}. Then, for all $(v,\theta) \in C_0(\bar{u},\bar{\tau}) \times \R$,
\begin{eqnarray*}
D^2 J(\bar{u},\bar{\tau})(v,\theta)^2 & \geq & 0.
\end{eqnarray*}
\end{theorem}

\begin{proof}
Let $(v,\theta) \in C_0(\bar{u},\bar{\tau}) \times \R$. Let $(\varepsilon_k)_{k \in \mathbb{N}}$ be a sequence of positive real numbers converging to 0. For all $k \in \mathbb{N}$, define for a.\,e.\@ $(t,x) \in Q$
\begin{eqnarray*}
v_k(t,x) & = & \frac{1}{\varepsilon_k} \Big( P_{[u^-(t,x),u^+(t,x)]} \big( \bar{u}(t,x)+ \varepsilon_k  v(t,x) \big)-\bar{u}(t,x) \Big),
\end{eqnarray*}
where $P_{[a,b]}$ denotes the projection on the interval $[a,b]$. Since $v \in T_{\mathcal{U}_{\text{ad}}}(\bar{u})$, for a.\,e.\@ $(t,x) \in Q$, there exists $\bar{k}$ such that for all $k \geq \bar{k}$, $v_k(t,x)= v(t,x)$. Moreover, for all $k$, for a.\,e.\@ $(t,x) \in Q$, $|v_k(t,x)| \leq |v(t,x)|$, therefore, by the dominated convergence theorem, $\| v-v_k \|_{L^2(Q)} \rightarrow 0$.
It can be easily checked that for all $k$, $v_k \in C_0(\bar{u},\bar{\tau})$ and that $\bar{u}+ \varepsilon_k v_k \in \mathcal{U}_{\text{ad}}$. By optimality of $(\bar{u},\bar{\tau})$ and by Taylor's theorem, there exists for all $k \in \mathbb{N}$ a real number $\mu_k \in [0,1]$ such that:
\begin{eqnarray*}
0 & \leq & \frac{1}{\varepsilon_k^2} \Big( J(\bar{u}+\varepsilon_k v_k,\bar{\tau}+ \varepsilon_k \theta)-J(\bar{u},\bar{\tau}) \Big)
= D^2J(\bar{u} + \mu_k \varepsilon_k v_k, \bar{\tau} + \mu_k \varepsilon_k \theta)(v_k,\theta)^2.
\end{eqnarray*}
Using Proposition~\ref{propositionContinuityL2}, we obtain
\begin{eqnarray*}
0 & \leq & \lim_{k \to \infty} D^2J(\bar{u} + \mu_k \varepsilon_k v_k, \bar{\tau} + \mu_k \varepsilon_k \theta)(v_k,\theta)^2
= D^2J(\bar{u},\bar{\tau})(v,\theta)^2,
\end{eqnarray*}
which concludes the proof.
\end{proof}

\begin{theorem} \label{theoremCS2}
Assume that the first-order necessary conditions~\eqref{eqNecCond1a} and~\eqref{eqNecCond1b} are satisfied. Assume that there exist $\delta >0$ and $\alpha > 0$ such that:
\begin{eqnarray*}
D^2 J(\bar{u},\bar{\tau})(v,\theta)^2  \geq  
\alpha \big( \| v \|_{L^2(Q)}^2 + \theta^2 \big), & &
\forall (v,\theta) \in C_\delta(\bar{u},\bar{\tau}) \times \R.
\end{eqnarray*}
Then, for all $\beta \in (0,\alpha)$, there exists $\varepsilon > 0$ such that for all $(u,\tau) \in \mathcal{U}_{\text{ad}} \times (0,T)$,
\begin{eqnarray*}
\| u - \bar{u} \|_{L^\infty(Q)} + \ |\tau-\bar{\tau}| \leq \varepsilon
& \Rightarrow &
J(u,\tau) \geq J(\bar{u},\bar{\tau}) + \frac{\beta}{2} \big( \| u-\bar{u} \|_{L^2(Q)} + (\tau-\bar{\tau}) \big)^2.
\end{eqnarray*}
\end{theorem}

\begin{proof}
Let $\varepsilon > 0$. For the sake of clarity, $\varepsilon$ is defined at the end of the proof.
Let $(u,\tau) \in \mathcal{U}_{\text{ad}} \times (0,T)$
be such that
\begin{eqnarray*}
\| u-\bar{u} \|_\infty +
|\tau-\bar{\tau}| & \leq & \varepsilon.
\end{eqnarray*}
Define: $v= u-\bar{u}$, $\theta= \tau-\bar{\tau}$.
By~\eqref{eqNecCond1b} and by Taylor's theorem, there exists $\mu \in [0,1]$ such that
\begin{eqnarray} \label{eqCS2low0}
J(u,\tau)-J(\bar{u},\bar{\tau}) & = &
D_uJ(\bar{u},\bar{\tau})v + \frac{1}{2} D^2 J \big(\bar{u} + \mu v,\bar{\tau} + \mu \theta \big)(v,\theta)^2.
\end{eqnarray}
We proceed now as follows: We look for a lower estimate of $D_u J(\bar{u},\bar{\tau})v$ (given by~\eqref{eqCS2low1}) and a lower estimate of $D^2 J \big( \bar{u} + \mu v,\bar{\tau} + \mu \theta \big)(v,\theta)^2$ (given by~\eqref{eqCS2low2}). We begin with the derivation of the lower estimate of $D_u J(\bar{u},\bar{\tau})v$. We split $v$ into two terms, given by
\begin{eqnarray*}
v_0(t,x)=
\begin{cases}
\begin{array}{ll}
v(t,x) & \text{ if $(t,x) \notin A_\delta(\bar{u},\bar{\tau})$}, \\
0 & \text{ if $(t,x) \in A_\delta(\bar{u},\bar{\tau})$,}
\end{array}
\end{cases}
& & 
v_1(t,x)=
\begin{cases}
\begin{array}{ll}
0 & \text{ if $(t,x) \notin A_\delta(\bar{u},\bar{\tau})$}, \\
v(t,x) & \text{ if $(t,x) \in A_\delta(\bar{u},\bar{\tau})$,}
\end{array}
\end{cases}
\end{eqnarray*}
for a.\,e.\@ $(t,x) \in Q$, so that $v= v_0+v_1$. It is easy to check that $v_0 \in C_\delta(\bar{u},\bar{\tau})$. Note also that
\begin{eqnarray*}
\| v \|_{L^2(Q)}^2 & = & \| v_0 \|_{L^2(Q)}^2 + \| v_1 \|_{L^2(Q)}^2.
\end{eqnarray*}
Since $v \in T_{\mathcal{U}_{\text{ad}}}(\bar{u})$, from~\eqref{eqPontryaginWeak} we obtain that
\begin{eqnarray*}
D_u h \big( t,\pi(t,\bar{\tau}),x,\bar{y}(t,x),\bar{u}(t,x),\bar{p}(t,x) \big) v(t,x) \geq 0,
& &  \text{for a.\,e.\@ $(t,x) \in Q$}.
\end{eqnarray*}
It holds: $\bar{u}+ v_0 \in \mathcal{U}_{\text{ad}}$. Therefore, by~\eqref{eqNecCond1a}, and since $v_0 \in C_\delta(\bar{u},\bar{\tau})$,
\begin{eqnarray}
D_uJ(\bar{u},\bar{\tau})v
\geq  D_u J(\bar{u},\bar{\tau})v_1
& = &  \int_{A_\delta(\bar{u},\bar{\tau})}
\dot{\pi}(t,\bar{\tau}) D_u h(t,x,\bar{y}(t,x),\bar{u}(t,x),\bar{p}(t,x))v(t,x) \dd x \dd t \notag \\
& = & \int_{A_\delta(\bar{u},\bar{\tau})}
\dot{\pi}(t,\bar{\tau}) |D_u h(t,x,\bar{y}(t,x),\bar{u}(t,x),\bar{p}(t,x)) | \cdot |v(t,x)| \dd x \dd t \notag \\
& \geq & \kappa \delta \| v_1 \|_{L^1(Q)}
 \geq  \frac{\kappa \delta}{\varepsilon} \| v_1 \|_{L^2(Q)}^2, \label{eqCS2low1}
\end{eqnarray}
where $\kappa= \min(\bar{\tau},T-\bar{\tau})= \displaystyle \min_{t \in [0,T]} \dot{\pi}(t,\bar{\tau}) >0$. The last inequality follows from:
\begin{eqnarray*}
\| v_1 \|_{L^\infty(Q)} \leq \| v \|_{L^\infty(Q)} \leq \varepsilon.
\end{eqnarray*}
We look now for a lower estimate of $D^2 J \big(\bar{u} + \mu v,\bar{\tau} + \mu \theta \big)(v,\theta)^2$. We have
\begin{eqnarray} \label{eqCS2tech1}
D^2 J \big( \bar{u} + \mu v,\bar{\tau} + \mu \theta \big)(v,\theta)^2
& \geq & D^2 J(\bar{u},\bar{\tau})(v,\theta)^2 - C_{1b} \varepsilon (\| v_1 \|_{L^2(Q)}^2 + \| v_0 \|_{L^2(Q)}^2  + \theta^2),
\end{eqnarray}
where $C_{1b}>0$ is the constant given by Proposition~\ref{propositionContinuityL2}.
Moreover, by Proposition~\ref{propositionContinuityL2} and Young's inequality, for all $\gamma > 0$,
\begin{eqnarray}
\big| D^2 J(\bar{\mu},\bar{\tau})(v,\theta)^2 -  D^2 J(\bar{\mu},\bar{\tau})(v_0,\theta)^2 \big|
& = & \big| D^2 J(\bar{\mu},\bar{\tau})\big( (v_1,0),(2v_0 + v_1, 2 \theta) \big) \big| \notag \\
& \leq & C_{1a} \| v_1 \|_{L^2(Q)} \cdot \sqrt{4 \| v_0 \|_{L^2(Q)}^2 + \| v_1 \|_{L^2(Q)^2}^2 + 4 \theta^2} \notag \\
& \leq & \frac{C_{1a}}{\gamma} \| v_1 \|_{L^2(Q)}^2 + C_{1a}\gamma \big( 4 \| v_0 \|_{L^2(Q)}^2 + \| v_1 \|_{L^2(Q)}^2 + 4 \theta^2 \big).  \notag \\ \label{eqCS2tech2}
\end{eqnarray}
By assumption,
\begin{eqnarray} \label{eqCS2tech3}
D^2J(\bar{u},\bar{\tau})(v_0,\theta)^2 & \geq & \alpha \big( \| v_0 \|_{L^2(Q)}^2 + \theta^2 \big).
\end{eqnarray}
Combining~\eqref{eqCS2tech1},~\eqref{eqCS2tech2}, and~\eqref{eqCS2tech3} we obtain, for all $\gamma>0$,
\begin{eqnarray}
 D^2 J \big( \bar{u}+ \mu v, \bar{\tau} + \mu \theta \big)(v,\theta)^2
& \geq &
 \big( \alpha-C_{1b} \varepsilon - 4C_{1a} \gamma )(\| v_0 \|_{L^2(Q)}^2 + \theta^2)
- \Big( C_{1b} \varepsilon + \frac{C_{1a}}{\gamma} \Big) \| v_1 \|_{L^2(Q)}^2. 
\notag \\ \label{eqCS2low2}
\end{eqnarray}
Then, combining the two lower estimates~\eqref{eqCS2low1} and~\eqref{eqCS2low2} with~\eqref{eqCS2low0}, we obtain:
\begin{eqnarray*}
J(u,\tau)-J(\bar{u},\bar{\tau})
& \geq & \frac{1}{2} \big( \alpha-C_{1b} \varepsilon - 4 C_{1a} \gamma \big) \big( \| v_0 \|_{L^2(Q)}^2 + \theta^2 \big) + \left( \frac{\kappa \delta}{\varepsilon}-\frac{C_{1b} \varepsilon}{2} - \frac{C_{1a}}{2\gamma} \right) \| v_1 \|_{L^2(Q)}^2.
\end{eqnarray*}
Choosing
\begin{eqnarray*}
\gamma= \frac{\alpha-\beta}{8 C_{1a}}, & & 
\varepsilon= \min \left( 1, \frac{2\kappa \delta}{\beta+ C_{1b} + C_{1a}/\gamma}, \frac{\alpha-\beta}{2C_{1b}} \right),
\end{eqnarray*}
the result follows.
\end{proof}

\section{Constancy of the Hamiltonian} \label{sectionConstancy}

In this last section, we provide an additional result related to the variation of $t \in [0,2] \mapsto H(t,\pi(t,\bar{\tau}),\bar{y}(t),\bar{u}(t))$, when $(\bar{u},\bar{\tau})$ is a global solution to problem~\eqref{problem1}. When $f_1$, $f_2$, $\ell_1$, and $\ell_2$ are time-independent, the Hamiltonian is constant. This property was already known in the absence of switching times. We show that it still holds in the case of hybrid problems, despite the jump of the costate at $\tau$. Our approach is inspired from~\cite{RZ99} and consists in reformulating the problem, once again, by a change of variables.

Let us define:
\begin{eqnarray*}
\mathcal{T} & = & \Big\{ \nu \in L^\infty((0,2),\R) \,|\, \exists \varepsilon > 0,\ \nu(t) \geq \varepsilon, \text{ for a.\,e.\@ $t \in (0,2)$} \text{ and } \int_0^2 \nu(t) \dd t = T \Big\}, \\
\mathcal{T}_0 & = & \left\{ \nu \in L^\infty((0,2),\R) \,|\, \int_0^2 \nu(t) \dd t = 0 \right\}.
\end{eqnarray*}
Let $\nu_0$ be a fixed element of $\mathcal{T}$. Observe that $\mathcal{T}$ is an open subset in the affine subspace $\nu_0 + \mathcal{T}_0 \subset L^\infty(0,2)$. For all $u \in L^\infty((0,2) \times \Omega)$ and $\nu \in \mathcal{T}$, we denote by $S_2(u,\nu) \in \mathcal{Y}$ the solution to the following system
\begin{eqnarray}
\begin{cases}
\begin{array}{rcll}
\dot{y}(t) - \nu(t) \big[ \Delta y(t) + F \big(t, \int_0^t \nu ,y(t),u(t) \big) \big] & = & 0 & \text{ in $Q$}, \\
y(t) & = & 0 & \text{ on $\Sigma$}, \\
y(0) & = & y_0 & \text{ in $\Omega$}.
\end{array}
\end{cases}
\end{eqnarray}
Given $u \in L^\infty(Q)$, $\nu \in \mathcal{T}$, and $y \in \mathcal{Y}$, we define:
\begin{eqnarray*}
\tilde{J}_2(u,y,\nu)& = & \int_0^2 \int_{\Omega} \nu(t) L \big(t, \intSmall_0^t \nu ,y(t),u(t) \big) \dd t
+ \int_{\Omega} \big[ \Phi_1 \big( \intSmall_0^1 \nu,y(1) \big) + \Phi_2\big( y(2)\big) \big] \dd x, \\
J_2(u,\nu)&= & \tilde{J}_2(u,S_2(u,\nu), \nu).
\end{eqnarray*}
Finally, consider the following problem:
\begin{equation*} \label{problem2} \tag{$\mathcal{P}_2$}
\min_{u \in \mathcal{U}_{\text{ad}},\, \nu \in \mathcal{T}} \
J_2(u,\nu)
\end{equation*}

In order to prove the equivalence of problems~\eqref{problem0} and~\eqref{problem2}, we introduce the following mappings:
\begin{eqnarray*}
&\chi_{2a} \colon (v(\cdot),\tau) \in \mathcal{V}_{\text{ad}} \mapsto \big( v(\pi(\cdot,\tau)),\dot{\pi}(\cdot,\tau) \big) \in \mathcal{U}_{\text{ad}} \times \mathcal{T}, &\\
&\chi_{2b} \colon  (u(\cdot),\nu) \in \mathcal{U}_{\text{ad}} \times \mathcal{T} \mapsto
\big( u(\theta(\cdot)), \intSmall_0^1 \nu(t) \dd t \big) \in \mathcal{V}_{\text{ad}}, 
&  \text{where: } \int_0^{\theta(t)} \nu(s) \dd s = t, \quad \forall t \in [0,T].
\end{eqnarray*}

\begin{lemma} \label{lemmaSecondEquivalence}
For all $(v,\tau) \in \mathcal{V}_{\text{ad}}$, $J_2(\chi_{2a}(v,\tau))= J_0(v,\tau)$. Conversely, for all $(u,\tau) \in \mathcal{U}_{\text{ad}} \times \mathcal{T}$, $J_0(\chi_{2b}(u,\tau))= J_2(u,\tau)$.

Therefore, if $(v,\tau) \in \mathcal{V}_{\text{ad}}$ is a global solution to~\eqref{problem0}, then $\chi_{2a}(v,\tau)$ is a global solution to~\eqref{problem2}, and conversly, if $(u,\nu) \in \mathcal{U}_{\text{ad}} \times \mathcal{T}$ is a global solution to~\eqref{problem2}, then $\chi_{2b}(u,\tau)$ is a global solution to~\eqref{problem0}.
\end{lemma}

The proof is elementary, and so left to the reader. Observe that $\chi_{2a}$ is not surjective and that $\chi_{2b}$ is not injective.

\begin{lemma}
The mapping $S_2:(u,\nu) \in L^\infty(Q) \times \mathcal{T} \mapsto S_2(u,\nu) \in \mathcal{Y}$ is continuously  Fr\'echet differentiable.
\end{lemma}

\begin{proof}
The proof is very similar to the one of Lemma~\ref{lemmaRegS}. The mapping $e$ is now defined as follows:
\begin{eqnarray*}
e \colon (y,u,\nu) \in \tilde{\mathcal{Y}} \times L^\infty(Q) \times \mathcal{\tau}
& \mapsto & \dot{y}(\cdot) - \nu(\cdot)\big[ \Delta y(\cdot) + F \big( N(\cdot), y(\cdot), u(\cdot) \big) \big]
\in L^2(Q),
\end{eqnarray*}
where $N(t)= \int_0^t \nu(s) \dd s$.
It can be represented as in~\eqref{eqRepresentationE}, if one replaces $\tau$ by $\nu$, and $G_3$, $G_4$ by the following definitions:
\begin{eqnarray*}
G_3 \colon \nu \in \mathcal{T} \mapsto \nu \in L^\infty(0,2), & & 
G_4 \colon \nu \in \mathcal{T} \mapsto \big( t \mapsto \intSmall_0^t \nu(s) \dd s \big) \in L^\infty(0,2).
\end{eqnarray*}
The surjectivity of $D_ye(y,u,\nu)$ is still a consequence of Lemma~\ref{lemmaRegularityLinParabolic}. Thus, $S_2$ is continuously Fr\'echet differentiable, by the implicit function theorem.
\end{proof}

Let $(\bar{v},\bar{\tau})$ be a global solution to~\eqref{problem0}. Let $(\bar{u},\bar{\nu})= \chi_{2a}(\bar{u},\bar{\tau})$. By Lemma~\ref{lemmaSecondEquivalence}, $(\bar{u},\bar{\nu})$ is a solution to~\eqref{problem2} and by Lemma~\ref{lemmaFirstEquivalence}, $(\bar{u},\bar{\tau})$ is a solution to~\eqref{problem1}. Let $\bar{p}$ be the associated costate.
We define a new Lagrangian as follows:
\begin{eqnarray*}
\begin{array} {rrcl}
\mathcal{L}_2 : & (L^\infty(\Omega)^2 \times \mathcal{Y}) \times L^\infty(Q) \times \mathcal{T} & \rightarrow & \R \\
 & \big( \mathbf{y}= (a_1,a_2,y), u, \nu \big) & \mapsto &
\displaystyle \int_{\Omega} \Phi_1 \big( \intSmall_0^1 \nu, a_1 \big) \dd x
+ \int_{\Omega} \Phi_2 \big(a_2 \big) \dd x \\
& & &  \displaystyle+ \int_0^2 \nu(t) H \big(t, \intSmall_0^t \nu,y(t),u(t),\bar{p}(t) \big) \dd t
\\
& & & \displaystyle - \phantom{\Big|} \langle \bar{p}, \dot{y}(\cdot) \rangle_{L^2(Q)} - \langle \bar{p}(0), y(0)-y_0 \rangle_{L^2(\Omega)} \\
& & &  + \phantom{\Big|} \langle \bar{p}(2), y(2)-a_2 \rangle_{L^2(\Omega)}
- \langle [\bar{p}](1), y(1)-a_1 \rangle_{L^2(\Omega)}.
\end{array}
\end{eqnarray*}
It is easy to prove that the mapping $\mathcal{L}$ is Fr\'echet differentiable, by adapting the proof of Lemma~\ref{lemmaTwiceDiffLag}. We define:
\begin{eqnarray*}
\xi(t) & = & \int_0^t \dot{\pi}(s,\bar{\tau}) D_t H \big(s,\pi(s,\bar{\tau}),\bar{y}(s),\bar{u}(s),\bar{p}(s) \big) \dd s.
\end{eqnarray*}
Recall that $D_t H$ denotes the partial derivative of $H$ with respect to its second variable.

\begin{theorem} \label{theoremConstancy}
The function $t \in [0,T] \mapsto H \big(t,\pi(t,\bar{\tau}),y(t),u(t) \big)-\xi(t)$ is constant.
\end{theorem}

\begin{proof}
For all $(u,\nu) \in \mathcal{U}_{\text{ad}} \times \mathcal{T}$,
\begin{eqnarray} \label{eqConstancy1}
J_2(u,\nu) & = & \mathcal{L}_2(\mathbf{S}(u,\nu),u,\nu).
\end{eqnarray}
Moreover, for all $\mathbf{z} \in L^\infty(\Omega)^2 \times \mathcal{Y}$,
\begin{eqnarray} \label{eqConstancy2}
D_{\mathbf{y}} \mathcal{L}_2(\bar{\mathbf{y}},\bar{u},\bar{\nu})\mathbf{z} & = & 0,
\end{eqnarray}
where $(\bar{u}, \bar{\nu})= (\bar{u},\dot{\pi}(\cdot,\bar{\pi})= \chi_{2a}(\bar{u},\bar{\tau})$.
To prove that $H[\cdot]-\xi(\cdot)$ is constant, it is sufficient to prove that for all $\omega \in L^\infty(0,2)$,
\begin{eqnarray*}
\int_0^2 (H[t]-\xi(t)-\bar{H}) \, \omega(t) \dd t & = & 0,
\end{eqnarray*}
where $\bar{H}:= \displaystyle \frac{1}{2} \int_0^2 H[t]-\xi(t) \dd t$. Let $\omega \in L^\infty(0,2)$, and set $\bar{\omega}:= \displaystyle\frac{1}{2} \int_0^2 \omega (t) \dd t$. Let $\tilde{\omega}:= \displaystyle \omega- \bar{\omega}$ and $\theta(t):= \displaystyle\int_0^t \tilde{\omega}(s) \dd s$. By construction, $\tilde{\omega} \in \mathcal{T}_0$ and $0= \theta(0)= \theta(2)$.
Since $\mathcal{T}$ is open and $(\bar{u},\bar{\nu})$ is optimal, we deduce from~\eqref{eqConstancy1} and~\eqref{eqConstancy2} that:
\begin{eqnarray*}
0 & = & D_\nu J_2(\bar{u},\bar{\nu}) \tilde{\omega}
= D_{\nu} \mathcal{L}_2(\bar{\mathbf{y}},\bar{u},\bar{\nu}) \tilde{\omega}
= \int_0^2 H[t] \tilde{\omega}(t) + \dot{\pi}(t,\bar{\tau})D_tH[t] \theta(t) \dd t.
\end{eqnarray*}
Integrating by parts and using the fact that $0= \theta(0)= \theta(2)$, we obtain that
\begin{eqnarray} \label{eqConstancy3}
0 & = & \int_0^2 (H[t]-\xi(t))\, \tilde{\omega}(t) \dd t.
\end{eqnarray}
Observing that
\begin{eqnarray*}
\int_0^2 (H[t]-\xi(t)-\bar{H}) \, \bar{\omega} \dd t = 0 & \text{and} &
\int_0^2 \bar{H} \tilde{\omega}(t) \dd t = 0,
\end{eqnarray*}
we deduce from~\eqref{eqConstancy3} that
\begin{eqnarray*}
 \int_0^2 (H[t]-\xi(t)-\bar{H})\omega(t) \dd t 
& = & \int_0^2 (H[t]-\xi(t)- \bar{H}) \bar{\omega}(t) \dd t \\
& & + \int_0^2 (H[t]-\xi(t)) \tilde{\omega}(t) \dd t
+ \int_0^2 \bar{H} \tilde{\omega}(t) \dd t 
 =   0,
\end{eqnarray*}
which concludes the proof.
\end{proof}

\begin{corollary}
If $f_1$, $f_2$, $\ell_1$ and $\ell_2$ do not depend on $t$, then the Hamiltonian is constant in time.
\end{corollary}

\begin{proof}
In this case, $D_tH\big(s, \pi(s,\bar{\tau}), \bar{y}(s),\bar{u}(s),\bar{p}(s) \big)= 0$ and therefore, $\xi= 0$.
\end{proof}

\section{Numerical experiment: The competitive prey-predator system} \label{best-sec}

As an illustration of the theoretical findings, we propose to consider the Lotka-Volterra PDE-system, describing the evolution of populations of preys and predators. Denote by $\Omega = (0,1)$, and let $y_1$, $y_2$ be the distributions of populations of preys and predators, respectively. The coupled partial differential equations system we are interested in is given by:
\begin{eqnarray}
\left\{ \begin{array} {rcl}
\dot{y}_1 - \nu_1\Delta y_1 = y_1(1-c_1y_1)\left(a-by_2 + u_1\mathds{1}_{\hat{\omega}_1}\right)
& &\text{in } (0,T) \times \Omega, \\
\dot{y}_2 - \nu_2\Delta y_2 = y_2(1-c_2y_2)\left(qy_1-r + u_2\mathds{1}_{\hat{\omega}_2}\right)
& &\text{in } (0,T) \times \Omega, \\
y_1 = y_2 = 0 & & \text{on } (0,T) \times \partial \Omega, \\
y_1(0,\cdot) = y_{1,0}, \quad y_2(0,\cdot) = y_{2,0} & & \text{in } \Omega.
\end{array} \right. \label{sysLotka}
\end{eqnarray}
Note that the control and the state are vector-valued. As mentioned in the introduction, the results of the article can be easily extended to such a situation.
In this system, $a$, $r$, $c_1$, $c_2$, $\nu_1$, and $\nu_2$ are given positive constants. The variables $q$ and $b$ are positive constants on $(0,\tau) \times \Omega$ and $(\tau,T) \times \Omega$. Their value changes at time $\tau$.
The initial data $y_{1,0}$ and $y_{2,0}$ that we have chosen are smooth and satisfy the homogeneous Dirichlet boundary condition on $\partial \Omega$. They are given by
\begin{eqnarray*}
	y_{1,0}(x) & = & 10.0*(1.0 - \exp(-(1.0-x)))*(\exp(-(1.0-x)) - \exp(-1.0)), \\
	y_{2,0}(x) & = & 20.0*(1.0 - \exp(-x))*(\exp(-x) - \exp(-1.0)).
\end{eqnarray*}
The control is realized through two functions $u_1$ and $u_2$, whose supports are given by $\hat{\omega}_1 = (0.00,0.25)$ and $\hat{\omega}_2 = (0.75,1.00)$, respectively. The objective we define is to maximize a quantity $\phi_1$ representing the population of preys on a subdomain $\omega := (0.48, 0.52)$ at some optimal time $\tau$ to be determined. The terminal cost $\phi_2$ is the same function as $\phi_1$. More precisely, by denoting $y = (y_1,y_2)^T$ we consider
\begin{eqnarray*}
	\phi_1(\tau,x,y(\tau,x)) =
	\frac{1}{2} |y_1(\tau,x) |^2\mathds{1}_{\omega}(x), 
	& & 
	 \phi_2(x,y(T,x)) = \frac{1}{2} |y_1(T,x) |^2\mathds{1}_{\omega}(x).
\end{eqnarray*}
Denoting by $u = (u_1,u_2)^T$, the running cost functional (to be maximized) is defined as
\begin{eqnarray*}
	\ell(y(t,x),u(t,x)) & = & -\frac{\alpha}{2} \left( u_1(t,x)^2 + u_2(t,x)^2 \right).
\end{eqnarray*}
The parameters of the model are set to
\begin{eqnarray*}
	T = 30.0, \quad \nu_1 = \nu_2 = 0.001, \quad a=0.3, \quad r = 0.2, \quad c_1 = c_2 = 0.05.
\end{eqnarray*}
The change of dynamics is realized through the coefficients $q$ and $b$, which can be interpreted respectively as the capacity of reproduction of predators and their tendency of killing preys: Their values are given by
\begin{eqnarray*}
	q = b & = & \left\{ \begin{array} {lcl}
		0.1 & & \text{ if } t \leq \tau, \\ 0.07 & & \text{ if } t > \tau.
	\end{array} \right. 
\end{eqnarray*}
It means that after having reached a maximum for the population of preys in the region $\omega$ at time $\tau$, through this change of dynamics we facilitate the subsistence of preys at time $T$ by restricting their interaction with predators. The large number of preys after $\tau$ can indeed help the predators too much to reproduce, and then could lead to the extinction of preys before time $T$.\\

The space discretization is made by P1-finite elements, with $1001$ degrees of freedom. The time-evolution discretization consists of $1000$ time steps, and is performed with a Crank-Nicholson scheme for the state system~\eqref{sysLotka}, and an implicit Euler scheme for the corresponding adjoint system. In practice, in order to make the transport effect easier in space, we add in the left-hand-side of the two first equations of system~\eqref{sysLotka} advection terms of type $\beta (y_1 \cdot \nabla)y_1$ and $\beta (y_2 \cdot \nabla)y_2$, with $\beta = 0.0005$. At each time step, nonlinearities are solved with a Newton method.\\
The results, for the cost parameter $\alpha = 10^{-6}$, are presented in figure~\ref{superfigure}. A Barzilai-Borwein method (see~\cite{BB}) initialized with a line-search step is performed in order to vanish the gradient of the cost function, given by Proposition~\ref{le:derredcost}. The optimize-discretize approach actually introduces a residual error (approximately $10^{-5}\%$ here), which is observed to be proportional to the space step. That is why we stop the algorithm when the variation of the norm of the gradient (between two iterations) as smaller than $10^{-10}\%$. We present only the evolution of the population of preys (the variable $y_1$) and the evolution of their control (the control $u_1$). The optimal control $u_2$ is actually not activated here.

An optimal switching time equal to $\tau \approx 13.5919$ has been obtained. We observe that the control starts by killing preys, likely in order to restrain the population of predators as much as possible. Then, once the population of predators is low enough, the control increases and becomes positive, so that this introduction of preys in the domain leads more efficiently to their reproduction. Secondly, we notice that the control is much more active (and negative) after time $\tau$, in order to respond to the terminal cost. This is probably due to the fact that after the first maximum reached at time $\tau$, the population of preys is much more important than at the initial time, and then a more -- negative -- active control is needed in order to avoid that the population of predators becomes too important. We also observe that the advection terms added in the model help the cost function to increase by approximately $24\%$, with respect to the simulation without control, instead of by approximately $3\%$ with $\beta =0$ (this case not presented here). Last, we observe that a delay is encoded into the model: The control is no more active shortly before the optimal time $\tau$, or before the terminal time. This is due to the time needed to transport preys into the domain from one point to another.

\section{Conclusion}

We have provided first- and second-order optimality conditions for a hybrid control problem for a system governed by a semilinear parabolic equation. The analysis is based on a reformulation of the problem obtained by a change of variable in time. We have also shown that the framework of strong variational solutions (for the state equation) is an appropriate framework for the derivation of optimality conditions. A theoretical property, related to the variation of the Hamiltonian along an optimal trajectory, has been proved. Finally, we have provided numerical results for a hybrid Lotka-Volterra PDE-system.

The results of this article could be extended to different models of PDEs, for example to PDEs of hyperbolic type.
It would also be of interest to allow for the presence of state constraints.
A second-order sensitivity analysis of the value problem of our problem, with the help of a conveniently defined Riccati equation, is another topic for future work.

\begin{minipage}{\linewidth}
	\hspace{-0.035\linewidth}
	\begin{minipage}{0.20\linewidth}
		\begin{figure}[H]
			\includegraphics[trim = 1cm 0cm 2cm 1cm, clip, scale=0.12]{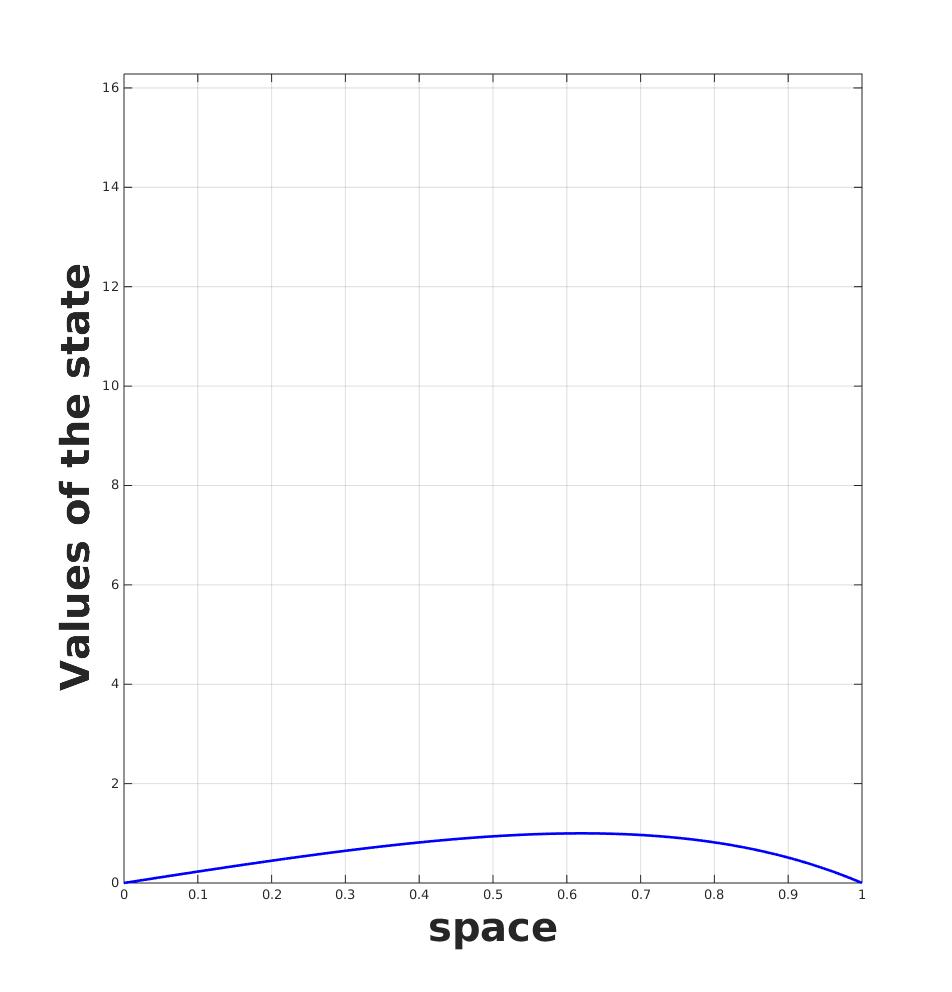} \\
			\includegraphics[trim = 1cm 0cm 2cm 1cm, clip, scale=0.12]{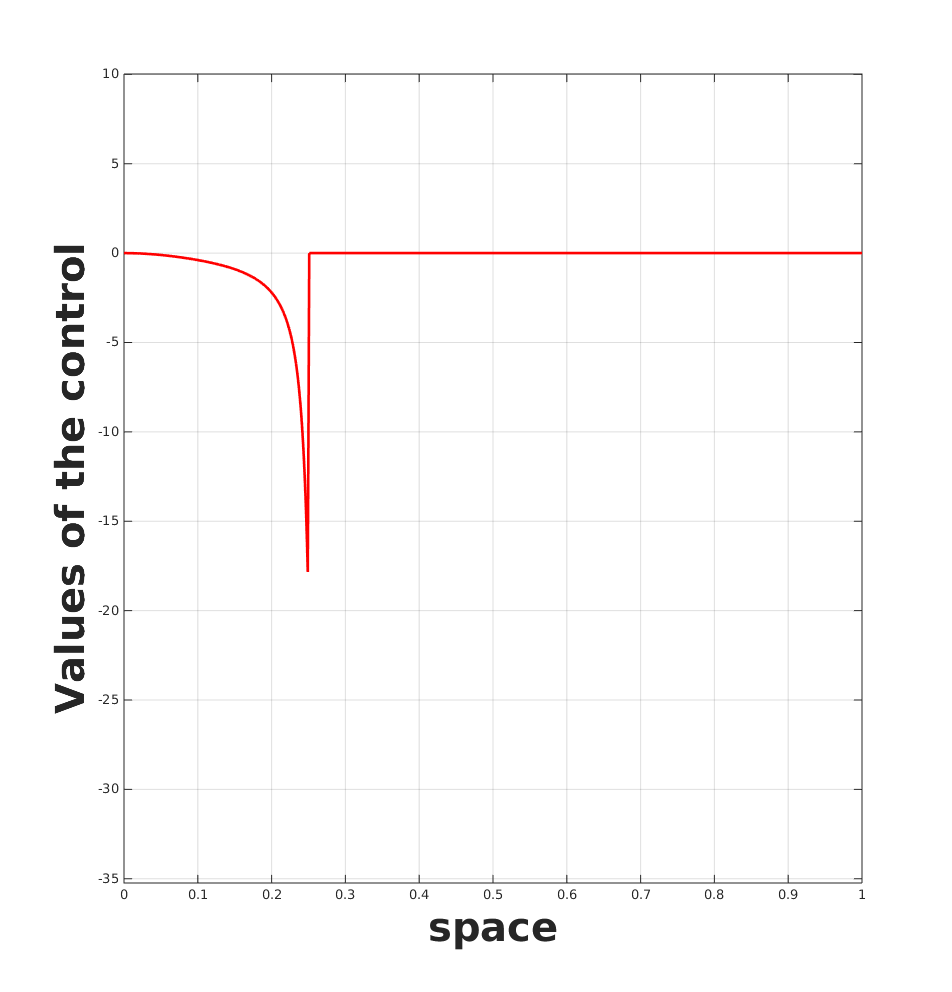}
			\begin{center}\begin{small} $ t = 0.0272 $ \end{small}\end{center}
		\end{figure}
	\end{minipage}
	\hspace{-0.01\linewidth}
	\begin{minipage}{0.20\linewidth}
		\begin{figure}[H]
			\includegraphics[trim = 1cm 0cm 2cm 1cm, clip, scale=0.12]{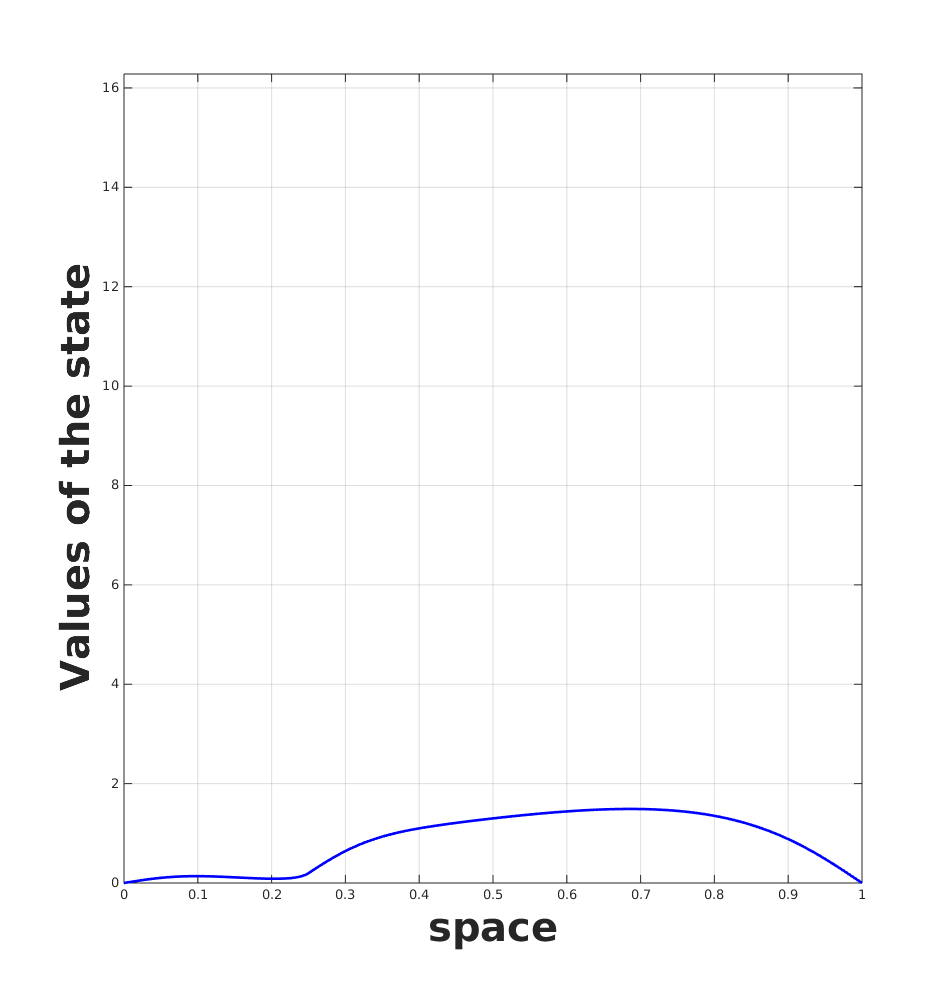} \\
			\includegraphics[trim = 1cm 0cm 2cm 1cm, clip, scale=0.12]{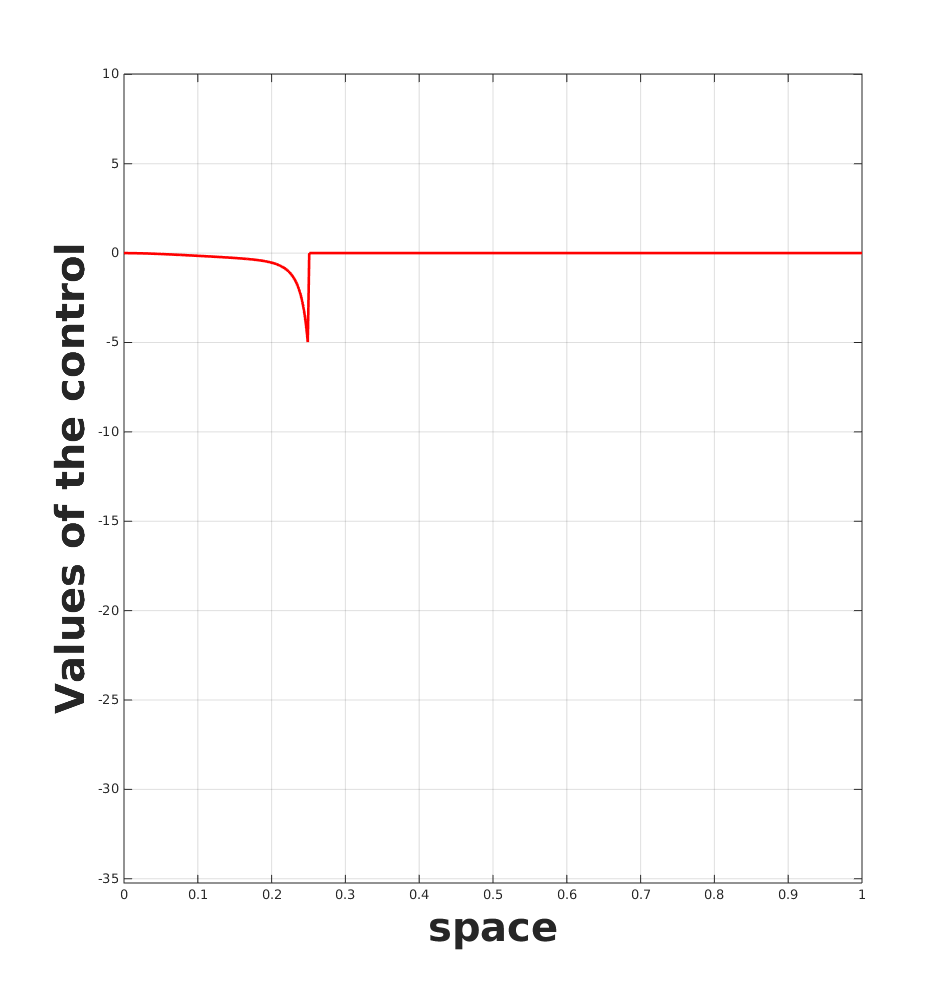} 
			\begin{center}\begin{small} $ t = 2.7184 $ \end{small}\end{center}
		\end{figure}
	\end{minipage}
	\hspace{-0.01\linewidth}
	\begin{minipage}{0.20\linewidth}
		\begin{figure}[H]
			\includegraphics[trim = 1cm 0cm 2cm 1cm, clip, scale=0.12]{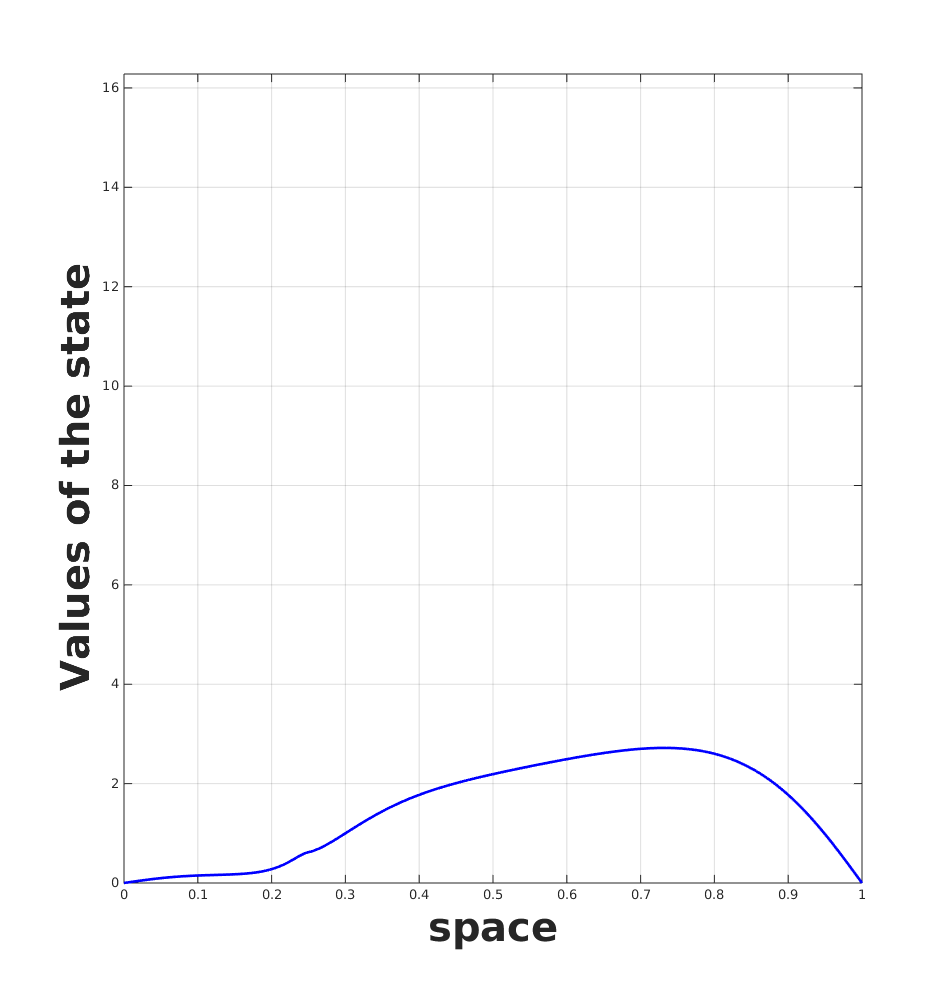} \\
			\includegraphics[trim = 1cm 0cm 2cm 1cm, clip, scale=0.12]{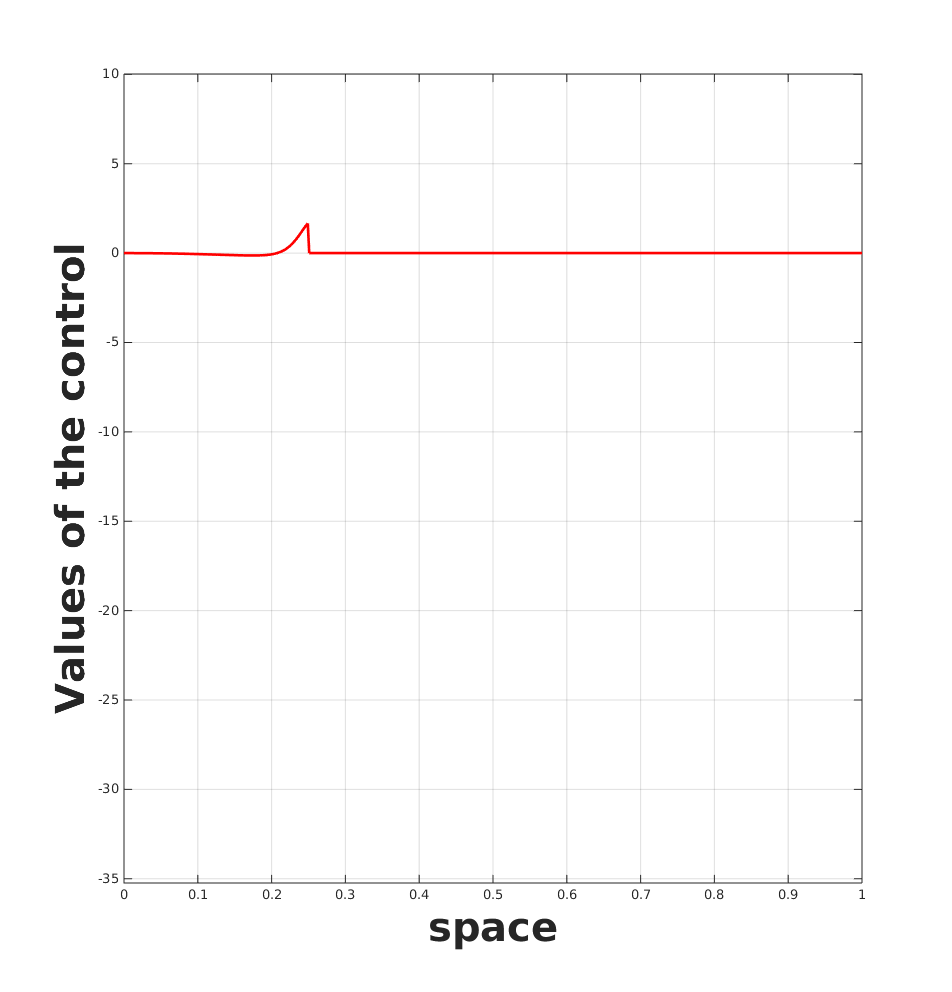} 
			\begin{center}\begin{small} $ t = 6.2523 $ \end{small}\end{center}
		\end{figure}
	\end{minipage}
	\hspace{-0.01\linewidth}
	\begin{minipage}{0.20\linewidth}
		\begin{figure}[H]
			\includegraphics[trim = 1cm 0cm 2cm 1cm, clip, scale=0.12]{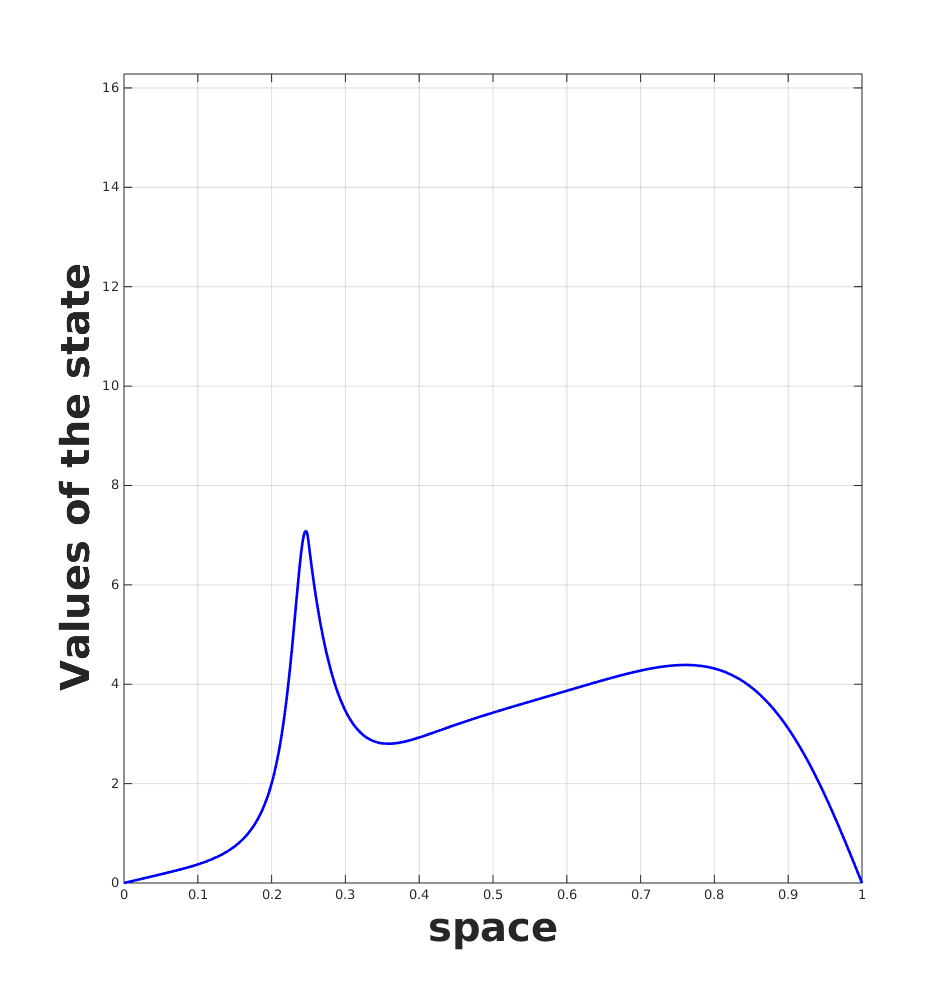} \\
			\includegraphics[trim = 1cm 0cm 2cm 1cm, clip, scale=0.12]{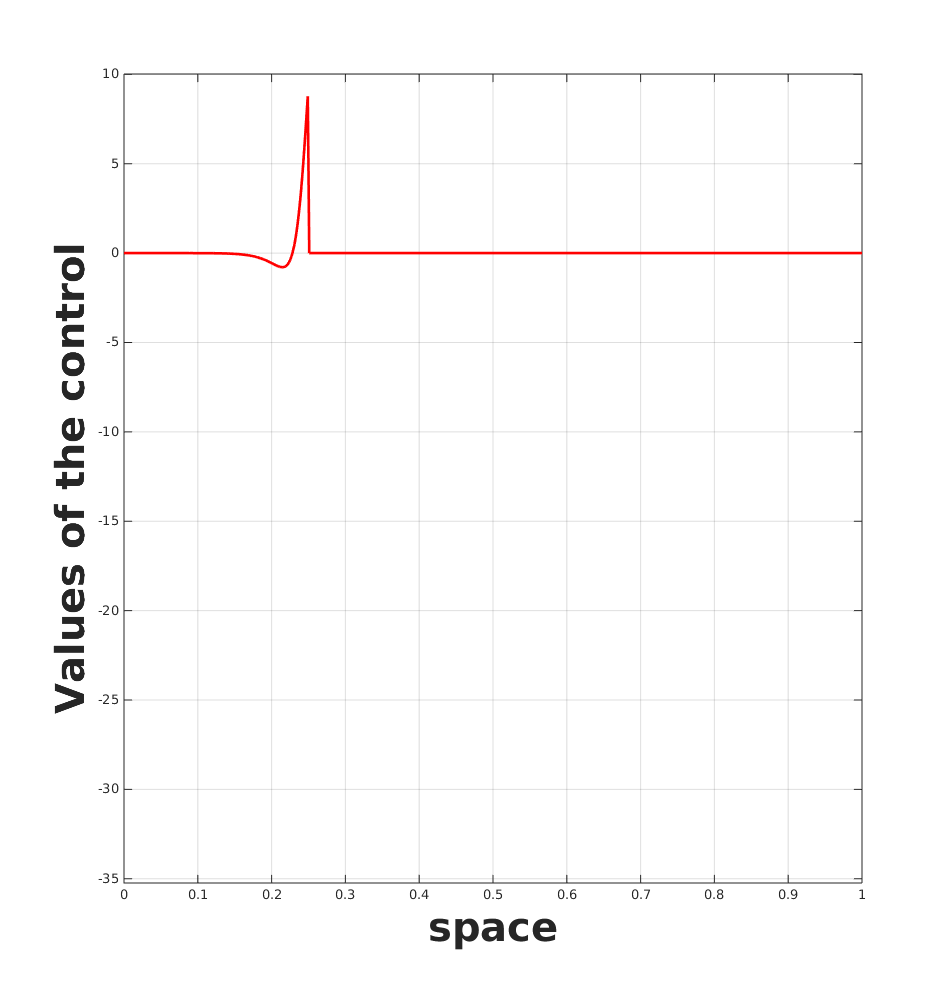} 
			\begin{center}\begin{small} $ t = 9.5144 $ \end{small}\end{center}
		\end{figure}
	\end{minipage}	
	\hspace{-0.01\linewidth}
	\begin{minipage}{0.20\linewidth}
		\begin{figure}[H]
			\includegraphics[trim = 1cm 0cm 2cm 1cm, clip, scale=0.12]{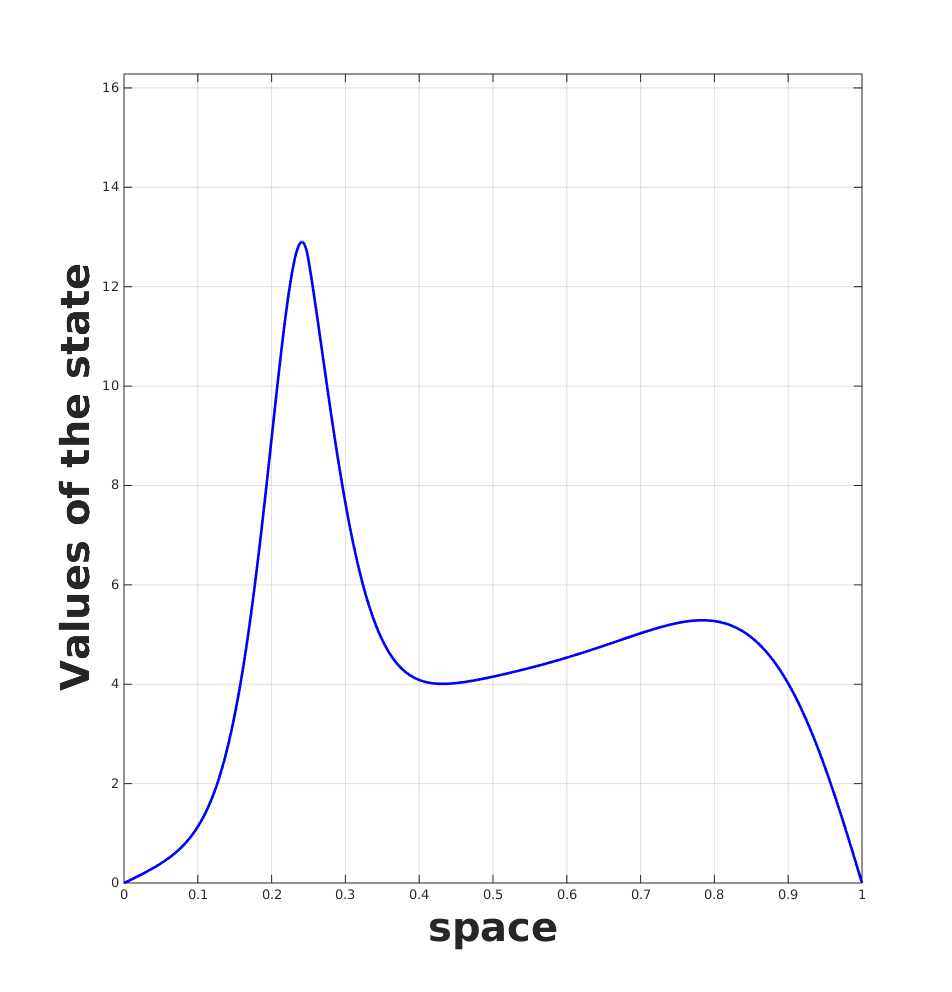} \\
			\includegraphics[trim = 1cm 0cm 2cm 1cm, clip, scale=0.12]{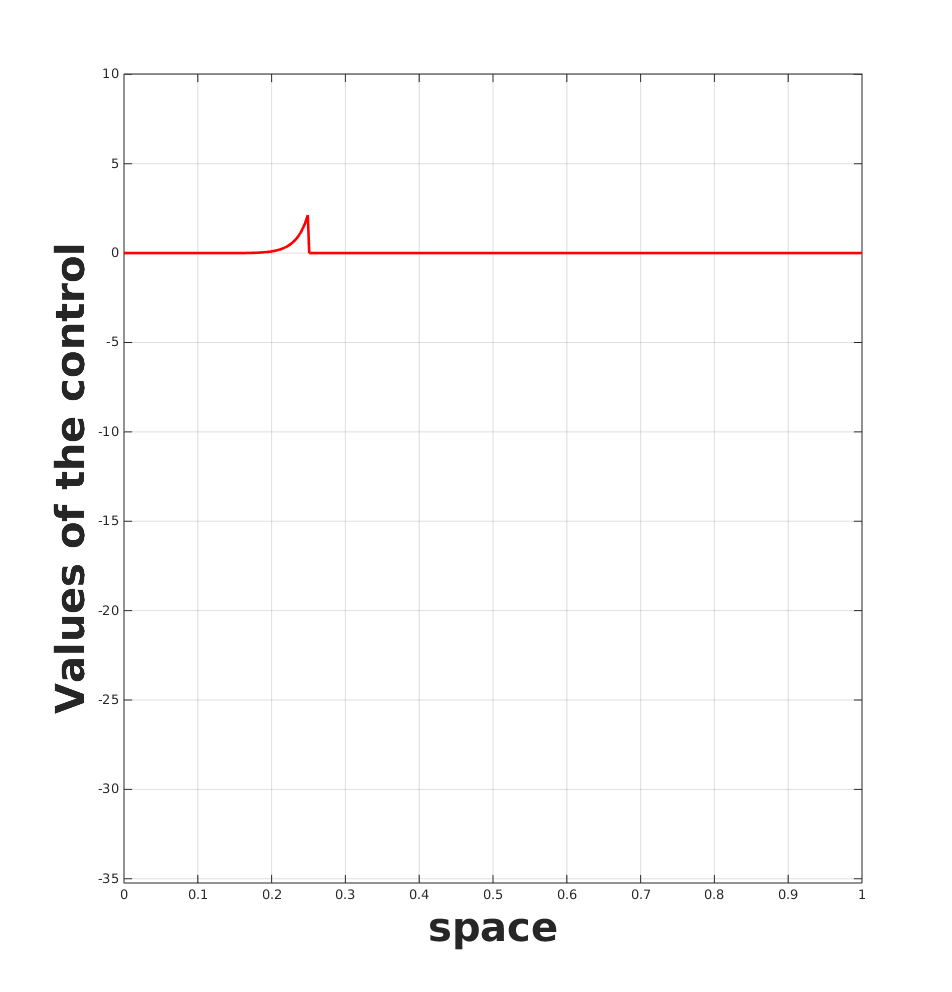} 
			\begin{center}\begin{small} $ t = 11.4172 $ \end{small}\end{center}
		\end{figure}
	\end{minipage}	
	\\ \vspace{-10pt}
	\hspace{-0.04\linewidth}
	\begin{minipage}{0.20\linewidth}
		\begin{figure}[H]
			\includegraphics[trim = 1cm 0cm 2cm 1cm, clip, scale=0.12]{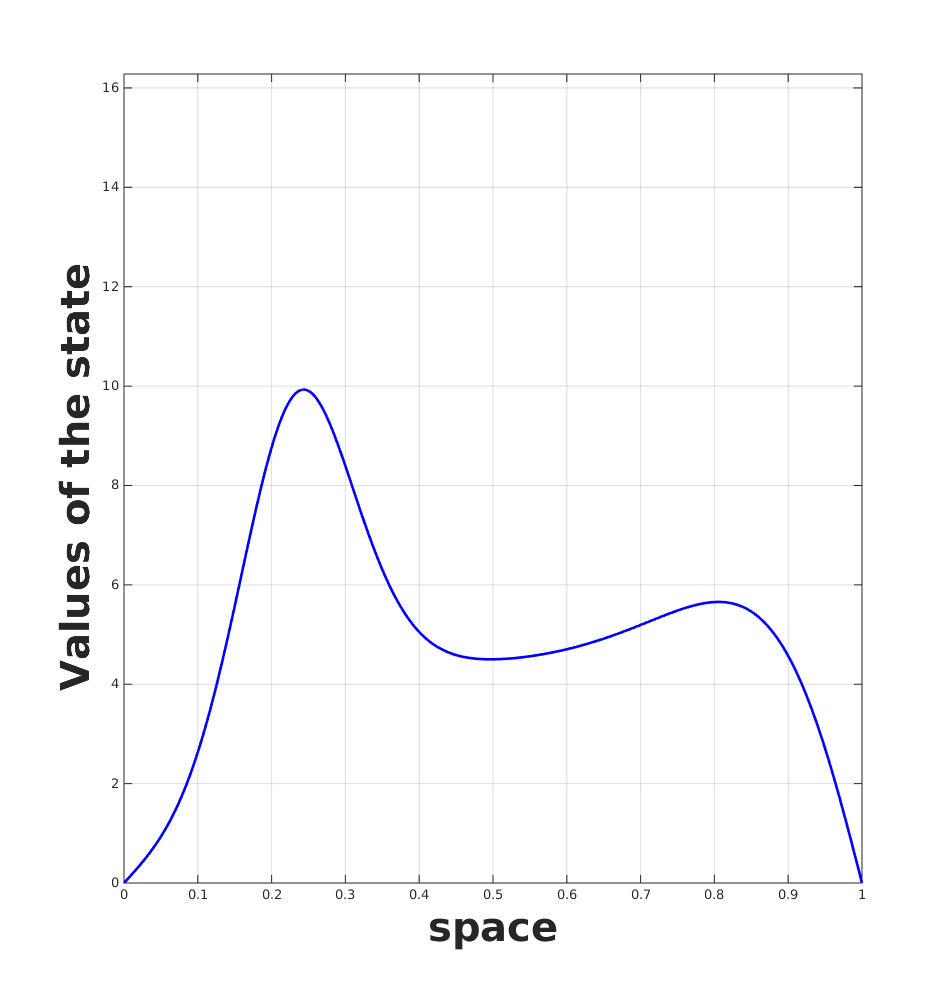} \\
			\includegraphics[trim = 1cm 0cm 2cm 1cm, clip, scale=0.12]{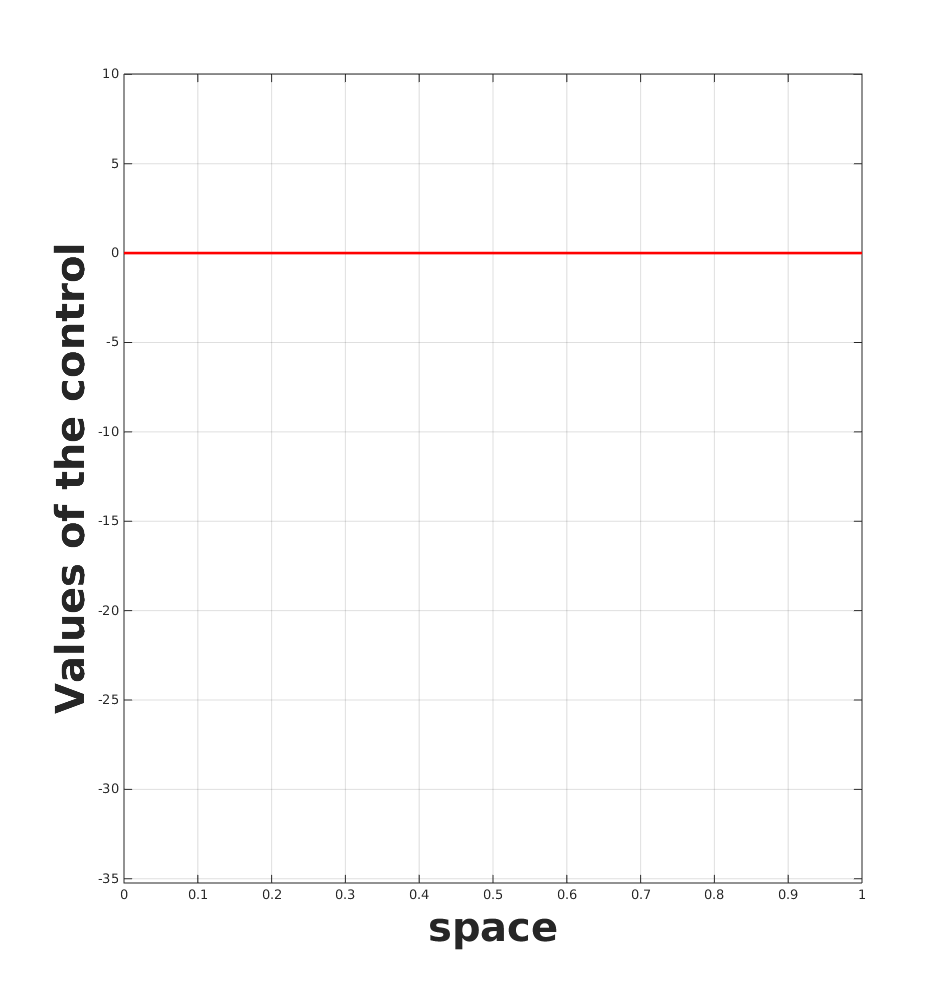}
			\begin{center}\begin{small} $ t = 12.7764 $ \end{small}\end{center}
		\end{figure}
	\end{minipage}
	\hspace{-0.01\linewidth}
	\begin{minipage}{0.20\linewidth}
		\begin{figure}[H]
			\includegraphics[trim = 1cm 0cm 2cm 1cm, clip, scale=0.12]{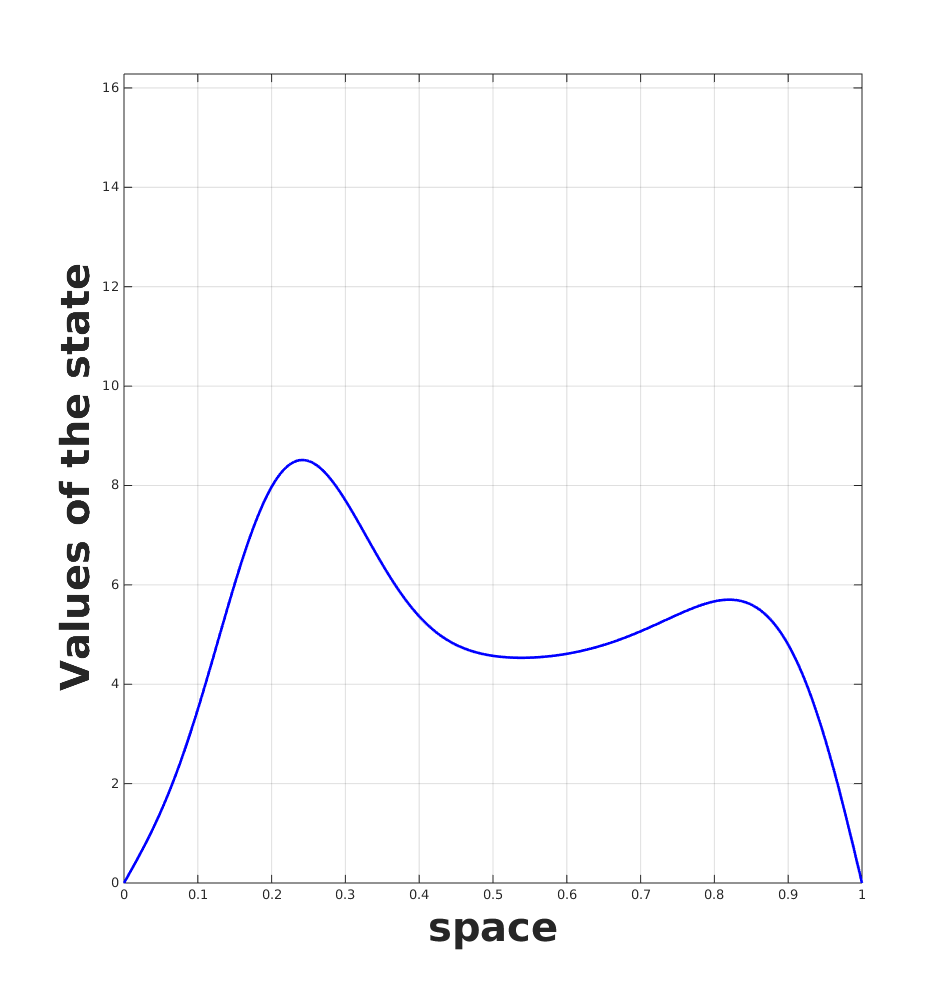} \\
			\includegraphics[trim = 1cm 0cm 2cm 1cm, clip, scale=0.12]{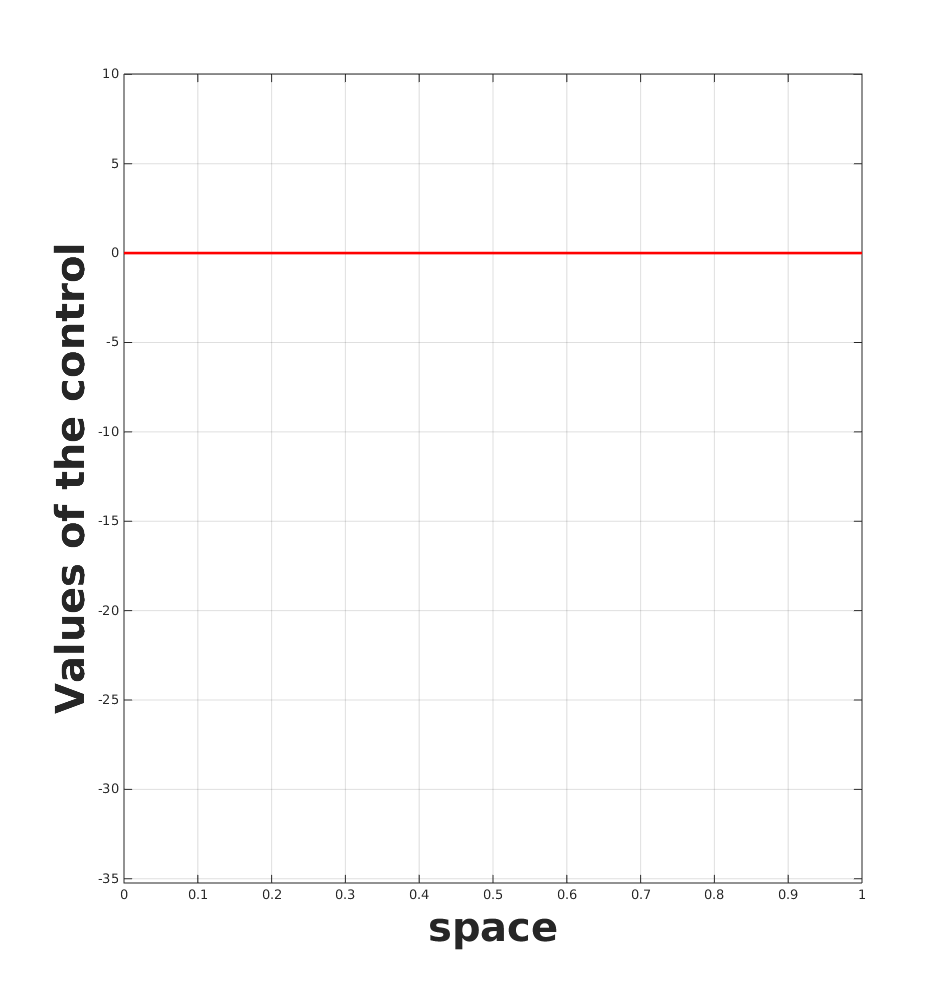} 
			\begin{center}\begin{small} $ t = 13.5919 \approx \tau $ \end{small}\end{center}
		\end{figure}
	\end{minipage}
	\hspace{-0.01\linewidth}
	\begin{minipage}{0.20\linewidth}
		\begin{figure}[H]
			\includegraphics[trim = 1cm 0cm 2cm 1cm, clip, scale=0.12]{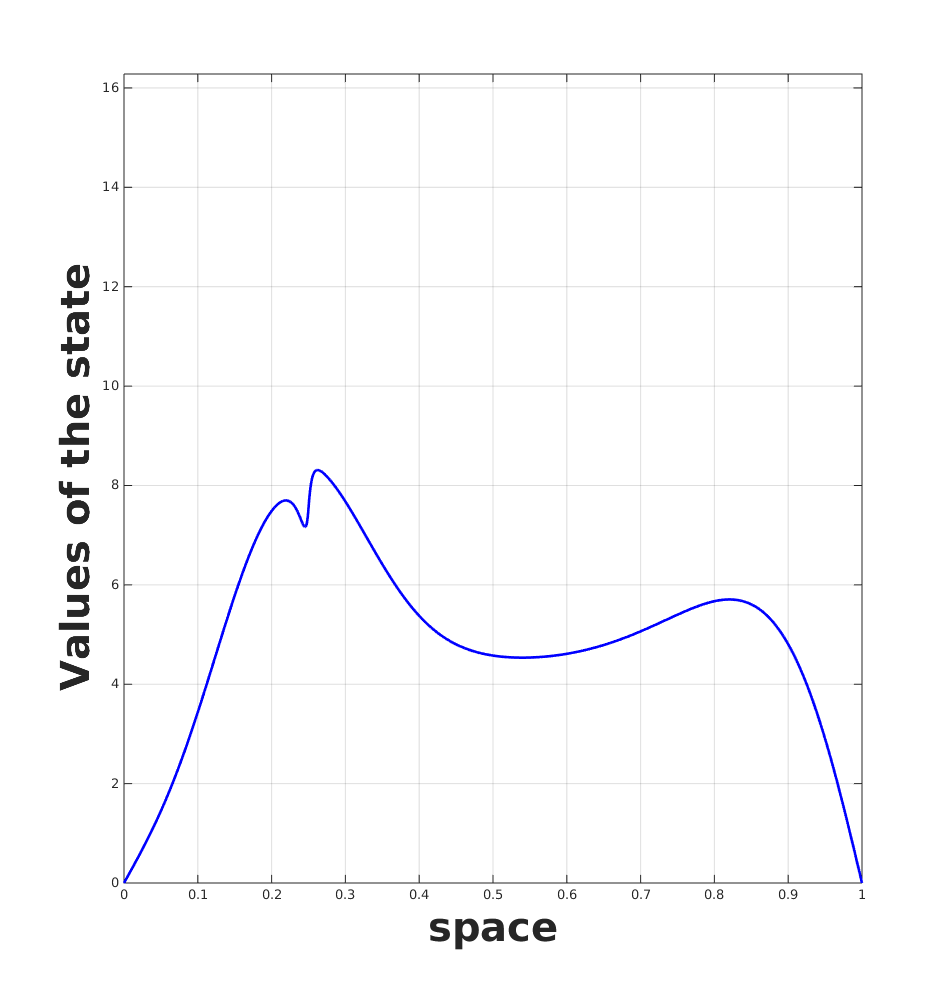} \\
			\includegraphics[trim = 1cm 0cm 2cm 1cm, clip, scale=0.12]{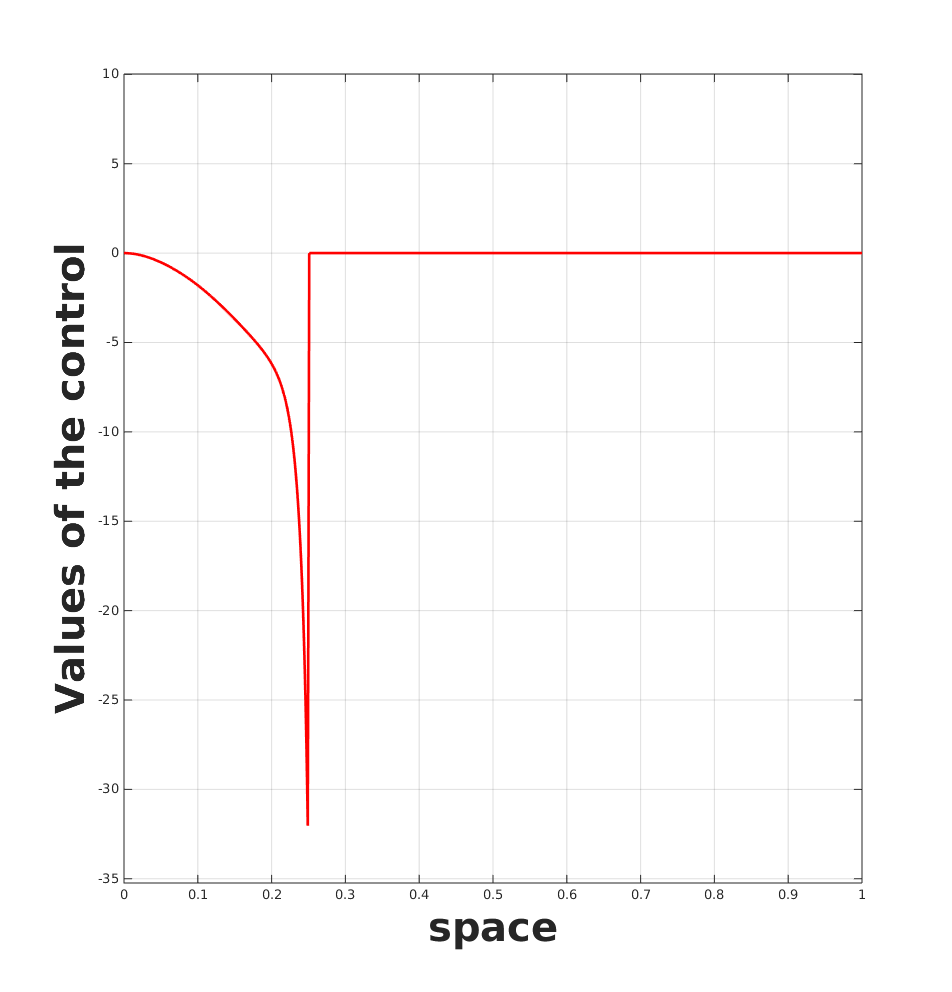} 
			\begin{center}\begin{small} $ t = 13.6248 $ \end{small}\end{center}
		\end{figure}
	\end{minipage}
	\hspace{-0.01\linewidth}
	\begin{minipage}{0.20\linewidth}
		\begin{figure}[H]
			\includegraphics[trim = 1cm 0cm 2cm 1cm, clip, scale=0.12]{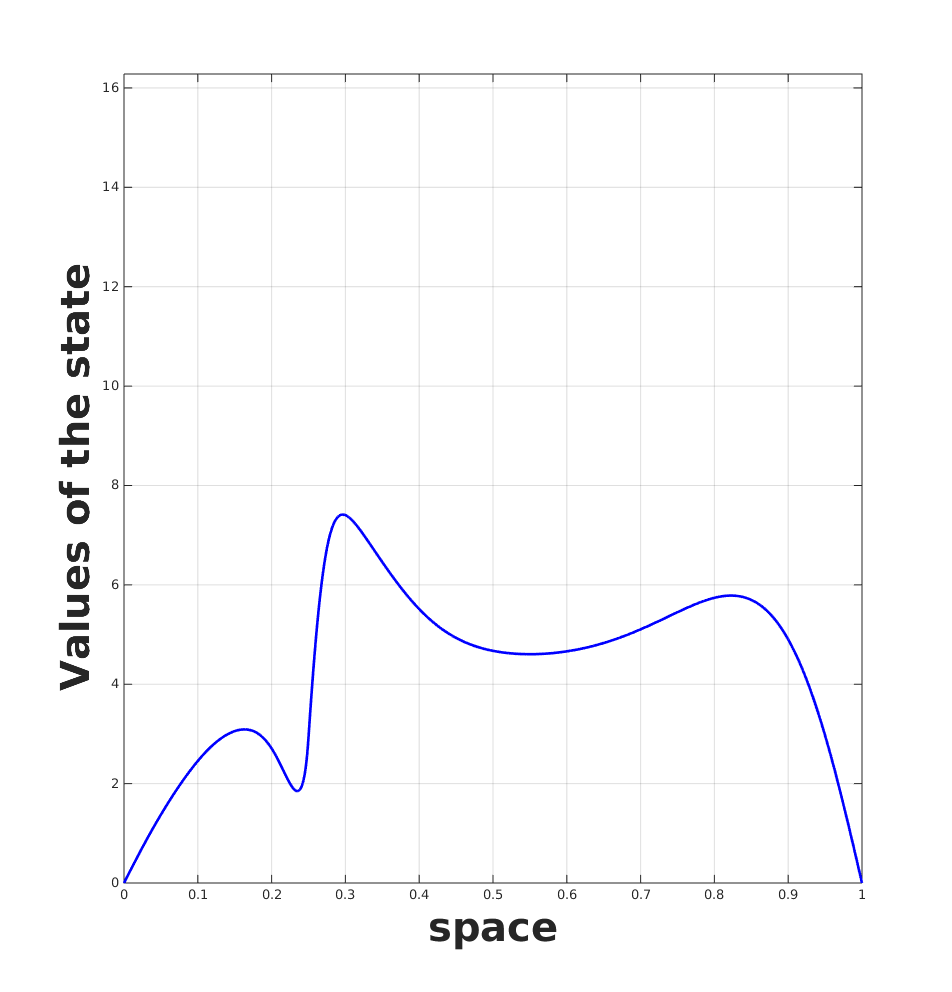} \\
			\includegraphics[trim = 1cm 0cm 2cm 1cm, clip, scale=0.12]{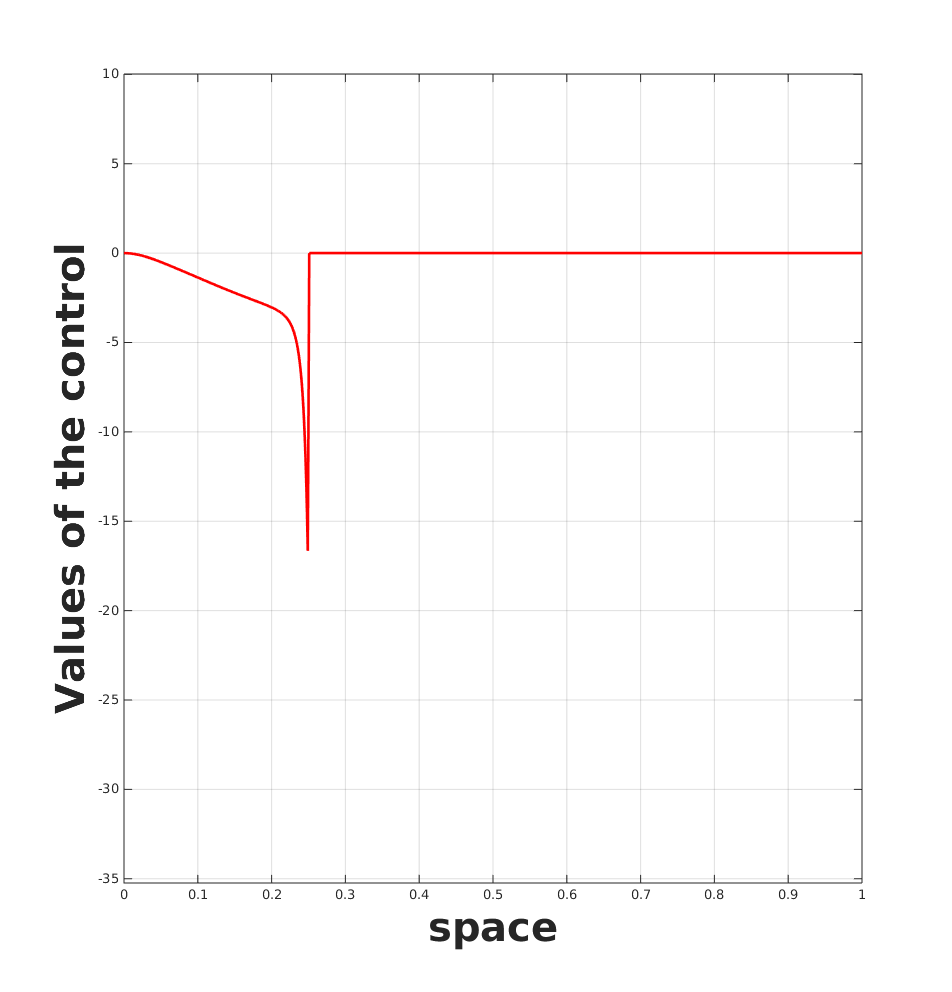} 
			\begin{center}\begin{small} $ t = 13.9201 $ \end{small}\end{center}
		\end{figure}
	\end{minipage}	
	\hspace{-0.01\linewidth}
	\begin{minipage}{0.20\linewidth}
		\begin{figure}[H]
			\includegraphics[trim = 1cm 0cm 2cm 1cm, clip, scale=0.12]{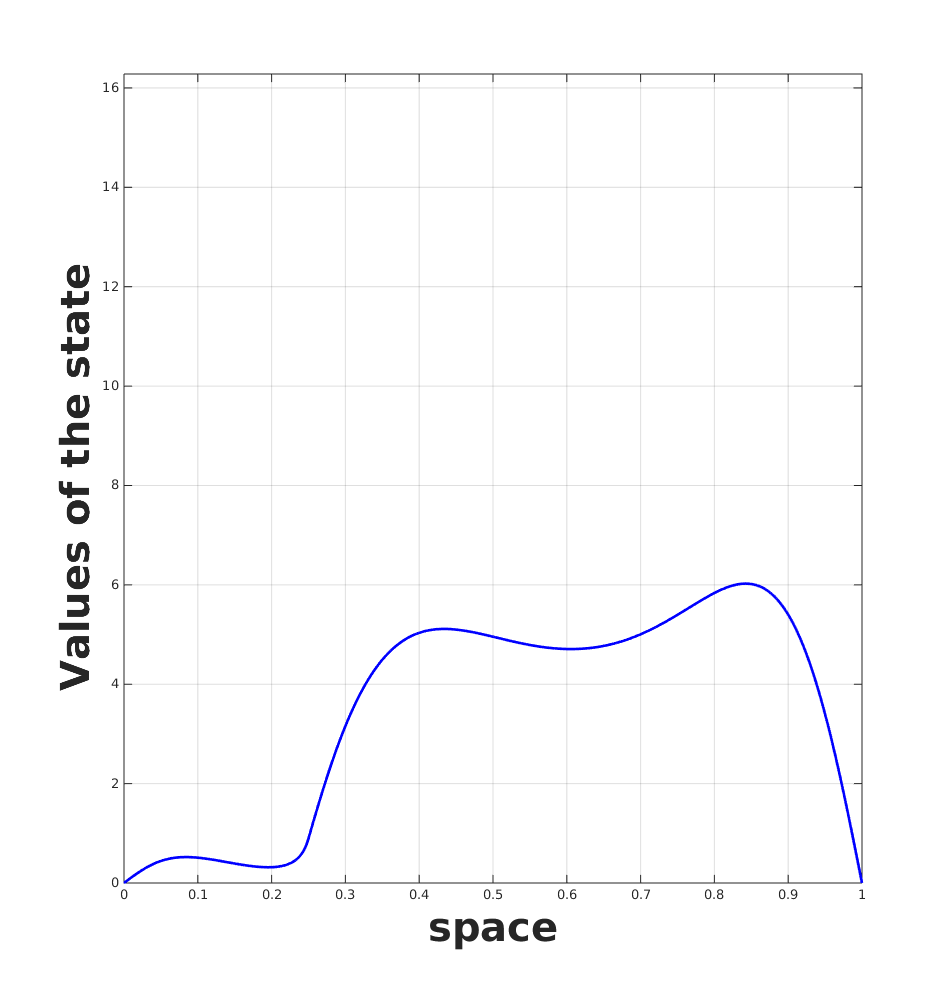} \\
			\includegraphics[trim = 1cm 0cm 2cm 1cm, clip, scale=0.12]{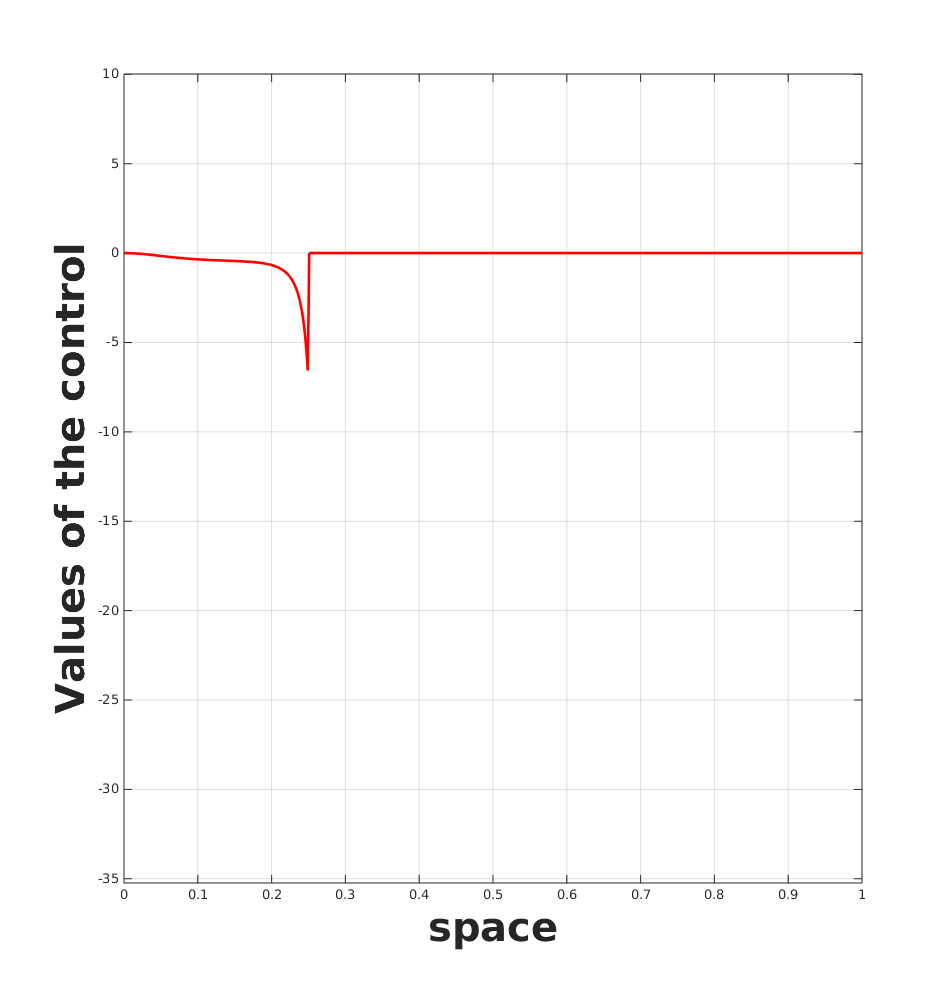} 
			\begin{center}\begin{small} $ t = 16.8736 $ \end{small}\end{center}
		\end{figure}
	\end{minipage}
	\\ \vspace{-10pt}
	\hspace{-0.04\linewidth}
	\begin{minipage}{0.20\linewidth}
		\begin{figure}[H]
			\includegraphics[trim = 1cm 0cm 2cm 1cm, clip, scale=0.12]{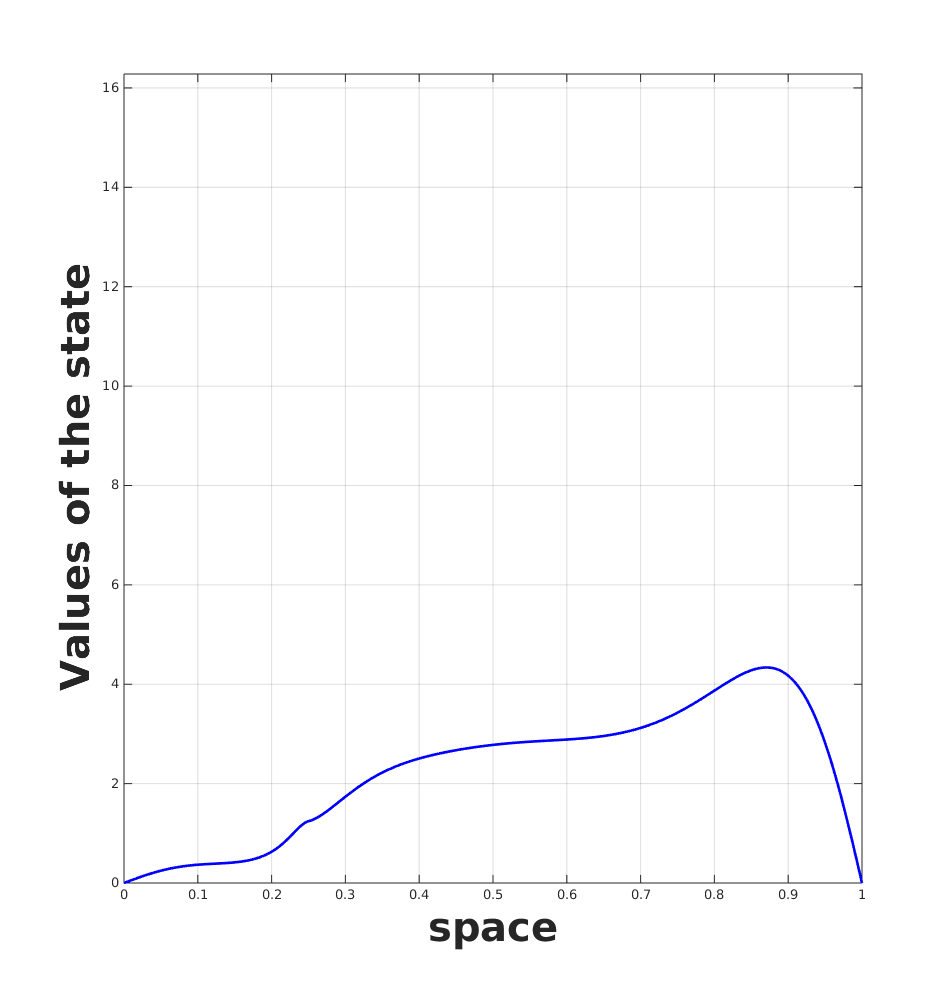} \\
			\includegraphics[trim = 1cm 0cm 2cm 1cm, clip, scale=0.12]{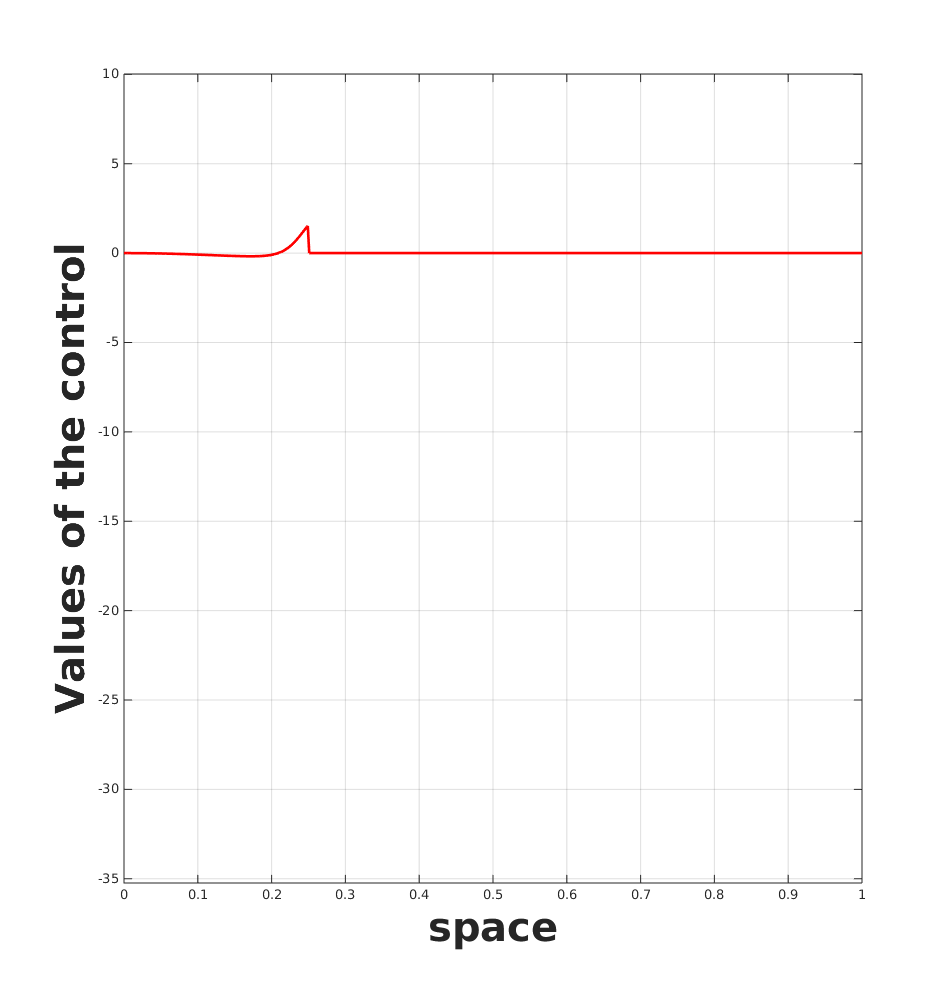}
			\begin{center}\begin{small} $ t = 21.7960 $ \end{small}\end{center}
		\end{figure}
	\end{minipage}
	\hspace{-0.01\linewidth}
	\begin{minipage}{0.20\linewidth}
		\begin{figure}[H]
			\includegraphics[trim = 1cm 0cm 2cm 1cm, clip, scale=0.12]{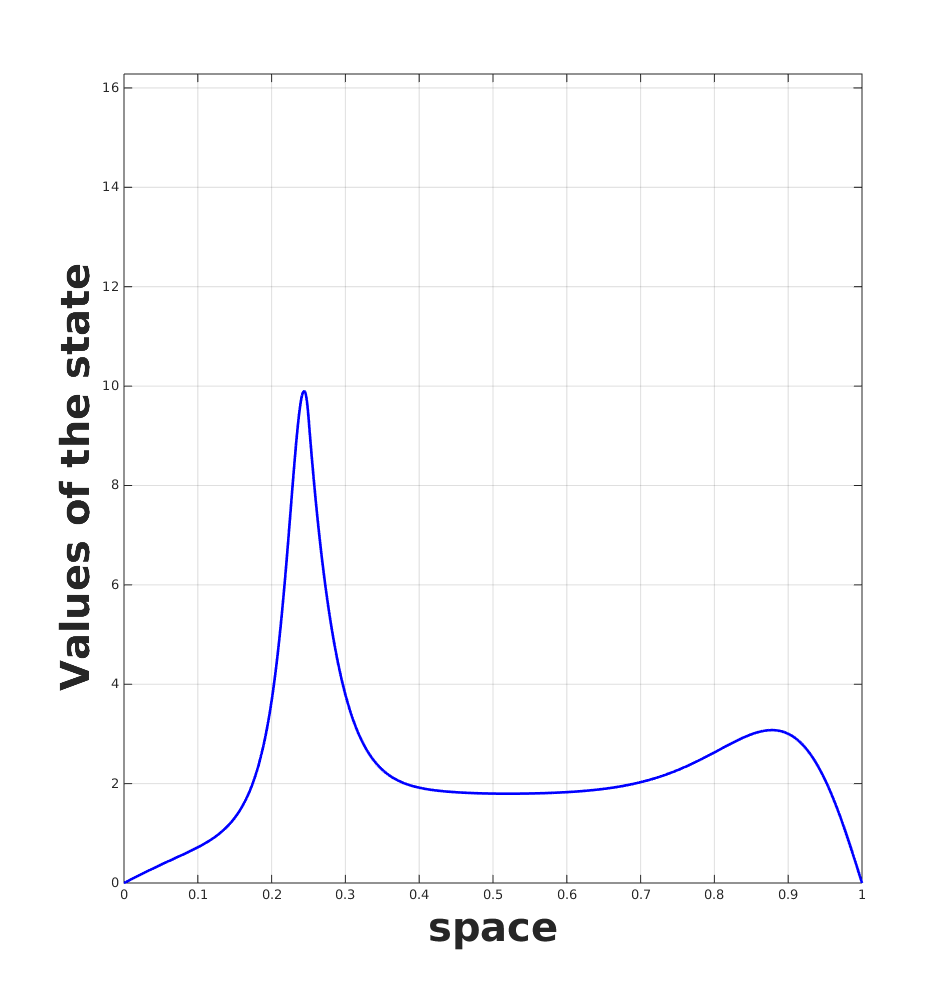} \\
			\includegraphics[trim = 1cm 0cm 2cm 1cm, clip, scale=0.12]{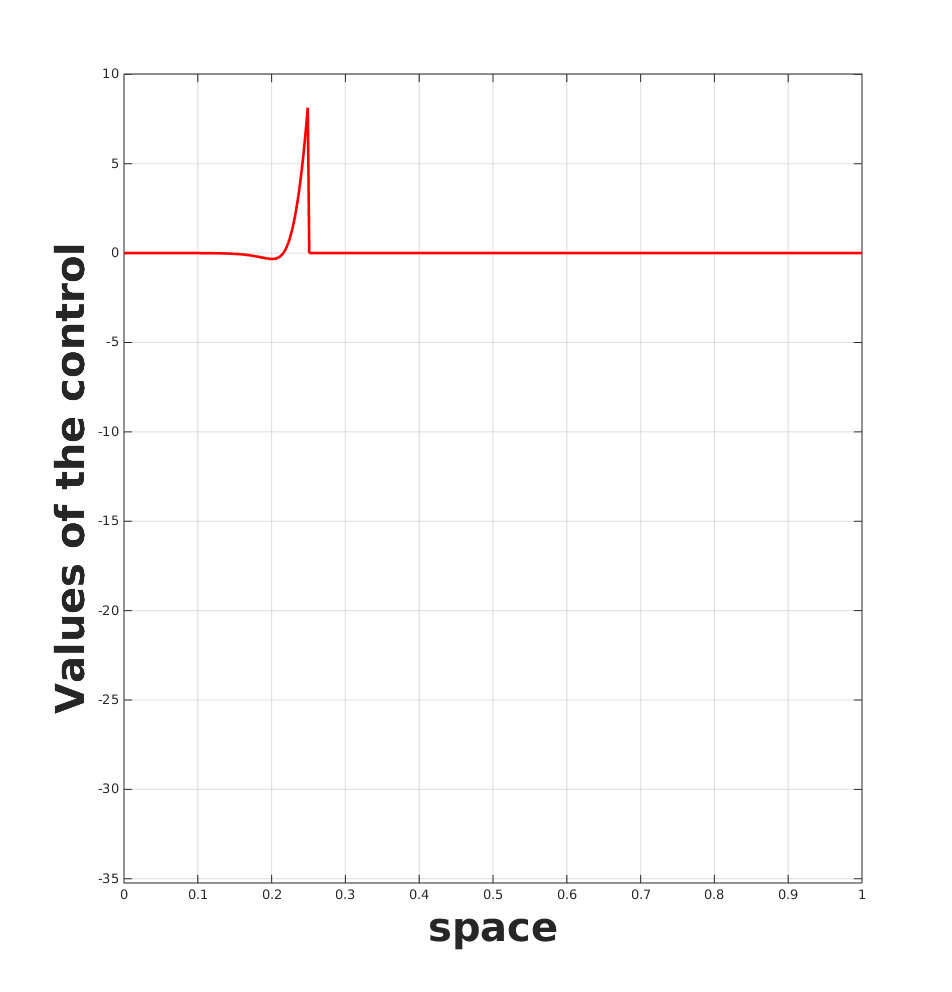} 
			\begin{center}\begin{small} $ t = 25.0776 $ \end{small}\end{center}
		\end{figure}
	\end{minipage}
	\hspace{-0.01\linewidth}
	\begin{minipage}{0.20\linewidth}
		\begin{figure}[H]
			\includegraphics[trim = 1cm 0cm 2cm 1cm, clip, scale=0.12]{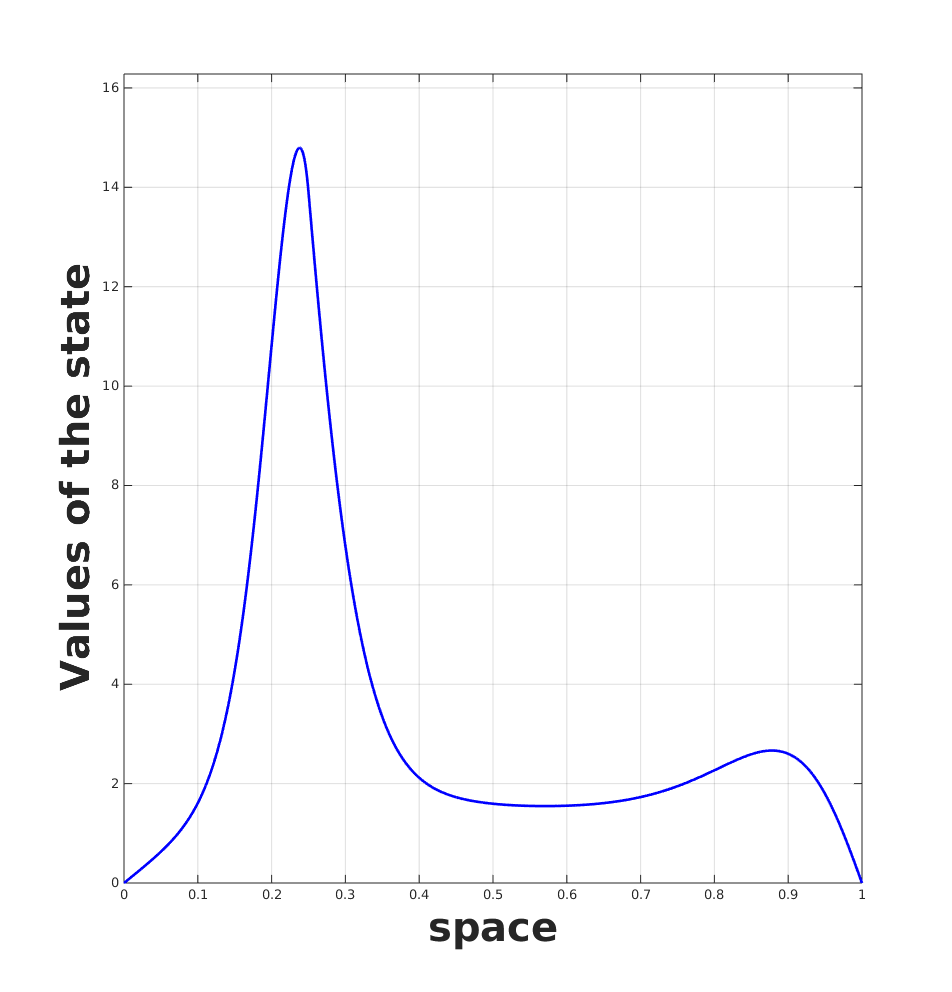} \\
			\includegraphics[trim = 1cm 0cm 2cm 1cm, clip, scale=0.12]{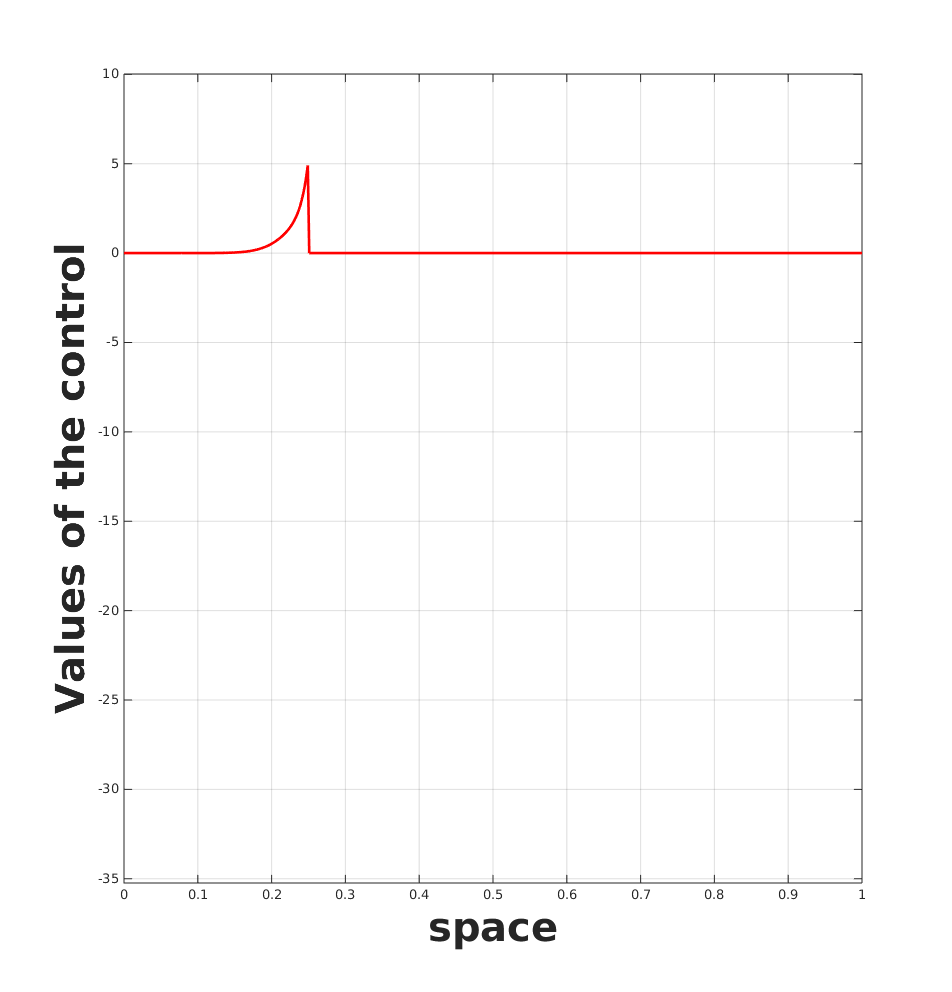} 
			\begin{center}\begin{small} $ t = 26.7184 $ \end{small}\end{center}
		\end{figure}
	\end{minipage}
	\hspace{-0.01\linewidth}
	\begin{minipage}{0.20\linewidth}
		\begin{figure}[H]
			\includegraphics[trim = 1cm 0cm 2cm 1cm, clip, scale=0.12]{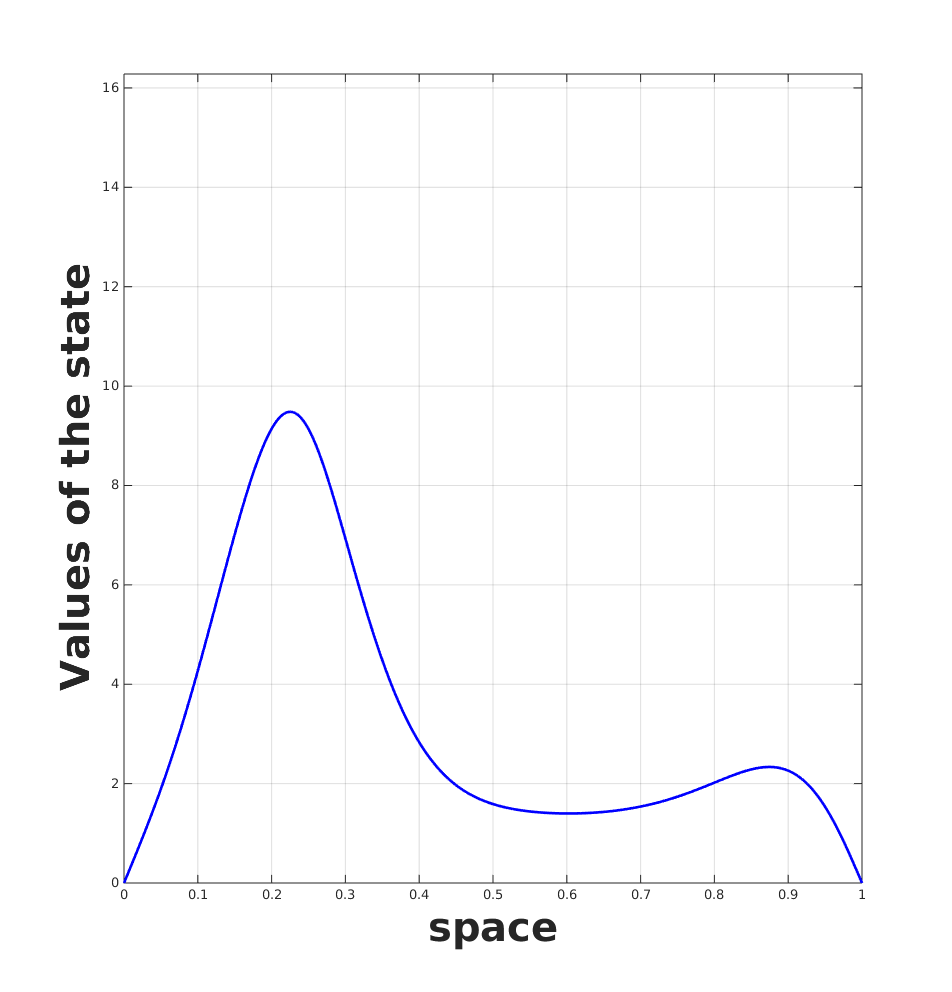} \\
			\includegraphics[trim = 1cm 0cm 2cm 1cm, clip, scale=0.12]{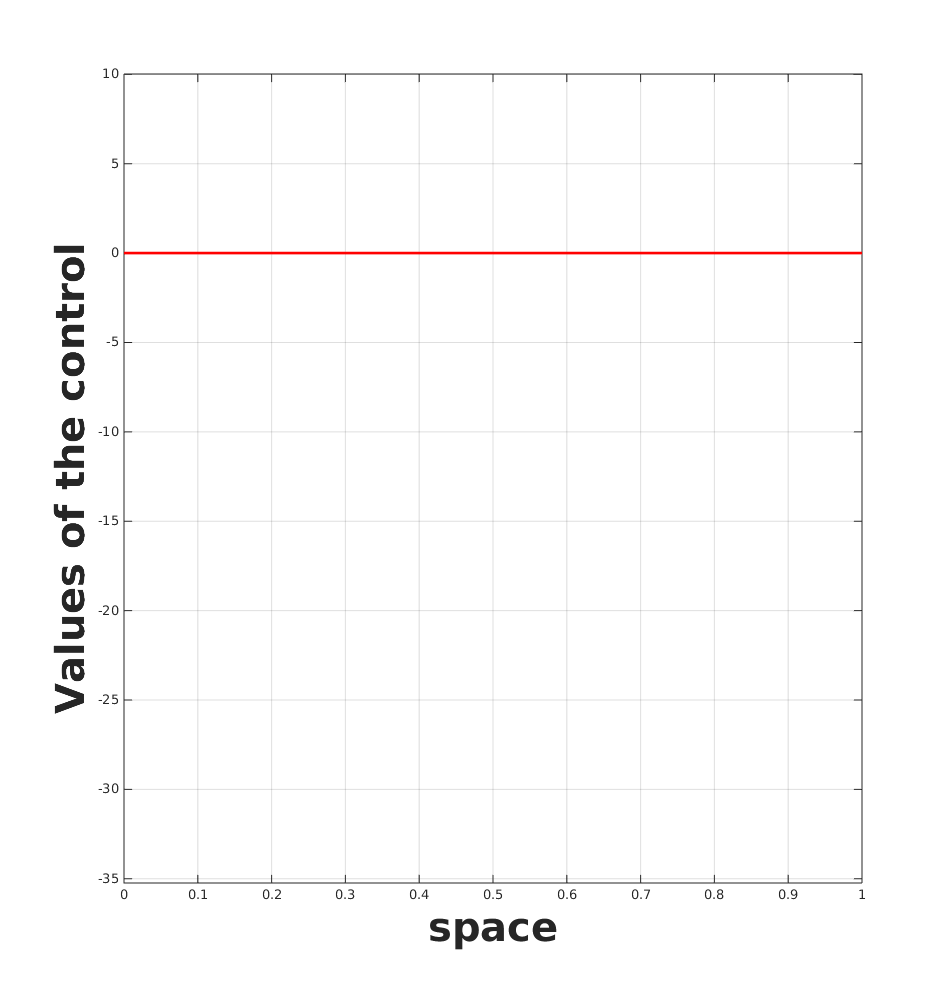} 
			\begin{center}\begin{small} $ t = 29.0155 $ \end{small}\end{center}
		\end{figure}
	\end{minipage}	
	\hspace{-0.01\linewidth}
	\begin{minipage}{0.20\linewidth}
		\begin{figure}[H]
			\includegraphics[trim = 1cm 0cm 2cm 1cm, clip, scale=0.12]{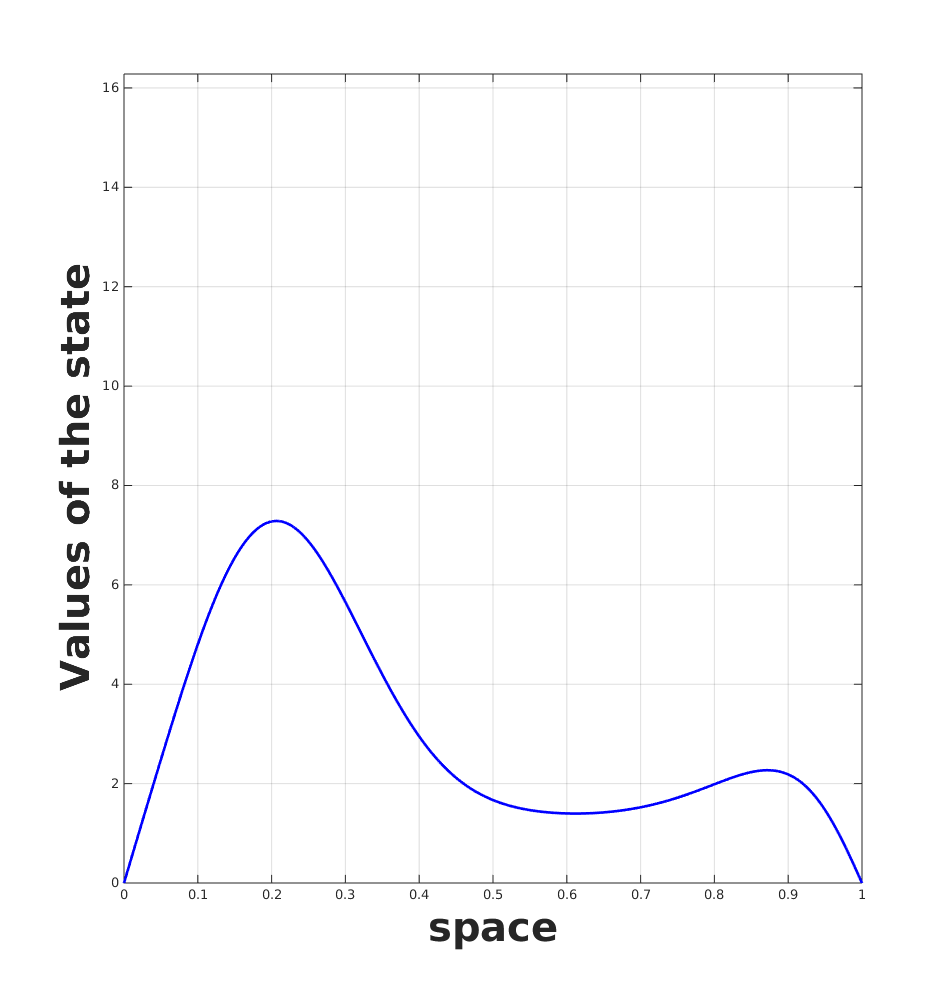} \\
			\includegraphics[trim = 1cm 0cm 2cm 1cm, clip, scale=0.12]{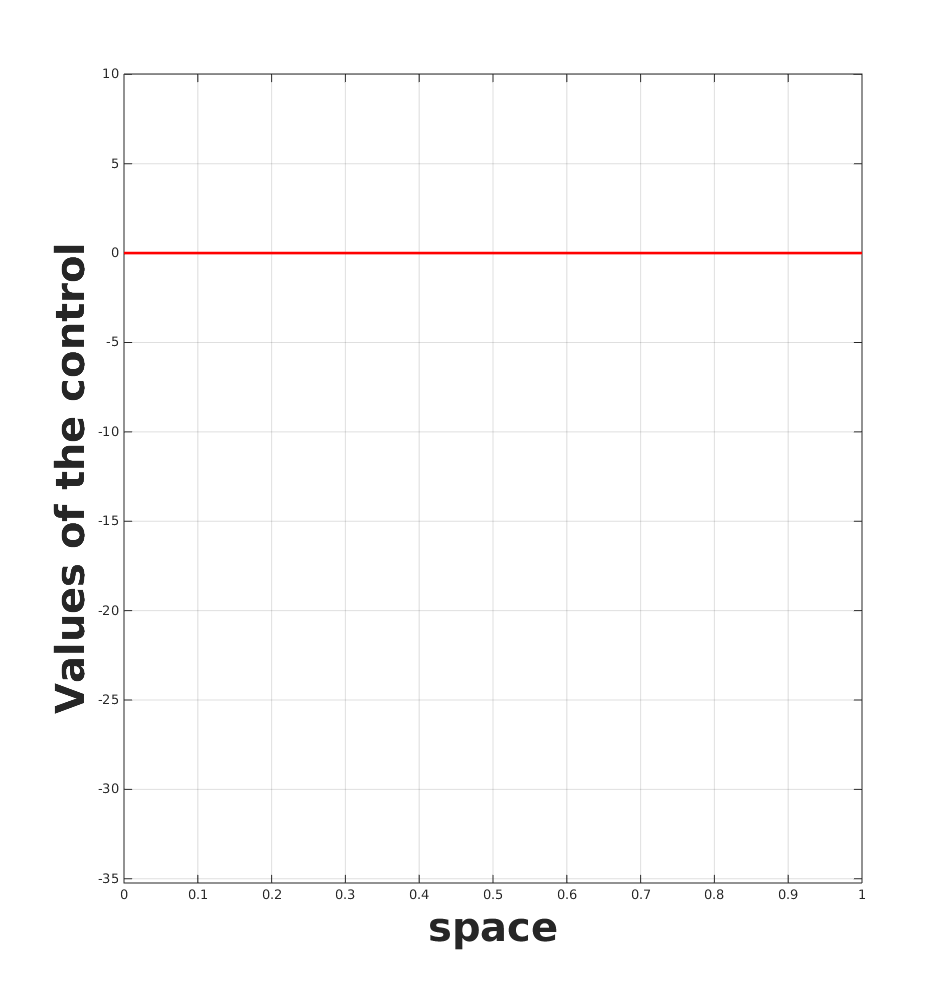} 
			\begin{center}\begin{small} $ t = 30.0000 $ \end{small}\end{center}
		\end{figure}
	\end{minipage}
\end{minipage}
\begin{minipage}{\linewidth}
	\begin{figure}[H]
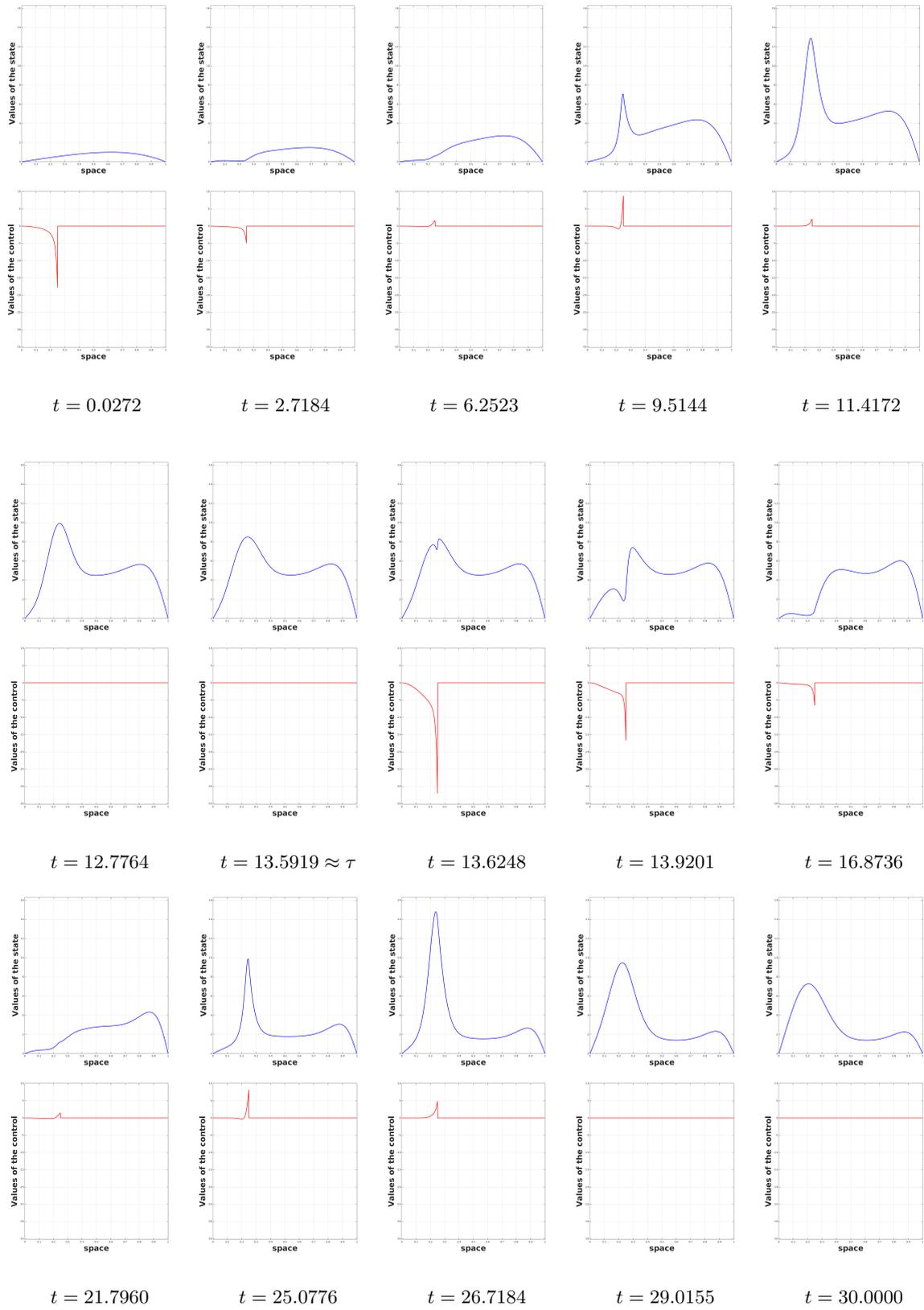

		\caption{Values of the state $y_1$ and the control $u_1$, for different values of the time. \textcolor{white}{}\label{superfigure}}
	\end{figure}
\end{minipage}\\
\hfill \\
\FloatBarrier

\appendix

\section{Regularity of Nemytskii operators} \label{appendix}

We recall in this section a technical result concerning the regularity of Nemytskii operators. The terminology comes from~\cite[Section 4.3.2]{Tro10}.

Let $E_x$ be a measurable subset of $\R^{n}$. Let $k \in \mathbb{N}$. Let $\phi:E_x \times \R \rightarrow \R$. Assume that $\phi$ is measurable and that for a.\,e.\@ $x \in E_x$, $y \mapsto \phi(x, y)$ is continuously $k$-times differentiable.

\begin{definition} \label{defNemytskii}
We say that $\phi$ satisfies the boundedness conditions of order $k$ (with respect to $x$) if there exists $C>0$ such that:
\begin{eqnarray*}
|D_{y^\ell}^\ell \phi(x,0)| \leq C, & & \text{for all $\ell=0,...,k$}, \quad \text{for a.\,e.\@ $x \in E_x$}.
\end{eqnarray*}
We say that $\phi$ satisfies the local Lipschitz condition of order $k$ (with respect to $y$) if for all $M \geq 0$, there exists $K \geq 0$ such that for all $x \in E_x$, for all $y_1$ and $y_2 \in \R$,
\begin{eqnarray} \label{eqLipschitzCond}
\Big(|y_1| \leq M \text{ and }\ |y_2| \leq M \Big) & \Rightarrow &
|D_y^k \phi(x,y_2)-D_y^k \phi(x,y_1)| \leq K |y_2-y_1|.
\end{eqnarray}
\end{definition}
Note that if the boundedness and Lipschitz conditions of order $k$ are satisfied, then for all $M \geq 0$, for $\ell= 0,...,k$, there exist a constant $C>0$ and a constant $K>0$ such that for all $x$ and for all $y_1$ and $y_2 \in \R$,
\begin{eqnarray} \label{eqNemytskiiGal}
\left\{\begin{array} {l}
|y_1| \leq M, \\ |y_2| \leq M ,
\end{array}\right.
 & \Rightarrow &
\left\{ \begin{array} {l}
|D_{y^\ell}^\ell \phi(x,y_1) | \leq C, \\
|D_{y^\ell}^\ell \phi(x,y_2)-D_{y^\ell}^\ell \phi(x,y_2)| \leq K |y_2-y_1|.
\end{array}\right.
\end{eqnarray}
In particular, if the boundedness and Lipschitz conditions of order $k$ are satisfied,
then they are also satisfied for any smaller order.
The following lemma is a direct extension of~\cite[Lemma 4.12]{Tro10}.

\begin{lemma} \label{lemmaNemytskii}
Assume that $\phi$ satisfies the boundedness condition of order $k$ and the Lipschitz condition of order $k$.
Then, the following mapping, called Nemytskii operator associated with $\phi$, is $k$ times continuously Fr\'echet differentiable:
\begin{eqnarray*}
\Phi: y \in L^\infty(E_x,\R) \mapsto \Big( x \in E_x \mapsto \phi(x, y(x)) \Big) \in L^\infty(E_x, \R).
\end{eqnarray*}
For all $\ell=1,...,k$, for all $z_1,..., z_\ell \in L^\infty(E_x,\R)$,
\begin{eqnarray} \label{eqDerivativeNem}
D_{y^\ell}^\ell \Phi(y)\big( z_1,...,z_\ell \big)
= \Big( x \mapsto D_{y^\ell}^\ell \phi(x, y(x))\big( z_1(x),...,z_\ell(x) \big) \Big) \in L^\infty(E_x,\R).
\end{eqnarray}
\end{lemma}

\begin{proof}
To simplify the notation, we write $\| \cdot \|_{\infty}$ instead of $\| \cdot \|_{L^\infty(E_x,\R)}$.
We prove the lemma by induction. The case $k=1$ is treated in~\cite[Lemma 4.12]{Tro10}. It is moreover proved that for all $M$, for all $y_1$ and $y_2 \in L^\infty(E_x,\R)$ such that $\| y_1 \|_{\infty} \leq M$ and $\| y_2 \|_{\infty} \leq M$,
\begin{eqnarray} \label{eqNemytskiiEstimate}
\| \Phi(y_2)-\Phi(y_1)-D\Phi(y_1)(y_2-y_1) \|_{\infty}
& \leq & K \| y_2 - y_1 \|_{\infty}^2,
\end{eqnarray}
where $K$ is the constant given by~\eqref{eqLipschitzCond}, for $k= 1$.

Let $q \in \mathbb{N}$ and assume that the lemma is satisfied for $k= q$. Assume that $\phi$ satisfies the boundedness and the Lipschitz conditions at order $q+1$. For all $y$, denote by $A: L^\infty(E_x,\R) \mapsto \mathcal{L}( L^\infty(E_x,\R)^{q+1}, L^\infty(E_x,\R) )$ the mapping defined by the right-hand side of~\eqref{eqDerivativeNem}:
\begin{eqnarray*}
A(y)(z_1,...,z_{q+1}) & = &
D_{y^{q+1}}^{q+1} \phi(\cdot, y(\cdot))\big( z_1(\cdot),...,z_{q+1}(\cdot) \big).
\end{eqnarray*}
By~\eqref{eqNemytskiiGal}, for all $M \geq 0$, there exist constants $C >0$ and $K>0$ such that for all $y_1$ and $y_2 \in L^\infty(E_x,\R)$ satisfying $\| y_1 \|_{\infty} \leq M$ and $\| y_2 \|_{\infty} \leq M$, for all $z_1$,...,$z_{q+1}$ in $L^\infty(E_x,\R)$,
\begin{eqnarray*}
 \| A(y_1)(z_1,...,z_{q+1}) \|_{\infty} & \leq & 
 C \| z_1 \|_{\infty} ... \| z_{q+1} \|_{\infty}, \\
 \| A(y_2)(z_1,...,z_{q+1})-A(y_1)(z_1,...,z_{q+1}) \|_{\infty}
& \leq & K \|y_2-y_1\|_{\infty} \| z_1 \|_{\infty} ... \| z_{q+1} \|_{\infty}.
\end{eqnarray*}
This proves that for all $y \in L^\infty(E_x,\R)$, the mapping $A(y)$ is a continuous $(q+1)$-linear mapping and that $A$ is locally Lipschitz-continuous, thus continuous. We prove now that for all $M \geq 0$, for all $y_1,y_2 \in L^\infty(E_x,\R)$ such that $\| y_1 \|_{\infty} \leq M$ and $\| y_2 \|_{\infty} \leq M$, for all $z:=(z_1,...,z_{q})$ in $[L^\infty(E_x,\R)]^q$,
\begin{eqnarray}
 \| D_{y^q}^q \Phi (y_2)(z)-D_{y^q}^q(y_1) \Phi (z) -A(y_1)(z,y_2-y_1) \|_{\infty} 
& \leq & K \| y_2 - y_1 \|_{\infty}^2 \prod_{i=1}^q\| z_i \|_{\infty} , \label{eqNewToBeProved}
\end{eqnarray}
which is enough to prove that $\Phi$ is continuously $(q+1)$-times Fr\'echet differentiable.
Let us fix $z = (z_1,...,z_q)$. Observe that $\Psi: y \mapsto D_{y^q}^q \Phi(y) (z)$ is the Nemytskii operator associated with:
\begin{eqnarray} \label{eqNemDerivative}
\psi \colon (x,y) \in E_x \times \R \mapsto D_{y^q}^q \phi(x, y(x))(z(x)).
\end{eqnarray}
It is easy to check that the above function satisfies the boundedness and the Lipschitz conditions of order 1. In particular, for all $M \geq 0$, for all $y_1$ and $y_2 \in \R$ with $|y_1| \leq M$ and $|y_2| \leq M$, for all $x \in E_x$,
\begin{eqnarray} \label{eqNemForPsi}
|D_y \psi(x,y_2)- D_y \psi(x,y_1)| & \leq & K |y_2-y_1| \prod_{i=1}^q\| z_i \|_{\infty} ,
\end{eqnarray}
where $K$ is given by~\eqref{eqNemytskiiGal}, for $\ell= q+1$. This proves that $\Psi$ is continuously differentiable and by~\eqref{eqDerivativeNem}, $D_y \Psi(y)\delta y= A(y)\delta y$. By~\eqref{eqNemytskiiEstimate} and~\eqref{eqNemForPsi},
\begin{eqnarray*}
\| \Psi(y_2)-\Psi(y_1)-A(y)(z,y_2-y_1) \|_{\infty} & \leq & 
K \| y_2 - y_1 \|_{\infty}^2 \prod_{i=1}^q\| z_i \|_{\infty},
\end{eqnarray*}
which proves~\eqref{eqNewToBeProved} and therefore concludes the proof.
\end{proof}

As mentioned in the proof, the derivative of order $\ell$ of the Nemytskii operator is the Nemytskii operator of the pointwise derivative of order $\ell$. For $\ell=1$ in particular, this means that for all $y \in L^\infty(E_x,\R)$, $D\Phi(y)$ can be seen as an element of $\mathcal{L} \big( L^\infty(E_x,\R),L^\infty(E_x,\R) \big)$ (as in~\eqref{eqNemDerivative}) or as an element of $L^\infty(E_x,\R)$:
\begin{eqnarray*}
x \mapsto D_y \phi(x,y(x)).
\end{eqnarray*}
In the article, one of the two points of view is adopted depending on the context. For example, in the definition of the adjoint equation~\eqref{eqAdjoint}, $D_y\Phi_1(\tau,\bar{y}(1))$ and $D_y \Phi_2(\bar{y}(2))$ are seen as element of $L^\infty(\Omega)$.

\begin{lemma} \label{lemmaNemytskii2}
Assume that $\phi$ satisfies the boundedness and the Lipschitz conditions of order $k$. Then, for all $\ell =1,...,k$, for all $y \in L^\infty(E_x,\R)$, the mapping $D_{y^\ell}^\ell \Phi(y)$ can be extended to a $\ell$-linear continuous mapping from $[L^\ell(E_x,\R)]^\ell$ to $L^1(E_x,\R)$. Moreover, for all $M \geq 0$, there exists $R \geq 0$ such that for all $y_1$ and $y_2 \in L^\infty(E_x,\R)$, for all $z = (z_1,...,z_\ell) \in [L^{\ell}(E_x,\R)]^{\ell}$,
\begin{small}
\begin{eqnarray*}
\left\{
\begin{array} {l} 
\| y_1 \|_{L^\infty(E_x,\R)} \leq M \\
\| y_2 \|_{L^\infty(E_x,\R)} \leq M 
\end{array}
\right.
& \Rightarrow &
 \| \big( D_{y^\ell}^{\ell}(y_2)- D_{y^\ell}^\ell(y_1) \big)(z) \|_{L^1(E_x,\R)}
\leq R \| y_2-y_1 \|_{L^\infty(E_x,\R)} 
\prod_{i=1}^{\ell}\| z_i \|_{L^\ell(E_x,\R)} .
\end{eqnarray*}
\end{small}
\end{lemma}

\begin{proof}
The two statements of the lemma are a direct consequence of~\eqref{eqDerivativeNem} and H\"older's inequality.
\end{proof}

\section*{Acknowledgments}
The authors gratefully acknowledge support by the Austrian Science Fund (FWF) special
research grant SFB-F32 "Mathematical Optimization and Applications in Biomedical
Sciences", and by the ERC advanced grant 668998 (OCLOC) under the EU's H2020
research program.

\bibliographystyle{plain}
\bibliography{biblio}

\end{document}